\documentclass[11pt,oneside,british]{amsart}
\usepackage[T1]{fontenc}
\usepackage[latin9]{inputenc}
\usepackage{prettyref}
\usepackage{bm}
\usepackage{amsthm}
\usepackage{amssymb}
\usepackage{esint}
\usepackage{constants}

\makeatletter
\numberwithin{equation}{section}
\numberwithin{figure}{section}
\numberwithin{table}{section}
\theoremstyle{plain}
\newtheorem{thm}{\protect\theoremname}[section]
  \theoremstyle{plain}
  \newtheorem{lem}[thm]{\protect\lemmaname}
  \theoremstyle{plain}
  \newtheorem{cor}[thm]{\protect\corollaryname}
  \theoremstyle{definition}
  \newtheorem{rem}[thm]{Remark}
  \theoremstyle{plain}
  \newtheorem{prop}[thm]{\protect\propositionname}
  \newtheorem*{acknowledgement*}{\protect\acknowledgementname}

\def\n{\noindent}
\def\dsl{\textstyle\sum\limits}
\def\f{\footnotesize}
\def\dis{\displaystyle}

\usepackage{amsfonts}
\usepackage[numbers,sort&compress,square]{natbib}
\usepackage{url}
\numberwithin{equation}{section} 
\newrefformat{pro}{Proposition \ref{#1}}
\newrefformat{cor}{Corollary \ref{#1}}
\newrefformat{sub}{Section \ref{#1}}
\newrefformat{rem}{Remark \ref{#1}}
\newrefformat{def}{Definition \ref{#1}}

\makeatother

\usepackage{babel}
  \providecommand{\acknowledgementname}{Acknowledgement}
  \providecommand{\corollaryname}{Corollary}
  \providecommand{\lemmaname}{Lemma}
  \providecommand{\propositionname}{Proposition}
\providecommand{\theoremname}{Theorem}

\newconstantfamily{c}{symbol=c}

\begin{document}

\title{Gumbel fluctuations for cover times in the discrete torus}

\author{David Belius}

\thanks{This research has been supported by the grant ERC-2009-AdG 245728-RWPERCRI}

\maketitle

\begin{abstract}
This work proves that the fluctuations of the cover time of simple
random walk in the discrete torus of dimension at least three with
large side-length are governed by the Gumbel extreme value distribution.
This result was conjectured for example in \cite{aldous-fill:book}.
We also derive some corollaries which qualitatively describe ``how''
covering happens. In addition, we develop a new and stronger
coupling of the model of random interlacements, introduced by Sznitman
in \cite{Sznitman2007}, and random walk in the torus. This coupling is used
to prove the cover time result and is also of independent interest.
\end{abstract}

\section*{0 Introduction}

In this article we prove that if $C_{N}$ is the cover time of the
discrete torus of side-length $N$ and dimension $d\ge3$ then $C_{N}/(g(0)N^{d})-\log N^{d}$
(where $g(\cdot)$ is the $\mathbb{Z}^{d}$ Green function) tends
in law to the Gumbel distribution as $N\rightarrow\infty$, or in
other words, that the fluctuations of the cover time are governed
by the Gumbel distribution. This result was conjectured e.g. in
Chapter 7, p. 23, \cite{aldous-fill:book}. We also construct a new
stronger coupling of random walk in the torus and the model of random
interlacements, thus improving a result from \cite{TeixeiraWindischOnTheFrag}.
The coupling is of independent interest and is also used as a tool
to prove the cover time result.

Cover times of finite graphs by simple random walk have been studied
extensively, see e.g. \cite{Aldous-OnTimeTaken...,Aldous-ThresholdLimitsforCT,aldous-fill:book,Brummelhuis1991,DemboPeresEtAl-CoverTimesforBMandRWin2D,MatthewsCoveringProbsForMCs}.
One important case is the cover time of the discrete torus $\mathbb{T}_{N}=(\mathbb{Z}/N\mathbb{Z})^{d}$.
Let $P$ be the canonical law of continuous time simple random walk
in this graph, starting from the uniform distribution, and let the
canonical process be denoted by $(Y_{t})_{t\ge0}$. The cover time
of the discrete torus is the first time $Y_{\cdot}$ has visited every
vertex of the graph, or in other words
\begin{equation}
C_{N}=\max_{x\in\mathbb{T}_{N}}H_{x},\label{eq:IntroDefOfCoverTime}
\end{equation}
where $H_{x}$ denotes the entrance time of the vertex $x\in\mathbb{T}_{N}$.
For $d\ge3$ it is known that $EC_{N}\sim g(0)N^{d}\log N^{d}$, as
$N\rightarrow\infty$, and that $C_N$ concentrates in the sense
that $\frac{C_{N}}{g(0)N^{d}\log N^{d}}\rightarrow1$ in probability,
as $N\rightarrow\infty$. However previously it was only conjectured
that the finer scaling result 
\begin{equation}
\mbox{\f $\dis\frac{C_{N}}{g(0)N^{d}}$}-\log N^{d}\overset{\rm{law}}{\longrightarrow}G\mbox{ as }N\rightarrow\infty\mbox{, \mbox{for }}d\ge3,\label{eq:IntroCoverTimeOfTorus}
\end{equation}

\medskip\n
holds, where $G$ denotes the standard Gumbel distribution, with cumulative
distribution function $F(z)=e^{-e^{-z}}$ (see e.g. Chapter 7, p.
22-23, \cite{aldous-fill:book}). In this article we prove \prettyref{eq:IntroCoverTimeOfTorus}.
In fact our result, \prettyref{thm:G_Gumbel}, proves more, namely
that the (appropriately defined) cover time of any ``large'' subset
of $\mathbb{T}_{N}$ satisfies a similar relation. (For $d=1,2$,
the asymptotic behaviour of $E[C_{N}]$ is different; see \cite{DemboPeresEtAl-CoverTimesforBMandRWin2D}
for $d=2$, the case $d=1$ is an exercise in classical probability
theory. The concentration result $C_{N}/E[C_{N}]\rightarrow1$
still holds for $d=2$, but the nature of the fluctuations is unknown.
For $d=1$ one can show that $C_{N}/N^{2}$ converges in
law to the time needed by Brownian motion to cover $\mathbb{R}/\mathbb{Z}$.)

Our second main result is a coupling of random walk in the discrete
torus and random interlacements, which we now introduce. To do so
we very briefly describe the model of random interlacements (see \prettyref{sec:Notation}
for more details). It was introduced in \cite{Sznitman2007} and helps
to understand ``the local picture'' left by the trace of a simple
random walk in (among other graphs) the discrete torus when $d\ge3$.
The random interlacement roughly speaking arises from a Poisson cloud
of doubly infinite nearest-neighbour paths modulo time-shift in $\mathbb{Z}^{d},d\ge3$,
with a parameter $u\ge0$ that multiplies the intensity. The trace
(vertices visited) of all the paths at some intensity $u$ is a random
subset of $\mathbb{Z}^{d}$ denoted by $\mathcal{I}^{u}$. Together
the $\mathcal{I}^{u}$, for $u\ge0$, form an increasing family $(\mathcal{I}^{u})_{u\ge0}$
of random subsets of $\mathbb{Z}^{d}$. We call the family $(\mathcal{I}^{u})_{u\ge0}$
a random interlacement and for fixed $u$ we call the random set $\mathcal{I}^{u}$
a random interlacement at level $u$. Random interlacements are intimitaly
related to random walk in the torus; intuitively speaking, the trace
of $Y_{\cdot}$ in a ``local box'' of the torus, when run up to
time $uN^{d}$, in some sense ``looks like'' $\mathcal{I}^{u}$
intersected with a box (see \cite{Windisch2008}, \cite{Sznitman2007}).

Our coupling result is one way to make the previous sentence precise
and can be formulated roughly as follows. Let $Y(0,t)$ denote the
trace of $Y_{\cdot}$ up to time $t$ (i.e. the set of vertices visited
up to time $t$). For $n\ge1$ pick vertices $x_{1},...,x_{n}\in\mathbb{T}_{N}$
and consider boxes $A_{1},...,A_{n}$, defined by $A_{i}=x_{i}+A,i=1,...,n,$
where $A\subset\mathbb{Z}^{d}$ is a box centred at $0$, of side-length
such that the $A_{1},...,A_{n}$ are ``well separated'' and at most
``mesoscopic''. Then for a ``level'' $u>0$ and a $\delta>0$
(which may not be too small)
\begin{equation}\label{eq:IntroCoupling}
\begin{array}{l}
\mbox{one can construct a coupling of random walk }Y_{\cdot}\mbox{ with law }P\mbox{ and independent }\\
\mbox{random interlacements }(\mathcal{I}_{1}^{v})_{v\ge0},...,(\mathcal{I}_{n}^{v})_{v\ge0},\mbox{ such that ``with high probability''}\\[1ex]
\qquad \mathcal{I}_{i}^{u(1-\delta)}\cap A\subset(Y(0,uN^{d})-x_{i})\cap A\subset\mathcal{I}_{i}^{u(1+\delta)}\cap A\mbox{ for }i=1,...,n.
\end{array}
\end{equation}

\medskip\n
The result is stated rigorously in \prettyref{thm:Coupling}. The
case $n=1$ (i.e. coupling for one box) and $u,\delta$ fixed (as
$N\rightarrow\infty$) was obtained in \cite{TeixeiraWindischOnTheFrag}
(and earlier a coupling of random walk in the so called discrete cylinder
and random interlacements was constructed in \cite{Sznitman2009-UBonDTofDCandRI,Sznitman2009-OnDOMofRWonDCbyRI}).
In this paper we strengthen the result from \cite{TeixeiraWindischOnTheFrag}
by coupling random interlacements with random walk in many separated
boxes (the most important improvement), and by allowing $\delta$
to go to zero and $u$ to go to zero or infinity. A similar improvement
of the discrete cylinder coupling from \cite{Sznitman2009-UBonDTofDCandRI,Sznitman2009-OnDOMofRWonDCbyRI}
can be found in \cite{BeliusCTinDC}.

A coupling of random interlacements and random walk is a very powerful tool to study the trace of
random walk. In this article we use the above coupling to study certain properties of the trace
relevant to the cover time result \prettyref{eq:IntroCoverTimeOfTorus}. Similar couplings have also
been used to study the percolative properties of the complement of the trace of random walk in terms of random interlacements;
in this case the relevant properties of the trace studied with the coupling are very different (see \cite{TeixeiraWindischOnTheFrag} for the torus,
and \cite{Sznitman2009-OnDOMofRWonDCbyRI,Sznitman2009-UBonDTofDCandRI} for the discrete cylinder). We
expect our coupling to find uses beyond the current cover time application. For more on this see
\prettyref{rem:EndOfSec3Remarks} (1).

We also prove two interesting corollaries of \prettyref{eq:IntroCoverTimeOfTorus}
(actually using the stronger subset version mentioned above) and \prettyref{eq:IntroCoupling}.
To formulate the first corollary we define the ``point process of
vertices covered last'', a random measure on the torus $(\mathbb{R}/\mathbb{Z})^{d}$, by
\begin{equation}
\mathcal{N}_{N}^{z}=\dsl_{x\in\mathbb{T}_{N}}\delta_{x/N}1_{\{H_{x}>g(0)N^{d}\{\log N^{d}+z\}\}},N\ge1,z\in\mathbb{R}.\label{eq:DefOfPPoVCV}
\end{equation}

\medskip\n
Note that $\mathcal{N}_{N}^{z}$ counts the vertices of $\mathbb{T}_{N}$
which have not yet been hit at time $g(0)N^{d}\{\log N^{d}+z\}$ (this
is the ``correct time-scale''; from \prettyref{eq:IntroCoverTimeOfTorus}
one sees that the probability that $\mathcal{N}_{N}^{z}$ is the zero
measure is bounded away from zero and one). The result is then that
(for $d\ge3$)
\begin{equation}
\begin{array}{l}
\mathcal{N}_{N}^{z}\overset{\rm law}{\longrightarrow}\mathcal{N}^{z}\mbox{ as }N\rightarrow\infty,\mbox{ where }\mathcal{N}^z\mbox{ is a Poisson point process}\\
\mbox{on }(\mathbb{R}/\mathbb{Z})^{d}\mbox{ of intensity }e^{-z}\lambda,\mbox{ and }\lambda\mbox{ denotes Lebesgue measure.}
\end{array}\label{eq:IntroPPoVCV}
\end{equation}

\medskip\n
This is proven in \prettyref{cor:PPoVCL}. Intuitively speaking it
means that the last vertices of the torus to be covered are approximately
independent and uniform.

As a consequence of \prettyref{eq:IntroPPoVCV} we obtain \prettyref{cor:LastTwoIndians}
which says, intuitively speaking, that
\begin{equation}\label{eq:IntroLastTwoIndians}
  \mbox{the last few vertices of }\mathbb{T}_{N}\mbox{ to be covered are far apart, \mbox{at distance of order }}N.
\end{equation}

\medskip\n
Note that a priori it is not clear if the correct qualitative picture
is that the random walk completes covering of the torus by ``taking
out several neighbouring vertices at once'' or if it ``takes out
the last few vertices one by one''. Roughly speaking \prettyref{eq:IntroLastTwoIndians}
implies that the latter is the case.

We now discuss the proofs of the above results. The result \prettyref{eq:IntroPPoVCV}
is proven using Kallenberg's theorem, which allows one to verify
convergence in law of certain point processes by checking only convergence
of the intensity measure and convergence of the probability of ``charging
a set''. The latter two quantities will be shown to converge using
\prettyref{eq:IntroCoverTimeOfTorus} (or rather the subset version
of this statement) and the coupling \prettyref{eq:IntroCoupling}.
The result \prettyref{eq:IntroLastTwoIndians} follows from \prettyref{eq:IntroPPoVCV},
using a calculation involving Palm measures and the fact that the
points of the limit Poisson point processes $\mathcal{N}^{z}$ are
``macroscopically separated''.

We now turn to the proof of \prettyref{eq:IntroCoverTimeOfTorus}.
The method builds on the ideas of the works \cite{Belius2010,BeliusCTinDC},
which contain the corresponding results for the so called cover levels
of random interlacements (\cite{Belius2010}) and the cover times
of finite sets in the discrete cylinder (\cite{BeliusCTinDC}). It
is a well known ``general principle'' that entrance times of small
sets in graphs often behave like exponential random variables. In
the case of the torus the entrance time $H_{x}$ of a vertex $x\in\mathbb{T}_{N}$
is approximately exponential with parameter $\frac{1}{g(0)}$ (this
can, for instance, be proven with a good quantitative bound using
the coupling \prettyref{eq:IntroCoupling}, see \prettyref{lem:RWOnePoint}).
Thus, in view of \prettyref{eq:IntroDefOfCoverTime}, we see that
the cover time $C_{N}$ is the maximum of identically distributed
exponential random variables with parameter $\frac{1}{g(0)}$. If
the $H_{x},x\in\mathbb{T}_{N}$, were also independent then standard
extreme value theory (or a simple direct calculation of the distribution
function of the maximum of i.i.d. exponentials) would give \prettyref{eq:IntroCoverTimeOfTorus}.
But clearly the $H_{x},x\in\mathbb{T}_{N}$, are not even approximately
independent, (for example if $x,y\in\mathbb{T}_{N}$ are neighbouring
vertices then $H_{x}$ and $H_{y}$ are highly dependent). There are
theorems that give distributional results for the maxima of random
fields with some sufficiently weak dependence (see \cite{LeadbetterRozen_OnExtValuesinStatRanFields,PereiraFerreira_LimitingCrossingProbs})
but these do not apply to the random field $(H_{x})_{x\in\mathbb{T}_{N}}$
because the dependence is too strong (using \prettyref{eq:IntroCoupling}
one can show that correlation between $1_{\{H_{x}>uN^{d}\}}$ and
$1_{\{H_{y}>uN^{d}\}}$ decays as $\frac{c(u)}{(d(x,y))^{d-2}}$,
see (1.68) of \cite{Sznitman2007}).

However for sets $F\subset\mathbb{T}_{N}$ that consist of isolated
vertices $x_{1},...,x_{n}$, that are ``well-separated'', it turns
out that using \prettyref{eq:IntroCoupling} one can show that $H_{x_{1}},...,H_{x_{n}}$,
are approximately independent (see \prettyref{lem:RWNPointFunc}).
By comparing to the independent case, we will therefore be able to show that for such $F$ 
\begin{equation}
\max_{x\in F}H_{x}\mbox{ has law close to }g(0)N^{d}\{\log|F|+G\},\label{eq:IntroSeparatedGumbel}
\end{equation}

\medskip
where $G$ is a standard Gumbel random variable. In particular it
will turn out that for such $F$, roughly speaking, the distribution
of $\max_{x\in F}H_{x}$ essentially depends only on the cardinality of $F$.

To enable the proof of \prettyref{eq:IntroCoverTimeOfTorus} we will
introduce the set of ``$(1-\rho)-$late points'' $F_{\rho}$ defined
as the vertices of $\mathbb{T}_{N}$ that have not been hit at time
$t(\rho)=(1-\rho)g(0)N^{d}\log N^{d}$, for a fixed but small $0<\rho<1$.
Note that this is a $(1-\rho)$ fraction of the ``typical'' cover
time $g(0)N^{d}\log N^{d}$. By the Markov property $Y_{t(\rho)+\cdot}$
is a random walk, and using a mixing argument one can show that for
our purposes it is basically independent from the random walk $(Y_{t})_{0\le t\le t(\rho)}$,
so that the law of $C_{N}$ is approximately the law of $t(\rho)+\max_{x\in F'}H_{x}$,
where $F'$ is a random set which is independent of the random walk
$Y_{\cdot}$, but has the law of $F_{\rho}$.

Furthermore we will be able to show, using \prettyref{eq:IntroCoupling},
that ``with high probability'' $F'$ (and $F_{\rho}$) consists
of ``well-separated'' vertices, and that the cardinality of $F'$
concentrates around its expected value, which is close to $|\mathbb{T}_{N}|^{\rho}=N^{d\rho}$.
Thus as long as $F'$ is ``typical'', in the sense that it is well-separated
and has cardinality close to $N^{d\rho}$, we will ``know'' that
$\max_{x\in F'}H_{x}$ has law close to $g(0)N^{d}\{\rho\log N^{d}+G\}$
(see \prettyref{eq:IntroSeparatedGumbel}). Adding the deterministic
time $t(\rho)$ the $\rho$ will cancel and we will get that $C_{N}$
has law close to $g(0)N^{d}\{\log N^{d}+G\}$, which is essentially
speaking the claim in \prettyref{eq:IntroCoverTimeOfTorus}.

We turn finally to the proof of the coupling \prettyref{eq:IntroCoupling}.
It roughly speaking adapts the ``poissonization'' method used for
the case $n=1$ in \cite{TeixeiraWindischOnTheFrag} to the case $n>1$,
and combines it with a decoupling technique from \cite{SznitmanDecoupling}. 

The first step is to consider the \emph{excursions} of $Y_{\cdot}$,
that is the pieces of path $Y_{(R_{k}+\cdot)\wedge U_{k},k=1,2,...}$
where $R_{k}$ and $U_{k}$ are recursively defined, $U_{0}=0$, $R_{k}$
is the first time $Y_{\cdot}$ enters $A_{1}\cup...\cup A_{n}$ after
time $U_{k-1}$ and $U_{k}$ is the first time after $R_{k}$ that
the random walk has spent ``a long time far away from $A_{1}\cup...\cup A_{n}$''
(see \prettyref{eq:DefOfUandtStar}). By proving, using a mixing argument,
that the distribution of $Y_{U_{k}}$ is close to a certain probability
distribution on $\mathbb{T}_{N}$ known as the quasistationary distribution,
regardless of the value of $Y_{R_{k}}$, we will be able to couple
the excursions $Y_{(R_{k}+\cdot)\wedge U_{k},k=1,2,...}$ with \emph{independent}
excursions $\tilde{Y}^{1},\tilde{Y}^{2},...$ that have the law of
$Y_{\cdot\wedge U_{1}}$, when $Y_{\cdot}$ starts from the quasistationary
distribution, such that ``with high probability'' the traces of
$Y_{(R_{k}+\cdot)\wedge U_{k}}$ and $\tilde{Y}^{k}$ in $A_{1}\cup...\cup A_{n}$
coincide.

We will then collect a Poisson number of such independent excursions
in a point process $\mu$ (in fact two different point processes,
but for the purpose of this discussion let us ignore this) which will
be such that the trace of $\mu$ in $A_{1}\cup....\cup A_{n}$ with
high probability coincides with the trace of the random walk $Y_{\cdot}$
run up to time $uN^{d}$ in that set. Because of the way we construct
the point process $\mu$, it will be a Poisson point process on the
space of paths in $\mathbb{T}_{N}$ of a certain intensity related
to the law of $\tilde{Y}^{1}$. This will complete the step that we
refer to as ``poissonization''.

We will see that an ``excursion'' in the Poisson point process $\mu$
may visit several of the boxes $A_1,\ldots,A_n$, and, roughly speaking, ``feels the geometry
of the torus'', since it may wind its way all around it before the
time $U$. To deal with this we in essence split the excursions of
$\mu$ into the pieces of excursion between successive returns to
the set $A_{1}\cup...\cup A_{n}$ and successive departures from $B_{1}\cup...\cup B_{n}$,
where the $B_{i}\supset A_{i}$ are larger boxes (still disjoint and
at most mesoscopic), and use a decoupling technique to remove the
dependence between pieces from the same excursion. 

We then collect these, now independent, pieces of excursion into a
point processes which we will be able to couple with a Poisson point
processes $\nu$ (in fact two independent Poisson point processes)
on the space of random paths in the torus, such that with high probability
the trace of $\nu$ in $A_{1}\cup...\cup A_{n}$ coincides with the
trace of $\mu$ (and therefore also with the trace of the random walk
$Y_{\cdot}$ run up to time $uN^{d}$) in that set. The ``excursions''
of $\nu$ start in a box $A_{i}$ and end upon leaving $B_{i}\supset A_{i}$
(which are disjoint), so they visit only one box and do not ``feel
the geometry of the torus'' since $B_{i}$ can be identified with
a subset of $\mathbb{Z}^{d}$.

Also since $\nu$ is a Poisson point process we will see that we can
split it into $n$ independent Poisson point processes, one for each
box $A_{i}$, such that roughly speaking the trace of $\nu_{i}$ in
$A_{i}$ coincides with high probability with that of $Y_{\cdot}$
(run up to time $uN^{d}$) in $A_{i}$. Thus we will have ``decoupled''
the boxes.

Now, as mentioned above, random interlacements are constructed from
a ``Poisson cloud'' on a certain space on paths, and we will see
that when restricted to a box $A_{i}$, a random interlacement has
the law of the trace of a Poisson number of random walks. This will
also basically be the law of trace of the $\nu_{i}$, with the difference
that the paths in $\nu_{i}$ do not return to $A_{i}$ after leaving
$B_{i}$, while for random interlacements a small but positive proportion
the paths do return. By taking a small number of the paths from the
$\nu_{i}$, and ``gluing them together'' to form paths that do return
to $A_{i}$ after leaving $B_{i}$, we will be able to construct from
the $\nu_{i}$ the random interlacements $(\mathcal{I}_{i}^{v})_{v\ge0}$
in \prettyref{eq:IntroCoupling}.

We now describe how this article is organized. In \prettyref{sec:Notation}
we introduce some notation and preliminary lemmas. In \prettyref{sec:ProperStatements}
we give the formal statements of the main theorems \prettyref{eq:IntroCoverTimeOfTorus}
and \prettyref{eq:IntroCoupling}, and of corollaries \prettyref{eq:IntroPPoVCV}
and \prettyref{eq:IntroLastTwoIndians}. We also derive the corollaries
from the main theorems. The proof of the cover time result \prettyref{eq:IntroCoverTimeOfTorus},
from \prettyref{eq:IntroCoupling}, is contained in \prettyref{sec:ProofOfGumbel}.
The subsequent sections deal with the proof of \prettyref{eq:IntroCoupling}.
In \prettyref{sec:Coupling} we introduce three further couplings
and use them to construct the coupling \prettyref{eq:IntroCoupling}.
The first of the three, a coupling of the excursions $Y_{(R_{k}+\cdot)\wedge U_{k}}$
with the Poisson point process $\mu$, is then constructed in \prettyref{sec:Quasistationary}
and \prettyref{sec:Poissonization}. In \prettyref{sec:Quasistationary}
we define the quasistationary distribution and prove that the law
of $Y_{U_{k}}$ is close to it. In \prettyref{sec:Poissonization}
we use this fact to construct the coupling of the excursions $Y_{(R_{k}+\cdot)\wedge U_{k}}$
with the Poisson point process $\mu$. The second of the three couplings,
a coupling of $\mu$ and the i.i.d. Poisson point processes $\nu_{1},...,\nu_{n}$,
is constructed in \prettyref{sec:OutOfTorus}. The third coupling,
a coupling of a Poisson point process $\nu$ with the law of the $\nu_{i}$
and random interlacements, is constructed in \prettyref{sec:ToRI}.
The appendix contains the proof of a certain lemma (\prettyref{lem:ETTEntraceTimesTorus})
which is stated in \prettyref{sec:Quasistationary}.

We finish this section with a remark on constants. Unnamed constants
are represented by $c$, $c'$, etc. Note that these may represent different constants
in different formulas or even within the same formula. Named constants
are denoted by $c_{4},c_{5},...$ and have fixed values. All constants
are understood to be positive and, unless stated otherwise, depend
only on the dimension $d$. Dependence on e.g. a parameter $\alpha$
is denoted by $c(\alpha)$ or $c_{4}(\alpha)$.

\section{\label{sec:Notation}Notation and preliminary lemmas}

In this section we introduce basic notation and a few preliminary
lemmas.

We write $[x]$ for the integer part of $x\in[0,\infty)$. The cardinality
of a set $U$ is denoted by $|U|$.

We denote the $d-$dimensional discrete torus of side length $N\ge3$
by $\mathbb{T}_{N}=(\mathbb{Z}/N\mathbb{Z})^{d}$ for $d\ge1$. If
$x\in\mathbb{Z}^{d}$ we write $|x|$ for the Euclidean norm of $x$
and $|x|_{\infty}$ for the $l_{\infty}$ norm of $x$. We take $d(\cdot,\cdot)$
to mean the distance on $\mathbb{T}_{N}$ induced by $|\cdot|$ and
$d_{\infty}(\cdot,\cdot)$ to mean the distance induced by $|\cdot|_{\infty}$.
The closed $l_{\infty}-$ball of radius $r\ge0$ with centre $x$
in $\mathbb{Z}^{d}$ or $\mathbb{T}_{N}$ is denoted by $B(x,r)$.

We define the inner and outer boundaries of a set $U\subset\mathbb{Z}^{d}$
or $U\subset\mathbb{T}_{N}$ by
\begin{equation}
\partial_{i}U=\{x\in U:d(x,U^{c})=1\},\partial_{e}U=\{x\in U^{c}:d(x,U)=1\}.\label{eq:DefOfBoundary}
\end{equation}

\n
For a set $U$ we write $\Gamma(U)$ for the space of all cadlag piecewise
constant functions from $[0,\infty)$ to $U$, with at most a finite
number of jumps in any compact interval. When only finitely many jumps
occur for a function $w\in\Gamma(U)$ we set $w(\infty)=\lim_{t\rightarrow\infty}w(t)$.
Usually we will work with $\Gamma(\mathbb{Z}^{d})$ or $\Gamma(\mathbb{T}_{N})$.
We write $Y_{t}$ for the canonical process on $\Gamma(\mathbb{Z}^{d})$
or $\Gamma(\mathbb{T}_{N})$. When $w\in\Gamma(\mathbb{Z}^{d})$ or
$w\in\Gamma(\mathbb{T}_{N})$ we take $w(a,b)$ to mean the range
$\{w(t):t\in[a,b]\cap[0,\infty)\}\subset\mathbb{Z}^{d}$ or $\mathbb{T}_{N}$
(with this definition the range is empty if $b<0$ or $a>b$). We
let $\theta_{t}$ denote the canonical shift on $\Gamma(\mathbb{Z}^{d})$
and $\Gamma(\mathbb{T}_{N})$. The jump times of $Y_{t}$ are defined
by
\begin{equation}
\tau_{0}=0,\tau_{1}=\inf\{t\ge0:Y_{t}\ne Y_{0}\}\mbox{ and }\tau_{n}=\tau_{1}\circ\theta_{\tau_{n-1}}+\tau_{n-1},n\ge2.\label{eq:DefOfJumpTimes}
\end{equation}
Due to the usual interpretation of the infimum of the empty set, $\tau_{m}=\infty$
for $m>n$ when only $n$ jumps occur. For a set $U\subset\mathbb{Z}^{d}$
or $\mathbb{T}_{N}$ we define the entrance time, return time and
exit time by
\begin{equation}
H_{U}=\inf\{t\ge0:Y_{t}\in U\},\tilde{H}_{U}=\inf\{t\ge\tau_{1}:Y_{t}\in U\},T_{U}=\inf\{t\ge0:Y_{t}\notin U\}.\label{eq:DefOfEntraceTime}
\end{equation}

\n
We let $P_{x}^{\mathbb{Z}^{d}}$ denote law on $\Gamma(\mathbb{Z}^{d})$
of continuous time simple random in $\mathbb{Z}^{d}$ and let $P_{x}$
denote the law on $\Gamma(\mathbb{T}_{N})$ of continuous time simple
random walk on $\mathbb{T}_{N}$ (so that $\tau_{1}$ is an exponential
random variable with parameter $1$). If $\nu$ is a measure on $\mathbb{Z}^{d}$
we let $P_{\nu}^{\mathbb{Z}^{d}}=\sum_{x\in\mathbb{Z}^{d}}\nu(x)P_{x}^{\mathbb{Z}^{d}}$.
We define $P_{\nu}$ analogously. Furthermore $\pi$ denotes the uniform
distribution on $\mathbb{T}_{N}$, and $P$
denotes $P_{\pi}$, i.e. the law of simple random walk starting from the uniform distribution. 

Essentially because the mixing time of the torus is of order $N^{2}$
(see Proposition 4.7, p. 50, and Theorem 5.5, p. 66 in \cite{LevPerWilMarkovChainsandMixingTimes})
we have that
\begin{equation}\label{eq:TorusMixing}
\begin{array}{l}
\mbox{for }N\ge3,\lambda\ge1,x\in\mathbb{T}_{N},\mbox{ a coupling }q(w,v),w,v\in\mathbb{T}_{N}\mbox{ exists for which the first}\\
\mbox{marginal is }P_{x}[Y_{\lambda N^{2}}\in dw],\mbox{ the second is uniform, and }\sum_{w\ne v}q(w,v)\le ce^{-c\lambda}.
\end{array}
\end{equation}

\n
The Green function is defined by
\[
g(x,y)=\int_{0}^{\infty}P_{x}^{\mathbb{Z}^{d}}[Y_{t}=y]dt,x,y\in\mathbb{Z}^{d},\mbox{ and }g(\cdot)=g(0,\cdot).
\]
We have the following classical bounds on $g(x,y)$ (see Theorem 1.5.4,
p. 31 in \cite{LawlersLillaGrona})
\begin{equation}
c|x-y|^{2-d}\le g(x,y)\le c'|x-y|^{2-d}\mbox{ for }x,y\in\mathbb{Z}^{d},x\ne y,d\ge3.\label{eq:GreensFunctionBounds}
\end{equation}

\smallskip\n
For $K\subset\mathbb{Z}^{d}$ we define the equilibrium measure $e_{K}$
and the capacity $\mbox{cap}(K)$ by
\begin{equation}
e_{K}(x)=P_{x}^{\mathbb{Z}^{d}}[\tilde{H}_{K}=\infty]1_{K}(x)\mbox{ and }\mbox{cap}(K)=\dsl_{x\in\partial_{i}K}e_{K}(x).\label{eq:DefOfeKandCap}
\end{equation}
It is well-known (see (2.16), Proposition 2.2.1 (a), p. 52-53 in \cite{LawlersLillaGrona})
that
\begin{equation}
cr^{d-2}\le\mbox{cap}(B(0,r))\le c'r^{d-2}\mbox{ for }r\ge1,d\ge3.\label{eq:AsymptoticsOfCapOfBox}
\end{equation}

\medskip\n
The normalised equilibrium distribution $\frac{e_{K}(\cdot)}{{\rm cap}(A)}$
can be thought of as the hitting distribution on $K$ when ``starting
from infinity'', and in this paper we will use that for all $K\subset B(0,r)\subset\mathbb{Z}^{d}$,$r\ge1$,
we have (see Theorem 2.1.3, Exercise 2.1.4 and (2.13) in \cite{LawlersLillaGrona})
\begin{equation}
\Cl[c]{eqhitlower}\mbox{\f $\dis\frac{e_{K}(y)}{\rm cap (K)}$}\le P_{x}[Y_{H_{K}}=y|H_{K}<\infty]\le \Cl[c]{eqhitupper}\mbox{\f $\dis\frac{e_{K}(y)}{\rm cap(K)}$}\mbox{ for all }y\in K,x\notin B(0,\Cl[c]{eqhitdist}r).\label{eq:EquilDistHitDistFromFar}
\end{equation}

\smallskip\n
If $K\subset U\subset\mathbb{Z}^{d}$, with $K$ finite, we define
the equilibrium measure and capacity of $K$ relative to $U$ by
\begin{equation}\label{eq:DefOfRelativeEquilAndCap}
e_{K,U}(x)=P_{x}^{\mathbb{Z}^{d}}[\tilde{H}_{K}>T_{U}]1_{K}(x)\mbox{ and }\mbox{cap}_{U}(K)=\dsl_{x\in\partial_{i}K}e_{K,U}(x).
\end{equation}

\smallskip\n
We will need the following bounds on the probability of hitting sets
in $\mathbb{Z}^{d}$ and $\mathbb{T}_{N}$.
\begin{lem}
\label{lem:BHBBallHittingBound}$(d\ge3,N\ge3)$
\begin{align}
 & P_{x}^{\mathbb{Z}^{d}}[H_{B(0,r_{1})}<\infty]\le c(r_{1}/r_{2})^{d-2}\mbox{ for }1\le r_{1}\le r_{2},x\notin B(0,r_{2}).\label{eq:BHBZdBallHitting}\\
 & \underset{x\notin B(0,r_{2})}{\sup}P_{x}[H_{B(0,r_{1})}<N^{2+\lambda}]\le c(\lambda)(r_{1}/r_{2})^{d-2}\mbox{ for }1\le r_{1}\le r_{2}\le N^{1-3\lambda},\lambda>0.\label{eq:BHBBallHittingBound}
\end{align}
\end{lem}
\begin{proof}
\eqref{eq:BHBZdBallHitting} follows from Proposition 2.2.2, p.53
in \cite{LawlersLillaGrona} and \eqref{eq:AsymptoticsOfCapOfBox}.
To prove \eqref{eq:BHBBallHittingBound} we let $K=\cup_{y\in\mathbb{Z}^{d},|y|_{\infty}\le N^{\lambda}}B(yN,r_{1})\subset\mathbb{Z}^{d}$
and note that by ``unfolding the torus'' we have
\begin{equation}
\sup_{z\notin B(0,r_{2})}P_{z}[H_{B(0,r_{1})}<N^{2+\lambda}]\le\sup_{z\in B(0,\frac{N}{2})\backslash B(0,r_{2})}P_{z}^{\mathbb{Z}^{d}}[H_{K}<\infty]+P_{0}^{\mathbb{Z}^{d}}[T_{B(0,\frac{N^{1+\lambda}}{2})}\le N^{2+\lambda}],\label{eq:BHBTorusToZd}
\end{equation}

\medskip\n
provided $N\ge c(\lambda)$. For any $z\in B(0,\frac{N}{2})\backslash B(0,r_{2})$
\begin{eqnarray*}
P_{z}^{\mathbb{Z}^{d}}[H_{K}<\infty] &\!\!\! \le &\!\!\! \dsl_{|y|_{\infty}\le1}P_{z}^{\mathbb{Z}^{d}}[H_{B(yN,r_{1})}<\infty]+\dsl_{1<|y|_{\infty}\le N^{\lambda}}P_{z}^{\mathbb{Z}^{d}}[H_{B(yN,r_{1})}<\infty]\\
 &\!\!\! \overset{\eqref{eq:BHBZdBallHitting}}{\le} &\!\!\! c\left(r_{1}/r_{2}\right)^{d-2}+cN^{\lambda d}(r_{1}/N)^{d-2}\le c(r_{1}/r_{2})^{d-2},
\end{eqnarray*}
since $N^{\lambda d}/N^{d-2}\overset{d\ge3}{\le}1/N^{(1-3\lambda)(d-2)}\overset{r_{2}\le N^{1-3\lambda}}{\le}1/r_{2}^{d-2}$. Furthermore
\[
P_{0}^{\mathbb{Z}^{d}}[T_{B(0,\frac{N^{1+\lambda}}{2})}\le N^{2+\lambda}]\le P_{0}^{\mathbb{Z}^{d}}\Big[|Y_{\tau_{n}}|>\mbox{\f $\dis\frac{N^{1+\lambda}}{2}$}\mbox{ for an }n\le2N^{2+\lambda}\Big]+P_{0}^{\mathbb{Z}^{d}}[\tau_{[2N^{2+\lambda}]}\le N^{2+\lambda}].
\]

\smallskip\n
But by applying Azuma's inequality (Theorem 7.2.1, p.~99 in \cite{AlonSpencerProbabilisticMethod})
to the martingale $Y_{\tau_{n}}$ we get that the first probability
on the right-hand side is bounded above by $cN^{c}e^{-cN^{\lambda}}$,
and by a standard large deviations bound the second probability is
bounded above by $e^{-cN^{2+\lambda}}$, so since $cN^{c}e^{-cN^{\lambda}}\le(r_{1}/r_{2})^{d-2}$
for $N\ge c(\lambda)$ we get \eqref{eq:BHBBallHittingBound}.
\end{proof}

Using \prettyref{lem:BHBBallHittingBound} we get the following bounds
for equilibrium measures.
\begin{lem}
\label{lem:EquilDistEstimates}$(d\ge3,N\ge3)$
\begin{eqnarray}
 & e_{K}(x)\ge cr^{-1},\mbox{ for }x\in\partial_{i}K,\mbox{ where }K=B(0,r),r\ge1.\label{eq:EquilDistLowerBound}\\
 & e_{K}(x)\le e_{K,U}(x),\mbox{ for all }x\in K\subset U\subset B(0,\frac{N}{4})\subset\mathbb{T}_{N}.\label{eq:RelativeEquilLowerBound}
\end{eqnarray}
\end{lem}

\medskip
Furthermore if $K\subset B(0,r)\subset U=B(0,r^{1+\lambda})\subset B(0,\frac{N}{4})\subset\mathbb{T}_{N},r\ge1,\lambda>0,$
we have 
\begin{equation}
e_{K,U}(x)\le(1+c(\lambda)r^{-\lambda})e_{K}(x),\mbox{ for all }x\in K.\label{eq:RelativeEquilUpperBound}
\end{equation}

\begin{proof}
For a large enough constant $c'$ we have $\inf_{x\in\partial_{i}B(0,c'r)}P_{x}^{\mathbb{Z}^{d}}[H_{B(0,r)}=\infty]\ge\frac{1}{2}$
(by \prettyref{eq:BHBZdBallHitting}), so \eqref{eq:EquilDistLowerBound}
follows from $\inf_{x\in\partial_{i}B(0,r)}P_{x}[\tilde{H}_{K}>T_{B(0,c'r)}]\ge cr^{-1}$
(which is a result of a one dimensional simple random walk estimate)
and the strong Markov property. The second inequality \eqref{eq:RelativeEquilLowerBound}
is obvious from \prettyref{eq:DefOfeKandCap} and \eqref{eq:DefOfRelativeEquilAndCap}.
 Finally for \eqref{eq:RelativeEquilUpperBound} note that $e_{K,U}(x)=e_{K}(x)+P_{x}^{\mathbb{Z}^{d}}[T_{U}<\tilde{H}_{K}<\infty]$ (by \prettyref{eq:DefOfeKandCap} and \prettyref{eq:DefOfRelativeEquilAndCap}),
and $P_{x}^{\mathbb{Z}^{d}}[T_{U}<\tilde{H}_{K}<\infty]\le P_{x}^{\mathbb{Z}^{d}}[T_{U}<\tilde{H}_{K}]\sup_{x\in\partial_{e}U}P_{x}^{\mathbb{Z}^{d}}[H_{K}<\infty]\le cr^{-\lambda}P_{x}^{\mathbb{Z}^{d}}[T_{U}<\tilde{H}_{K}]$
(by \prettyref{eq:BHBZdBallHitting}). Now $\inf_{x\in\partial_{e}U}P_{x}^{\mathbb{Z}^{d}}[H_{K}=\infty]$
is always positive and at least $\frac{1}{2}$ when $r\ge c(\lambda)$
(by \eqref{eq:BHBZdBallHitting}), so in fact $P_{x}^{\mathbb{Z}^{d}}[T_{U}<\tilde{H}_{K}<\infty]\le c(\lambda)r^{-\lambda}P_{x}^{\mathbb{Z}^{d}}[T_{U}<\tilde{H}_{K}]\inf_{x\in\partial_{e}U}P_{y}^{\mathbb{Z}^{d}}[H_{K}=\infty]\le c(\lambda)r^{-\lambda}e_{K}(x)$,
so \eqref{eq:RelativeEquilUpperBound} follows. 
\end{proof}
We now introduce some notation related to Poisson point processes.
Let $\Gamma=\Gamma(\mathbb{T}_{N})$ or $\Gamma=\Gamma(\mathbb{Z}^{d})$.
When $\mu$ is a Poisson point process on $\Gamma^{i},i\ge1,$ we
denote the trace of $\mu$ by
\begin{equation}
\mathcal{I}(\mu)=\mbox{\f $\dis\bigcup\limits_{(w_{1},...w_{i})\in\rm{Supp}(\mu)}$}\;\mbox{\f $\dis\bigcup\limits_{j=1}^{i}$}\;w_{i}(0,\infty)\subset\mathbb{Z}^{d}\mbox{ or }\mathbb{T}_{N}.\label{eq:DefOfPPPTrace}
\end{equation}

\smallskip\n
We will mostly consider Poisson point processes $\mu$ on $\Gamma$
where this simplifies to $\mathcal{I}(\mu)=$ \linebreak $\bigcup_{w\in\rm{Supp}(\mu)}w(0,\infty)$,
but in \prettyref{sec:OutOfTorus} and \prettyref{sec:ToRI} we will
also consider Poisson point processes $\mu$ on $\Gamma^{i},i\ge2.$
If $\mu$ is a Poisson point process of labelled paths (that is a
Poisson point process on $\Gamma\times[0,\infty)$, where $\Gamma$
is as above) we denote the trace up to label $u$ by
\begin{equation}
\mathcal{I}^{u}(\mu)=\mathcal{I}(\mu_{u})\subset\mathbb{Z}^{d}\mbox{ or }\mathbb{T}_{N},\mbox{ where }\mu_{u}(dw)=\mu(dw\times[0,u]).\label{eq:PPPTraceOfProcess}
\end{equation}

\n
Next let us recall some facts about random interlacements. They are,
roughly speaking, defined as a Poisson point process on a certain
space of labelled doubly-infinite trajectories modulo time-shift.
The \emph{random interlacement at level $u$,} or $\mathcal{I}^{u}\subset\mathbb{Z}^{d}$,
is the trace of trajectories with labels up to $u\ge0$. The family
of random subsets $(\mathcal{I}^{u})_{u\ge0}$ is called a \emph{random
interlacement}. The rigorous definitions and construction that make
this informal description precise can be found in Section 1 of \cite{Sznitman2007}
or Section 1 of \cite{SidoraviciusSznitman2009}. In this article
we will only use the facts \prettyref{eq:LawOfRIInFiniteSet}-\eqref{eq:RIOneAndTwoPointFunctions} which
now follow.
\begin{align}\label{eq:LawOfRIInFiniteSet}
 &   \mbox{There exists a space }(\Omega_{0},\mathcal{A}_{0},Q_{0})\mbox{ and a family }(\mathcal{I}^{u})_{u\ge0}\mbox{ of random subsets }\\[-1ex]
&\mbox{of }\mathbb{Z}^{d}\mbox{ such that }(\mathcal{I}^{u}\cap K)_{u\ge0}\overset{\rm law}{=}(\mathcal{I}^{u}(\mu_{K})\cap K)_{u\ge0}\mbox{ for all finite }K\subset\mathbb{Z}^{d}, \nonumber
\\
 & \mbox{where }\mu_{K}\mbox{ is a Poisson point process on }\Gamma(\mathbb{Z}^{d})\times[0,\infty)\mbox{ of intensity }P_{e_{K}}^{\mathbb{Z}^{d}}\otimes\lambda,\label{eq:DefOfMuK}
\end{align}
and $\lambda$ denotes Lebesgue measure (see (0.5), (0.6), (0.7) in
\cite{Sznitman2007}, also cf. (1.67) in \cite{Sznitman2007}).

(From \prettyref{eq:LawOfRIInFiniteSet} and \prettyref{eq:DefOfMuK}
we see that to ``simulate'' $\mathcal{I}^{u}\cap K$ for a finite
$K\subset\mathbb{Z}^{d}$ one simply samples an integer $n\ge0$ according
to the Poisson distribution with parameter $u\mbox{cap}(K)$, picks
$n$ random starting points $Z_{1},...,Z_{n}$ according to the normalized
equilibrium distribution $\frac{e_{K}(\cdot)}{\mbox{cap}(K)}$ and
runs from each starting point $Z_{i}$ an independent random walk,
recording the trace the walks leave in the set $K$. The ``full'' random 
interlacement $\mathcal{I}^u$ can be seen, intuitively speaking, as a
``globally consistent'' version of the traces of the Poisson point processes $\mu_K$.)

Random interlacements also satisfy (see above (0.5), (0.7), below (0.8) and (1.48) in \cite{Sznitman2007})
\begin{align}
 & \mbox{the law of }\mathcal{I}^{u}\mbox{ under }Q_{0}\mbox{ is translation invariant for all }u\ge0,\label{eq:TranslationInvariantInterlacements}\\[1ex]
 & \mathcal{I}^{u}\mbox{ is increasing in the sense that }Q_{0}-\mbox{almost surely }\mathcal{I}^{v}\subset\mathcal{I}^{u}\mbox{ for }v\le u,\label{eq:IncreasingInterlacements}
 \\[1ex]
 &\mbox{if }\mathcal{I}_{1}^{u}\mbox{ and }\mathcal{I}_{2}^{v}\mbox{ are independent with the laws under }Q_{0}\mbox{ of }\mathcal{I}^{u}\mbox{ and }\mathcal{I}^{v}\label{eq:AdditiveInterlacements}
 \\[-1ex]
&\mbox{respectively, then }(\mathcal{I}_{1}^{u},\mathcal{I}_{1}^{u}\cup\mathcal{I}_{2}^{v})\mbox{ has the law of }(\mathcal{I}^{u},\mathcal{I}^{u+v})\mbox{ under }Q_{0}. \nonumber
\end{align}

\smallskip\n
Finally (see (1.58) and (1.59) of \cite{Sznitman2007})
\begin{equation}
Q_{0}[x\notin\mathcal{I}^{u}]=\exp\big(-\mbox{\f $\dis\frac{u}{g(0)}$}\big)\mbox{ and }Q_{0}[x,y\notin\mathcal{I}^{u}]=\exp\big(-\mbox{\f $\dis\frac{2u}{g(0)+g(x-y)}$}\big),x,y\in\mathbb{Z}^{d}.\label{eq:RIOneAndTwoPointFunctions}
\end{equation}
(The properties \eqref{eq:IncreasingInterlacements}-\eqref{eq:RIOneAndTwoPointFunctions}
in fact also follow from \eqref{eq:LawOfRIInFiniteSet} and \eqref{eq:DefOfMuK}).

\smallskip
The following lemma, which follows from Lemma 1.5 from \cite{BeliusCTinDC}
(by letting $a$ in (1.39) in \cite{BeliusCTinDC} go to infinity,
and using \eqref{eq:RIOneAndTwoPointFunctions}), will be crucial
in our proof of \prettyref{eq:IntroCoverTimeOfTorus}.
\begin{lem} \label{lem:TwoPointCalculation}$(d\ge3)$ 

\medskip
There is a constant $\Cl[c]{TPC}>1$ such that or any $K\subset\mathbb{Z}^{d}$ with $0\notin K$ 
\begin{equation}
\dsl_{v\in K}Q_{0}[0,v\notin\mathcal{I}^{u}]\le|K|\left(Q_{0}[0\notin\mathcal{I}^{u}]\right)^{2}\{1+cu\}+ce^{-\Cr{TPC}\frac{u}{g(0)}},\mbox{ for all }u\ge0.\label{eq:TwoPointCalculation}
\end{equation}
\end{lem}

Finally we define the cover time $C_{F}$ of a set $F\subset\mathbb{T}_{N}$
by
\begin{equation}
C_{F}\stackrel{\rm def}{=} \max_{x\in F}H_{x}=\inf\{t\ge0:F\subset Y(0,t)\}.\label{eq:DefOfCoverTime}
\end{equation}
Note that $C_{N}=C_{\mathbb{T}_{N}}$, cf. \prettyref{eq:IntroDefOfCoverTime}.
For convenience we introduce the notation
\begin{equation}
u_{F}(z)=g(0)\{\log|F|+z\},\label{eq:DefOfuF}
\end{equation}

\smallskip\n
so that $\{\frac{C_{N}}{g(0)N^{d}}-\log N^{d}\le z\}=\{C_{N}\le u_{\mathbb{T}_{N}}(z)N^{d}\}$,
cf. \prettyref{eq:IntroCoverTimeOfTorus}.

\bigskip
\section{\label{sec:ProperStatements}Gumbel fluctuations, coupling with random
interlacements and corollaries}

In this section we state our two main theorems and derive two corollaries.
The first main result is \prettyref{thm:G_Gumbel}, which, roughly
speaking, says that the cover times of large subsets of the torus
(for $d\ge3$) have Gumbel fluctuations, (and implies \prettyref{eq:IntroCoverTimeOfTorus}
from the introduction). The second main result is \prettyref{thm:Coupling}
and it states the coupling of random walk in the torus and random
interlacements (see \prettyref{eq:IntroCoupling} in the introduction).
The proofs of the theorems are contained in the subsequent sections.

The first corollary is \prettyref{cor:PPoVCL} which in essence says
that the ``point process of vertices covered last'' (see \prettyref{eq:DefOfPPoVCV})
converges in law to a Poisson point process as the side length of
the torus goes to infinity (see \prettyref{eq:IntroPPoVCV}). The
second corollary is \prettyref{cor:LastTwoIndians} and roughly says
that for any fixed $k\ge1$ the last $k$ vertices to be hit by the
random walk are far apart, at distance of order $N$.

We now state our result about fluctuations of the cover time. Recall
the notation from \prettyref{eq:DefOfCoverTime} and \prettyref{eq:DefOfuF}.

\begin{thm} \label{thm:G_Gumbel}($d\ge3$,$N\ge3$) 

\medskip 
For all $F\subset\mathbb{T}_{N}$ we have
\begin{equation}
\sup_{z\in\mathbb{R}}\left|P[C_{F}\le N^{d}u_{F}(z)]-\exp(-e^{-z})\right|\le c|F|^{-c}.\label{eq:G_Gumbel}
\end{equation}
\end{thm}

We will prove \prettyref{thm:G_Gumbel} in \prettyref{sec:ProofOfGumbel}.

Next we will state the coupling result. For $n\ge1$ and $x_{1},...,x_{n}\in\mathbb{T}_{N}$
we define the 
\begin{equation}\label{eq:DefOfSeparation}
\mbox{separation $s$ of the vertices $x_{1},...,x_{n}\in\mathbb{T}_{N}$, by $s$} =  \left\{ \begin{array}{ll}
\!\! N & \mbox{if} \; n=1,\\
\!\! \min_{i\ne j}d_{\infty}(x_{i},x_{j}) & \mbox{if}\; n>1.
\end{array}\right. 
\end{equation}

\medskip\n
For an arbitrarily small $\varepsilon>0$ which does not depend on $N$ we define the box
\begin{equation}
A=B(0,s^{1-\varepsilon}).\label{eq:DefOfA}
\end{equation}
The result will couple the trace of random walk in the boxes $A+x_{1},...,A+x_{n}$
with independent random interlacements. Note that thanks to \eqref{eq:DefOfSeparation}
and \eqref{eq:DefOfA} the boxes are ``far part'' (in the sense that the distance between them, of order $s$,
is ``much larger'' than their radius, which is $s^{1-\varepsilon}$) and ``at most mesoscopic'' (in that their radius, at most $N^{1-\varepsilon}$, is ``much smaller'' than the side-length $N$ of the torus).
Given a level $u>0$ and a $\delta>0$, with these parameters satisfying appropriate
conditions, we will construct independent random interlacements $(\mathcal{I}_{1}^{v})_{v\ge0},...,(\mathcal{I}_{n}^{v})_{v\ge0}$
such that (recall the notation $Y(a,b)$ from above \prettyref{eq:DefOfJumpTimes})
\begin{equation}
Q_{1}[\mathcal{I}_{i}^{u(1-\delta)}\cap A\subset(Y(0,uN^{d})-x_{i})\cap A\subset\mathcal{I}_{i}^{u(1+\delta)}]\ge1-\Cl[c]{couplingerr}ue^{-cs^{\Cr{couplingerr}}}\mbox{ for all }i.\label{eq:CouplingInclusion}
\end{equation}

\smallskip\n
Formally we have
\begin{thm}\label{thm:Coupling} ($d\ge3$,$N\ge3$) 

\medskip
For $n\ge1$ let $x_{1},...,x_{n}\in\mathbb{T}_{N}$
be distinct and have separation $s$ (see \prettyref{eq:DefOfSeparation}),
and let $\varepsilon\in(0,1)$. Further assume $u\ge s^{-\Cr{couplingerr}},1\ge\delta\ge\frac{1}{\Cr{couplingerr}}s^{-\Cr{couplingerr}}$,
$n\le s^{\Cr{couplingerr}}$, where $\Cr{couplingerr}=\Cr{couplingerr}(\varepsilon)$. We can then
construct a space $(\Omega_{1},\mathcal{A}_{1},Q_{1})$ with a random
walk $Y_{\cdot}$ with law $P$ and independent random interlacements
$(\mathcal{I}_{i}^{v})_{v\ge0},i=1,...,n$, each with the law of $(\mathcal{I}^{v})_{v\ge0}$
under $Q_{0}$, such that \prettyref{eq:CouplingInclusion} holds.
\end{thm}

\prettyref{thm:Coupling} will be proved in \prettyref{sec:Coupling}.

In the proof of \prettyref{cor:PPoVCL} we will need the following
estimate on the probability of hitting a point, which is a straight-forward
consequence of \prettyref{thm:Coupling}, (when $n=1$).
\begin{lem} \label{lem:RWOnePoint} ($d\ge3$,$N\ge3$) 

\medskip
There exists a constant $\Cl[c]{OPu}$ such that if $N^{-\Cr{OPu}}\le u\le N^{\Cr{OPu}}$
then for all $x\in\mathbb{T}_{N}$ 
\begin{equation}
Q_{0}[0\notin\mathcal{I}^{u}](1-cN^{-c})\le P[x\notin Y(0,uN^{d})]\le Q_{0}[0\notin\mathcal{I}^{u}](1+cN^{-c}).\label{eq:RWOPFuncMultError}
\end{equation}
\end{lem}
\begin{proof}
We apply \prettyref{thm:Coupling} with $n=1$ (so that the separation
$s$ is $N$), $x_{1}=0$, $\varepsilon=\frac{1}{4}$ (say) and $\delta=\Cr{couplingerr}^{-1}N^{-\Cr{couplingerr}}$
(where $\Cr{couplingerr}=\Cr{couplingerr}(\frac{1}{4})$ is the constant from \prettyref{thm:Coupling}).
By choosing $\Cr{OPu}\le \Cr{couplingerr}$ we have $u\ge s^{-\Cr{couplingerr}}$, so that
all the conditions of \prettyref{thm:Coupling} are satisfied and
the coupling $Q_{1}$ of $Y_{\cdot}$ and random interlacements can
be constructed. Therefore it follows from \prettyref{eq:CouplingInclusion}
that 
\begin{equation}
Q_{0}[0\notin\mathcal{I}^{u(1+\delta)}]-cue^{-cN^{\Cr{couplingerr}}}\le P[0\notin Y(0,uN^{d})]\le Q_{0}[0\notin\mathcal{I}^{u(1-\delta)}]+cue^{-cN^{\Cr{couplingerr}}}.\label{eq:KallenbergProbabilityBound}
\end{equation}
But we also have, if we pick $\Cr{OPu}<\Cr{couplingerr}$, that
\begin{eqnarray*}
Q_{0}[0\notin\mathcal{I}^{u(1-\delta)}]+cue^{-N^{\Cr{couplingerr}}} & \!\!\! \overset{\eqref{eq:RIOneAndTwoPointFunctions}}{=} & \!\!\!  Q_{0}[0\notin\mathcal{I}^{u}](e^{\frac{\delta u}{g(0)}}+cue^{\frac{u}{g(0)}-cN^{\Cr{couplingerr}}})
\\
 & \!\!\! \le & \!\!\!  Q_{0}[0\notin\mathcal{I}^{u}](1+cN^{-c}),
\end{eqnarray*}

\medskip\n
since $cue^{\frac{u}{g(0)}-cN^{\Cr{couplingerr}}}\le cN^{c}e^{-cN^{c}}$ and
$\delta u\le cN^{\Cr{OPu}-\Cr{couplingerr}}\le cN^{-c}$ (recall $u\le N^{\Cr{OPu}},\delta=cN^{-\Cr{couplingerr}}$
and $\Cr{OPu}<\Cr{couplingerr}$). Similarly if $\Cr{OPu}<\Cr{couplingerr}$ we have $Q_{0}[0\notin\mathcal{I}^{u(1+\delta)}]-cue^{-cN^{\Cr{couplingerr}}}\ge Q_{0}[0\notin\mathcal{I}^{u}](1-cN^{-c})$,
so \prettyref{eq:RWOPFuncMultError} follows.
\end{proof}
We now state and prove the first corollary. The proof uses \prettyref{thm:G_Gumbel}
and \prettyref{thm:Coupling} (via \prettyref{lem:RWOnePoint}). Recall
the definition of $\mathcal{N}_{N}^{z}$ from \prettyref{eq:DefOfPPoVCV}.
\begin{cor}
\label{cor:PPoVCL}($d\ge3$) 

\medskip
\prettyref{eq:IntroPPoVCV} holds.\end{cor}

\begin{proof}
By Kallenberg's Theorem (Proposition 3.22, p. 156 of \cite{Resnick_ExtremeValRegVarAndPP})
the result follows from
\begin{align}
\lim_{N\rightarrow\infty}P[\mathcal{N}_{N}^{z}(I)=0]  = & \exp(-\lambda(I)e^{-z})\mbox{ and }\label{eq:Kallenberg1}\\
\lim_{N\rightarrow\infty}E[\mathcal{N}_{N}^{z}(I)]  = & \lambda(I)e^{-z},\label{eq:Kallenberg2}
\end{align}

\smallskip\n
for all $I$ in the collection $\mathcal{J}=\{I:I\mbox{ a finite union of products of open intervals in }(\mathbb{R}/\mathbb{Z})^{d},\lambda(I)>0\}$.
Note that
\begin{equation}
\lim_{N\rightarrow\infty}\mbox{\f $\dis\frac{|NI\cap\mathbb{T}_{N}|}{|\mathbb{T}_{N}|}$} =\lambda(I)\mbox{ for all }I\in\mathcal{J},\label{eq:ConvergenceToLebesgueMeasure}
\end{equation}
since $|NI\cap\mathbb{T}_{N}|/|\mathbb{T}_{N}|$ is the mass assigned to the set $I$ by the measure
$\sum_{x\in\mathbb{T}_{N}}N^{-d}\delta_{x/N}$ on $(\mathbb{R}/\mathbb{Z})^{d}$, which converges weakly to $\lambda$ as $N\rightarrow\infty$
(note that $I$ is a continuity set of $\lambda$, and see the Portmanteau theorem, Proposition 5.1, p. 9, \cite{RevuzYor}).

\smallskip\n
Fix an $I\in\mathcal{J}$. To check \prettyref{eq:Kallenberg1} note
that $\{\mathcal{N}_{N}^{z}(I)=0\}=\{C_{NI\cap\mathbb{T}_{N}}\le N^{d}u_{\mathbb{T}_{N}}(z)\}$
(see \prettyref{eq:DefOfPPoVCV}, \prettyref{eq:DefOfCoverTime} and
\prettyref{eq:DefOfuF}). Since $u_{\mathbb{T}_{N}}(z)=u_{NI\cap\mathbb{T}_{N}}(z-\log\frac{|NI\cap\mathbb{T}_{N}|}{N^{d}})$
(see \prettyref{eq:DefOfuF}) and $|NI\cap\mathbb{T}_{N}|\rightarrow\infty$
as $N\rightarrow\infty$ (recall $\lambda(I)>0$) we get from \prettyref{eq:G_Gumbel} that
\[
\big| P[\mathcal{N}_{N}^{z}(I)=0]-\exp\big(-e^{-(z-\log\frac{|NI\cap\mathbb{T}_{N}|}{N^{d}})}\big)\big|\rightarrow0\mbox{ as }N\rightarrow\infty.
\]

\n
But by \prettyref{eq:ConvergenceToLebesgueMeasure} we have $\exp(-e^{-(z-\log\frac{|NI\cap\mathbb{T}_{N}|}{N^{d}})})\rightarrow\exp(-\lambda(I)e^{-z})$
as $N\rightarrow\infty$ so \prettyref{eq:Kallenberg1} follows.

\medskip
To check \prettyref{eq:Kallenberg2} note that by \prettyref{eq:DefOfPPoVCV}
and \prettyref{eq:DefOfuF} 
\begin{equation}
E[\mathcal{N}_{N}^{z}(I)]=|NI\cap\mathbb{T}_{N}|P[0\notin Y(0,N^{d}u_{\mathbb{T}_{N}}(z))].\label{eq:KallenbergExpectationEquality}
\end{equation}

\medskip\n
Using \prettyref{eq:RWOPFuncMultError} with $u=u_{\mathbb{T}_{N}}(z)$
($\in[N^{-\Cr{OPu}},N^{\Cr{OPu}}]$ for $N\ge c(z)$) and $Q_{0}[0\notin\mathcal{I}^{u_{\mathbb{T}_{N}}(z)}]=\frac{e^{-z}}{N^{d}}$
(see \prettyref{eq:RIOneAndTwoPointFunctions}) we get from \prettyref{eq:KallenbergExpectationEquality} that
\[
e^{-z}\lim_{N\rightarrow\infty}\mbox{\f $\dis\frac{|NI\cap\mathbb{T}_{N}|}{N^{d}}$}(1-cN^{-c})\le\lim_{N\rightarrow\infty}E[\mathcal{N}_{N}^{z}(I)]\le e^{-z}\lim_{N\rightarrow\infty}\mbox{\f $\dis\frac{|NI\cap\mathbb{T}_{N}|}{N^{d}}$}(1+cN^{-c}).
\]

\n
Now using \prettyref{eq:ConvergenceToLebesgueMeasure} we find that
\prettyref{eq:Kallenberg2} holds. Thus the proof of \prettyref{eq:IntroPPoVCV}
is complete.
\end{proof}
See \prettyref{rem:endremark} for a potential generalisation of \prettyref{cor:PPoVCL}.
We now state the second corollary, which is a consequence of the
first. The proof follows those of Corollary 2.3 of \cite{BeliusCTinDC}
and Proposition 2.8 of \cite{Belius2010}, so we omit some details.
\begin{cor} \label{cor:LastTwoIndians}($d\ge3$) 

\medskip
Let $Z_{1},...,Z_{N^{d}}$, be the vertices of $\mathbb{T}_{N}$ ordered by entrance time,
\[
\mbox{so that }C_{\mathbb{T}_{N}}=H_{Z_{1}}>H_{Z_{2}}>...>H_{Z_{N^{d}}}.
\]

\n
Then for any $k\ge2$ and $0<\delta\le\frac{1}{4}$ 
\begin{equation}
\lim_{\delta\rightarrow0}\limsup_{N\rightarrow\infty}P[\exists1\le i<j\le k\mbox{ with }d(Z_{i},Z_{j})\le\delta N]=0,\label{eq:LastTwoIndians}
\end{equation}
or in other words ``the last $k$ vertices to be hit are far apart,
at distance of order $N$''.\end{cor}
\begin{proof}
Note that for any $N,\delta,k$ and $z\in\mathbb{R}$ the probability
in \prettyref{eq:LastTwoIndians} is bounded above by
\begin{equation}
P[\exists x,y\in\mbox{Supp}(\mathcal{N}_{N}^{z})\mbox{ with }d(x,y)\le\delta N]+P[\mathcal{N}_{N}^{z}((\mathbb{R}/\mathbb{Z})^{d})<k].\label{eq:LastTwoIndiansFirstUpperBound}
\end{equation}

\n
Furthermore the first probability in \prettyref{eq:LastTwoIndiansFirstUpperBound}
is bounded above by $E[\mathcal{N}_{N}^{z}\otimes\mathcal{N}_{N}^{z}(g)-$ $\mathcal{N}_{N}^{z}((\mathbb{R}/\mathbb{Z})^{d})]$, where $g$ is the function $(x,y)\rightarrow f(x-y)$, for $f:(\mathbb{R}/\mathbb{Z})^{d}\rightarrow[0,1]$
continuous with $f(x)=1$ when $d(0,x)\le\delta$ and $f(x)=0$ if
$d(0,x)\ge2\delta$. Thus using \prettyref{eq:IntroPPoVCV}, one sees
that the limsup in \prettyref{eq:LastTwoIndians} is bounded above by
\begin{equation}
\mathbb{E}[\mathcal{N}^{z}\otimes\mathcal{N}^{z}(g)-\mathcal{N}^{z}((\mathbb{R}/\mathbb{Z})^{d})]+\mathbb{P}[\mathcal{N}^{z}((\mathbb{R}/\mathbb{Z})^{d})<k].\label{eq:LastTwoIndiansSecondUpperBound}
\end{equation}

\n
Using a calculation involving Palm measures (we omit the details,
cf. (2.19) of \cite{BeliusCTinDC}) the first term in \prettyref{eq:LastTwoIndiansSecondUpperBound}
can be shown to be equal to $(e^{-z}\int_{(\mathbb{R}/\mathbb{Z})^{d}}f(x)dx)^{2}.$
Since $\int_{(\mathbb{R}/\mathbb{Z})^{d}}f(x)dx\le c\delta^{d}$ one
thus finds that for all $z\in\mathbb{R}$ the left-hand side of \prettyref{eq:LastTwoIndians}
is bounded above by $\mathbb{P}[\mathcal{N}^{z}((\mathbb{R}/\mathbb{Z})^{d})<k]$.
But $\mathbb{P}[\mathcal{N}^{z}((\mathbb{R}/\mathbb{Z})^{d})<k]\rightarrow0$
as $z\rightarrow-\infty$, so \prettyref{eq:LastTwoIndians} follows.
\end{proof}


\section{\label{sec:ProofOfGumbel}Coupling gives Gumbel fluctuations}

In this section we use the coupling result \prettyref{thm:Coupling}
to prove \prettyref{thm:G_Gumbel}, i.e. to prove that cover times
in $\mathbb{T}_{N}$ have Gumbel fluctuations. Essentially speaking
we combine the method of the proofs of Theorem 0.1 in \cite{Belius2010}
and Theorem 0.1 in \cite{BeliusCTinDC} with the coupling \prettyref{thm:Coupling}.

The first step is to prove that if $F\subset\mathbb{T}_{N}$ is smaller
than $N^{c}$ for some small exponent $c$ (but still ``large'')
and consists of isolated vertices separated by a distance at least
$|F|^{c}$, then the cover time $C_{F}$ is approximately distributed
as $N^{d}\{g(0)\log|F|+G\}$, where $G$ is a standard Gumbel random
variable. We prove this (uniformly in the starting vertex of the random
walk) in \prettyref{lem:GSPGumbelForSeparatedStartingFromPoint}.
To prove \prettyref{lem:GSPGumbelForSeparatedStartingFromPoint} we
use a mixing argument to reduce to the case when random walk starts
from the uniform distribution. This case is then proven in \prettyref{lem:GumbelForSeparatedStartingFromUniform}
using the fact that for vertices that are far apart, entrance times
are approximately independent (which is proven in \prettyref{lem:RWNPointFunc},
using the coupling result \prettyref{thm:Coupling}), and from this
a simple calculation will show that the cover time (which is then
the maximum of almost i.i.d. random variables) has distribution close
to the Gumbel distribution.

To get the Gumbel limit result for arbitrary $F\subset\mathbb{T}_{N}$
(e.g. for $F=\mathbb{T}_{N}$, where the entrance times of close vertices
are far from being independent) we consider the set of $(1-\rho)-$\emph{late
points} (using the terminology from \cite{Demboetal_LatePoints}),
which is the set of vertices that are not yet hit at a $1-\rho$ fraction
of the typical time it takes to cover $F$\emph{, }or more formally
\begin{equation}
F_{\rho}=F\backslash Y(0,t(\rho)),\mbox{ where }t(\rho)=N^{d}(1-\rho)g(0)\log|F|\mbox{ and }0<\rho<1.\label{eq:G_DefOfUncoveredSet}
\end{equation}
It turns out, roughly speaking, that if we use a fixed but small $\rho$
then with high probability this set consists of isolated vertices
that are at distance at least $|F|^{c}$ from one another, and that
$|F_{\rho}|$ concentrates around $E[|F_{\rho}|]$, which we will
see, is close to $|F|^{\rho}$. This is done in \prettyref{lem:G_GE_GoodEvent}
by relating the probability of two vertices not being hit by random
walk to the probability of two vertices not being in a random interlacement
using the coupling (in \prettyref{lem:RWTwoPoint}), and using \prettyref{lem:TwoPointCalculation}.

This will imply that with high probability the set $F_{\rho}$ fulfils
the conditions of \prettyref{lem:GSPGumbelForSeparatedStartingFromPoint}
(i.e. is ``smaller than $N^{c}$'' and separated), so that, using
the Markov property, we will be able to show that conditioned on $F_{\rho}$
the time $C_{F_{\rho}}\circ\theta_{t(\rho)}$ has distribution close
to $N^{d}g(0)\{\log|F_{\rho}|+G\}\approx N^{d}g(0)\{\rho\log|F|+G\}$,
where $G$ is a standard Gumbel random variable. Since, on the event
that $F_{\rho}$ is non-empty (which has probability close to one),
$C_{F}=t(\rho)+C_{F_{\rho}}\circ\theta_{t(\rho)}$ we will be able
to conclude that $C_{F}$ has distribution close to $N^{d}g(0)\{\log|F|+G\}$,
which is the claim of \prettyref{thm:G_Gumbel}.

We start by stating \prettyref{lem:GSPGumbelForSeparatedStartingFromPoint},
which says that the cover time has law close to the Gumbel distribution
for ``well-separated'' sets that are not too large.
\begin{lem} \label{lem:GSPGumbelForSeparatedStartingFromPoint}($d\ge3$,$N\ge3$)

\medskip

There are constants $\Cl[c]{smallsetgumbsize}$ and $\Cl[c]{smallsetgumbsep}$
such that if $F\subset\mathbb{T}_{N}$ satisfies $2\le|F|\le N^{\Cr{smallsetgumbsize}}$
and $d_{\infty}(x,y)\ge|F|^{\Cr{smallsetgumbsep}}$ for all $x,y\in F,x\ne y$, then
\begin{equation}
\sup_{z\in\mathbb{R},x\in\mathbb{T}_{N}}\big|P_{x}[F\subset Y(0,u_{F}(z)N^{d})]-\exp(-e^{-z})\big|\le c|F|^{-c}.\label{eq:GSPGumbelForSeparatedStartingFromPoint}
\end{equation}

\end{lem}

(Recall that we have defined the range $Y(0,u_{F}(z)N^{d})$ such
that it is empty if $u_{F}(z)N^{d}<0$.) We will prove \prettyref{lem:GSPGumbelForSeparatedStartingFromPoint}
after proving \prettyref{thm:G_Gumbel}. To prove \prettyref{thm:G_Gumbel}
we must prove something like \prettyref{eq:GSPGumbelForSeparatedStartingFromPoint}
for arbitrary $F\subset\mathbb{T}_{N}$. We do so by studying the
set of late points $F_{\rho}$ (recall \prettyref{eq:G_DefOfUncoveredSet}).
We will show that ``with high probability'' it belongs to the collection
$\mathcal{G}$ of ``good subsets of $F$'', where
\begin{equation}
\mathcal{G}=\{F'\subset F:\left||F'|-|F|^{\rho}\right|\le|F|^{\frac{2}{3}\rho}\mbox{ and }\inf_{x,y\in F',x\ne y}d_{\infty}(x,y)\ge|F|^{\frac{1}{2d}}\},\label{eq:GDefGoodSets}
\end{equation}
or in other words that it has cardinality close to $|F|^{\rho}$ and
is ``well-separated''. Formally:
\begin{lem} \label{lem:G_GE_GoodEvent}($d\ge3$,$N\ge3$) 

\medskip
There exists a constant $\Cl[c]{rhobound}$ such that for $0 < \rho\le \Cr{rhobound}$ and $F\subset\mathbb{T}_{N}$ with $|F|\ge c(\rho)$
\begin{equation}
P[F_{\rho}\notin\mathcal{G}]\le c|F|^{-c(\rho)}.\label{eq:G_GE_GoodEvent}
\end{equation}
\end{lem}

Before proving \prettyref{lem:G_GE_GoodEvent} we use it to prove
\prettyref{thm:G_Gumbel}. Recall that $\{C_{F}\le t\}\overset{\prettyref{eq:DefOfCoverTime}}{=}\{F\subset Y(0,t)\}$.
\begin{proof}[Proof of \prettyref{thm:G_Gumbel}.]
By \prettyref{eq:G_GE_GoodEvent} we have for $0<\rho\le \Cr{rhobound}$ and
$|F|\ge c(\rho)$ 
\begin{equation}
\left|P[C_{F}\le u_{F}(z)N^{d}]-P[F\subset Y(0,u_{F}(z)N^{d}),F_{\rho}\in\mathcal{G}]\right|\le c|F|^{-c(\rho)}.\label{eq:GToOnGoodEvent}
\end{equation}

\medskip\n
Also by the Markov property, if $|F|\ge c(\rho)$ (so that $\emptyset\notin\mathcal{G}$),
\begin{equation}\label{eq:G_ConditionOnFPrime}
\begin{array}{l}
P[F\subset Y(0,u_{F}(z)N^{d}),F_{\rho}\in\mathcal{G}]  \overset{\prettyref{eq:G_DefOfUncoveredSet}}{=} 
\\[0.5ex]
\underset{x\in\mathbb{T}_{N},F'\in\mathcal{G}}{\sum}  P[F_{\rho}=F',Y_{t(\rho)}=x] P_{x}[F'\subset Y(0,u_{F}(z)N^{d}-t(\rho))].
\end{array}
\end{equation}

\medskip\n
Set $h=\log\frac{|F'|}{|F|^{\rho}}$ so that $P_{x}[F'\subset Y(0,u_{F}(z)N^{d}-t(\rho))]=P_{x}[F'\subset Y(0,u_{F'}(z-h)N^{d})]$
(see \prettyref{eq:DefOfuF}).
Also fix $\rho=c \le \Cr{rhobound}$ small enough so that
$2d\rho\le \Cr{smallsetgumbsize}$ and $\frac{1}{4d\rho}\ge \Cr{smallsetgumbsep}$.
Then \eqref{eq:GToOnGoodEvent} holds when $|F|\ge c$. Furthermore 
\prettyref{lem:GSPGumbelForSeparatedStartingFromPoint} applies
to all $F'\in\mathcal{G}$ when $|F|\ge c$, since by \eqref{eq:GDefGoodSets}
every $F'\in\mathcal{G}$ satisfies $|F'|\le2|F|^{\rho}\le|F|^{2\rho}$,
if $|F|\ge c=c(\rho)$, so that $2d\rho\le \Cr{smallsetgumbsize}$ implies $|F'|\le|F|^{2\rho}\le N^{2d\rho}\le N^{\Cr{smallsetgumbsize}}$
and $\frac{1}{4d\rho}\ge \Cr{smallsetgumbsep}$ implies $\inf_{x,y\in F',x\ne y}d_{\infty}(x,y)\ge|F|^{\frac{1}{2d}}\ge|F'|^{\frac{1}{4d\rho}}\ge|F'|^{\Cr{smallsetgumbsep}}$.
So applying \prettyref{lem:GSPGumbelForSeparatedStartingFromPoint}
with $F'$ in place of $F$, we get that for all $|F|\ge c$, $x\in\mathbb{T}_{N}$
and $F'\in\mathcal{G}$ we have
\begin{equation}
\big|P_{x}[F'\subset Y(0,u_{F}(z)N^{d}-t(\rho))]-\exp(-e^{-(z-h)})\big|\le c|F|^{-c}.\label{eq:GApplyGSP}
\end{equation}
But it is elementary that
\begin{equation}
\big|\exp(-e^{-(z-h)})-\exp(-e^{-z})\big|\le c|h|\mbox{ for all }z,h\in\mathbb{R},\label{eq:GumbelDistFuncIsLipschitz}
\end{equation}

\n
and for the present $h$ we have $|h|\le\max(\log(1+|F|^{-\frac{1}{3}\rho}),-\log(1-|F|^{-\frac{1}{3}\rho}))\le c|F|^{-\frac{1}{3}\rho}$
(see \eqref{eq:GDefGoodSets}) provided $|F|\ge c$, so in fact $|P_{x}[F'\subset Y(0,u_{F}(z)N^{d}-t(\rho))]-\exp(-e^{-z})|\le c|F|^{-c}$ for all $F'\in\mathcal{G}$. Thus \prettyref{eq:G_ConditionOnFPrime}
implies that 
\[
|P[F\subset Y(0,u_{F}(z)N^{d}),F_{\rho}\in\mathcal{G}]-\exp(-e^{-z})P[F_{\rho}\in\mathcal{G}]|\le c|F|^{-c}.
\]

\n
Combining this with \prettyref{eq:GToOnGoodEvent} and one more application
of \prettyref{eq:G_GE_GoodEvent}, the claim \prettyref{eq:G_Gumbel}
follows for $|F|\ge c$ (recall that $\rho$ is itself a constant).
But by adjusting constants \prettyref{eq:G_Gumbel} holds for all
$F$, so the proof of \prettyref{thm:G_Gumbel} is complete.
\end{proof}

We now turn to the proof of \prettyref{lem:GSPGumbelForSeparatedStartingFromPoint}.
We will need the following lemma, which says that ``distant vertices
have almost independent entrance times''.
\begin{lem} \label{lem:RWNPointFunc}($d\ge3,$$N\ge3$) 

\medskip
Let $x_{1},...,x_{n}\in\mathbb{T}_{N}$
and let $s$ be the separation defined as in \prettyref{eq:DefOfSeparation}.
There is a constant $\Cl[c]{FarApartHitTimeIndep}$ such that if $n\le s^{\Cr{FarApartHitTimeIndep}}$
and $u\in[1,s^{\Cr{FarApartHitTimeIndep}}]$ we have 
\begin{equation}
\left(Q_{0}[0\in\mathcal{I}^{u}]\right)^{n}-cs^{-c}\le P[x_{1},...,x_{n}\in Y(0,uN^{d})]\le\left(Q_{0}[0\in\mathcal{I}^{u}]\right)^{n}+cs^{-c}.\label{eq:RWNPointFunc}
\end{equation}

\begin{proof}
We apply \prettyref{thm:Coupling} with $\varepsilon=\frac{1}{2}$
(say). We pick $\Cr{FarApartHitTimeIndep}\le\frac{\Cr{couplingerr}}{3}$, where $\Cr{couplingerr}=\Cr{couplingerr}(\frac{1}{2})$
is the constant from \prettyref{thm:Coupling}, so that that $n\le s^{\Cr{couplingerr}}$.
Thus if $\delta=\Cr{couplingerr}^{-1}s^{-\Cr{couplingerr}}$ and $s\ge c$ (so that $u\ge1\ge s^{-\Cr{couplingerr}}$)
all the conditions of \prettyref{thm:Coupling} are satisfied. By
\prettyref{eq:CouplingInclusion} we have for $s\ge c$
\begin{equation}
Q_{1}[0\in\mathcal{I}_{i}^{u(1-\delta)}\forall i]-cs^{-c}\le P[x_{1},...,x_{n}\in Y(0,uN^{d})]\le Q_{1}[0\in\mathcal{I}_{i}^{u(1+\delta)}\forall i]+cs^{-c},\label{eq:RWNPointFuncSandwich}
\end{equation}
where we also use that $cnue^{-cs^{\Cr{couplingerr}}}\le cs^{-c}$, since $u,n\le s^{c}$.
By the independence of $\mathcal{I}_{1},..,\mathcal{I}_{n},$ 
\begin{eqnarray*}
Q_{1}[0\in\mathcal{I}_{i}^{u(1+\delta)}\forall i] & \!\!\!\! = &\!\!\!\!  \big(Q_{0}[0\in\mathcal{I}^{u(1+\delta)}]\big)^{n}\\
 &\!\!\!\!  \overset{\eqref{eq:RIOneAndTwoPointFunctions}}{=} &\!\!\!\!  \left(Q_{0}[0\in\mathcal{I}^{u}]\right)^{n}\bigg(\frac{1-e^{-\frac{u(1+\delta)}{g(0)}}}{1-e^{-\frac{u}{g(0)}}}\bigg)^{n}\le(Q_{0}[0\in\mathcal{I}^{u}])^{n}(1+cs^{-c}),
\end{eqnarray*}
where we use that 
\begin{equation*}
\frac{1-e^{-\frac{u(1+\delta)}{g(0)}}}{1-e^{-\frac{u}{g(0)}}}=1+e^{-\frac{u}{g(0)}}\;\frac{1-e^{-\frac{u\delta}{g(0)}}}{1-e^{-\frac{u}{g(0)}}}\le1+c(1-e^{-\frac{u\delta}{g(0)}})\le1+cs^{-2\Cr{FarApartHitTimeIndep}}
\end{equation*}
(note $u\ge1$ and $u\delta\le cs^{\Cr{FarApartHitTimeIndep}-\Cr{couplingerr}}=cs^{-2\Cr{FarApartHitTimeIndep}}$) and $(1+cs^{-2\Cr{FarApartHitTimeIndep}})^{n}\le1+cs^{-c}$ (note $n\le s^{\Cr{FarApartHitTimeIndep}}$). Similarly
\[
Q_{1}[0\in\mathcal{I}_{i}^{u(1-\delta)}\forall i]\ge\left(Q_{0}[0\in\mathcal{I}^{u}]\right)^{n}(1-cs^{-c})\mbox{ if }s\ge c.
\]

\medskip\n
Applying these inequalities to the right- and left-hand sides of \eqref{eq:RWNPointFuncSandwich}
yields \prettyref{eq:RWNPointFunc} for $s\ge c$. But by adjusting
constants in \eqref{eq:RWNPointFunc} the same holds for all $s\ge1$.
\end{proof}
\end{lem}
We will now use \prettyref{lem:RWNPointFunc} to show \eqref{eq:GumbelForSeparatedStartingFromUniform},
which is almost like our goal \eqref{eq:GSPGumbelForSeparatedStartingFromPoint},
but has random walk starting from the uniform distribution.
\begin{lem} \label{lem:GumbelForSeparatedStartingFromUniform}($d\ge3$,$N\ge3$)

\medskip
There are constants $\Cl[c]{GumbelFromUnifSize}$ and $\Cl[c]{GumbelFromUnifSep}$,
such that for $F\subset\mathbb{T}_{N}$ satisfying $|F|\le N^{\Cr{GumbelFromUnifSize}}$ and
$d_{\infty}(x,y)\ge|F|^{\Cr{GumbelFromUnifSep}}$ for all $x,y\in F,x\ne y$, we have
\begin{equation}
\sup_{z\in\mathbb{R}}\big|P[F\subset Y(0,u_{F}(z)N^{d})]-\exp(-e^{-z})\big|\le c|F|^{-c}.\label{eq:GumbelForSeparatedStartingFromUniform}
\end{equation}
\end{lem}
\begin{proof}
The claim \prettyref{eq:GumbelForSeparatedStartingFromUniform} follows from
\begin{equation}
\sup_{z\in[-\frac{\log|F|}{4},\frac{\log|F|}{4}]}\big|P[F\subset Y(0,u_{F}(z)N^{d})]-\exp(-e^{-z})\big|\le c|F|^{-c},\label{eq:GumbelSepUniToShow}
\end{equation}

\n
since $P[F\subset Y(0,u_{F}(z))]$ is monotone in $z$, $\exp(-e^{-z})\le c|F|^{-c}$
when $z\le-\frac{\log|F|}{4}$ and $\exp(-e^{-z})\ge1-c|F|^{-c}$
when $z\ge\frac{\log|F|}{4}$. For the rest of the proof we therefore
assume that $z\in[-\frac{\log|F|}{4},\frac{\log|F|}{4}]$. 

\medskip 
First, assume also that $|F| \ge c$. Let $F=\{x_{1},...,x_{n}\}$ so that $n=|F|$ and the separation $s$
satisfies $s\ge|F|^{\Cr{GumbelFromUnifSep}}$. Also let $u=u_{F}(z)$. To be able
to apply \prettyref{lem:RWNPointFunc} we pick $\Cr{GumbelFromUnifSep}$ large enough
so that $\Cr{GumbelFromUnifSep}\Cr{FarApartHitTimeIndep}\ge1$, and thus $|F|\le|F|^{\Cr{GumbelFromUnifSep}\Cr{FarApartHitTimeIndep}}\le s^{\Cr{FarApartHitTimeIndep}}$,
which implies $n\le s^{\Cr{FarApartHitTimeIndep}}$ and $1\le g(0)\frac{3}{4}\log|F|\le u\le g(0)\frac{5}{4}\log|F|\le|F|\le s^{\Cr{FarApartHitTimeIndep}}$
since $z\in[-\frac{1}{4}\log|F|,\frac{1}{4}\log|F|]$ (recall that we assumed $|F|\ge c$).
Thus by \prettyref{eq:RWNPointFunc} 
\begin{equation}
\big|P[F\subset Y(0,u_{F}(z)N^{d})]-\big(Q_{0}[0\in\mathcal{I}^{u_{F}(z)}]\big)^{|F|}\big|\le s^{-c}\le c|F|^{-c}\mbox{ for }|F|\ge c.\label{eq:GumbelSepUniSandwich3}
\end{equation}

\medskip\n
We have $(Q_{0}[0\in\mathcal{I}^{u_{F}(z)}])^{|F|}=(1-\frac{e^{-z}}{|F|})^{|F|}$
by \eqref{eq:RIOneAndTwoPointFunctions}. But it is elementary that
\[
\exp(-e^{-z})-c|F|^{-c}\le \Big(1-\mbox{\f $\dis\frac{e^{-z}}{|F|}$}\Big)^{|F|}\le\exp(-e^{-z})\mbox{ for }z\ge-\mbox{\f $\dis\frac{1}{4}$}\log|F|\mbox{ and }|F|\ge c,
\]

\medskip\n
(since $-\frac{e^{-z}}{|F|}-c|F|^{-\frac{1}{2}}\le\log(1-\frac{e^{-z}}{|F|})\le-\frac{e^{-z}}{|F|}$
using the Taylor expansion of $\log(1-x)$ and $e^{-z}\le|F|^{\frac{1}{4}}$).
Thus \prettyref{eq:GumbelSepUniToShow} follows for $|F|\ge c$. Finally
by adjusting constants \prettyref{eq:GumbelSepUniToShow} holds for
any $F$, so the proof of the lemma is complete.
\end{proof}
Now we use \prettyref{eq:GumbelForSeparatedStartingFromUniform} to
get the desired result \eqref{eq:GSPGumbelForSeparatedStartingFromPoint}
(where the random walk can start from any vertex). To do this we roughly
speaking show that in time $|F|N^{2}$ the random walk will, with
high probability, hit only the vertex in $F$ closest to the starting
point, if it hits any vertices at all. But it will turn out that in
time $|F|N^{2}$ the random walk mixes so that what happens after
this time is governed by \prettyref{eq:GumbelForSeparatedStartingFromUniform},
and from this \prettyref{eq:GSPGumbelForSeparatedStartingFromPoint}
will follow.
\begin{proof}[Proof of \prettyref{lem:GSPGumbelForSeparatedStartingFromPoint}]
Fix $x\in\mathbb{T}_{N}$ and $z\in\mathbb{R}$. By \prettyref{eq:TorusMixing}
we can (on an extended probability space $(\Omega,\mathcal{A},Q)$) construct a coupling of
$(Y_{t})_{t\ge0}$ with law $P_{x}$ and a process $(Z_{t})_{t\ge0}$
with law $P$ such that $(Y_{N^{2}|F|+t})_{t\ge0}$ coincides with
$(Z_{t})_{t\ge0}$ with probability at least $1-ce^{-c|F|}$. Then
if $z_{-}=z-\frac{1}{|F|}$ (so that $u_{F}(z_{-})N^{d}\overset{\eqref{eq:DefOfuF},g(0)\ge1}{\le}u_F(z)N^{d}-\frac{N^{d}}{|F|}\overset{|F|\le N^{\Cr{smallsetgumbsize}}}{\le}u_{F}(z)N^{d}-|F|N^{2}$,
provided $\Cr{smallsetgumbsize}$ is chosen small enough) we have 
\begin{equation}
\begin{split}
Q[F\subset Z(0,u_{F}(z_{-})N^{d})] \; - & \;ce^{-c|F|}\le Q[F\subset Y(0,u_{F}(z)N^{d})]\\
\;  \le &\; Q[F\subset Y(0,|F|N^{2})\cup Z(0,u_{F}(z)N^{d})]+ce^{-c|F|},
\end{split}\label{eq:GSPSandwich}
\end{equation}
Now (possibly making $\Cr{smallsetgumbsize}$ smaller and $\Cr{smallsetgumbsep}$ larger) we see
that \prettyref{lem:GumbelForSeparatedStartingFromUniform} applies
to the left-hand side of \prettyref{eq:GSPSandwich} (recall $Z_{\cdot}$
has law $P$) and we get $\exp(-e^{-z_{-}})-c|F|^{-c}\le Q[F\subset Y(0,u_{F}(z)N^{d})]$,
which together with \prettyref{eq:GumbelDistFuncIsLipschitz} and $h=\frac{1}{|F|}$ implies 
\begin{equation}
\exp(-e^{-z})-c|F|^{-c}\le Q[F\subset Y(0,u_{F}(z)N^{d})] = P_{x}[F\subset Y(0,u_{F}(z)N^{d})].\label{eq:GSPLowerBound}
\end{equation}

\medskip\n
It remains to bound the right-hand side of \prettyref{eq:GSPSandwich}
from above. Let $y$ denote a vertex of $F$ of minimal distance from
$x$ and let $F'=F\backslash\{y\}$. We then have 
\begin{align*}
Q[F\subset Y(0,|F|N^{2})\cup Z(0,u_{F}(z)N^{d})]  \le & \;Q[F'\subset Y(0,|F|N^{2})\cup Z(0,u_{F}(z)N^{d})],\\
  \le &\; P_{x}[H_{F'}<|F|N^{2}]+Q[F'\subset Z(0,u_{F}(z)N^{d})].
\end{align*}
Now $P_{x}[H_{F'}<|F|N^{2}]\le\sum_{v\in F'}P_{x}[H_{v}<N^{2+\frac{1}{10}}]$
(possibly decreasing $\Cr{smallsetgumbsize}$ so that $|F|\le N^{\frac{1}{10}}$).
Now by our assumption on $F$ and choice of $y$ we have $d(x,v)\ge\frac{1}{2}|F|^{\Cr{smallsetgumbsep}}$
for all $v\in F'$. Therefore using \prettyref{eq:BHBBallHittingBound}
with $\lambda=\frac{1}{10}$, $r_{1}=1$ and $r_{2}=\frac{1}{2}|F|^{\Cr{smallsetgumbsep}}$
(possibly decreasing $\Cr{smallsetgumbsize}$ even more so that $r_{2}\le|F|^{\Cr{smallsetgumbsep}}\le N^{\Cr{smallsetgumbsize}\Cr{smallsetgumbsep}}\le N^{1-3\lambda}$)
we get $P_{x}[H_{v}<N^{2+\frac{1}{10}}]\le c|F|^{-\Cr{smallsetgumbsep}}$ for all $v\in F'$ and thus $P_{x}[H_{F'}<|F|N^{2}]\le c|F|^{1-\Cr{smallsetgumbsep}}\le c|F|^{-c}$, since we may increase $\Cr{smallsetgumbsep}$ so that $\Cr{smallsetgumbsep}>1$. Letting $z_{+}=z+\frac{2}{|F|}$ we have 
\begin{equation*}
u_{F}(z)\overset{\prettyref{eq:DefOfuF}}{=}u_{F'}\Big(z+\log\mbox{\f $\dis\frac{|F|}{|F'|}$}\Big)\overset{|F'|=|F|-1,|F|\ge2}{\le}u_{F'}(z_{+}),
\end{equation*}
so that
\begin{equation}\label{eq:GSPSufficesToShow}
\begin{split}
Q[F\subset Y(0,|F|N^{2})\cup Z(0,u_{F}(z)N^{d})]  \le & \;Q[F'\subset Z(0,u_{F'}(z_{+})N^{d})]+c|F|^{-c}
\\
  = & \; P[F'\subset Y(0,u_{F'}(z_{+})N^{d})]+c|F|^{-c}
\\
  \le & \; \exp(-e^{-z})+c|F|^{-c},
\end{split}\end{equation}
where the last inequality follows from \prettyref{lem:GumbelForSeparatedStartingFromUniform} and \prettyref{eq:GumbelDistFuncIsLipschitz} (similarly to above
\eqref{eq:GSPLowerBound}). Together with \prettyref{eq:GSPLowerBound}
and \prettyref{eq:GSPSandwich} this implies \prettyref{eq:GSPGumbelForSeparatedStartingFromPoint}.
\end{proof}
It still remains to prove \prettyref{lem:G_GE_GoodEvent}, the other
ingredient in the proof of \eqref{thm:G_Gumbel}. For the proof we
will use the following bounds on the probability of not hitting two
points, which is a consequence of the coupling.
\begin{lem}\label{lem:RWTwoPoint}($d\ge3,$$N\ge3$) 

\medskip
There exists a constant $\Cl[c]{TPerr}$ such that for all $x,y\in\mathbb{T}_{N}$ we have, letting $v=d_{\infty}(x,y)$,
\begin{align}
 & P[x,y\notin Y(0,uN^{d})]\le(1+cv^{-\Cr{TPerr}})(Q_{0}[0\notin\mathcal{I}^{u}])^{2}\mbox{ if }u\in[1,v^{\Cr{TPerr}}],\label{eq:RWTPFuncUpperBoundFar}\\
 & P[x,y\notin Y(0,uN^{d})]\le(1+cN^{-\Cr{TPerr}})Q_{0}[0,x-y\notin\mathcal{I}^{u}]\mbox{ if }v\le N^{\frac{1}{2}},u\in[1,N^{\Cr{TPerr}}].\label{eq:RWTPFuncUpperBoundClose}
\end{align}
\end{lem}
\begin{proof}
We start with \prettyref{eq:RWTPFuncUpperBoundFar}. Let $n=2,x_{1}=x,x_{2}=y$
(so that the separation $s$ is $v$) and $\varepsilon=\frac{1}{2}$
(say). We have $u\ge s^{-\Cr{couplingerr}}$ and thus letting $\delta=\Cr{couplingerr}^{-1}s^{-\Cr{couplingerr}}$
it follows from \prettyref{thm:Coupling} (see \prettyref{eq:CouplingInclusion})
that, choosing $\Cr{TPerr}<\frac{1}{2}\Cr{couplingerr}$ so that $u\delta\le cs^{\Cr{TPerr}}s^{-2\Cr{TPerr}}=cs^{-\Cr{TPerr}}$,
\begin{equation}
\begin{array}{ccl}
P[x,y\notin Y(0,uN^{d})] &\!\!\!\! \le &\!\!\!\! \left(Q_{0}[0\notin\mathcal{I}^{u(1-\delta)}]\right)^{2}+cue^{-cs^{\Cr{couplingerr}}}
\\[1ex]
 &\!\!\!\! \overset{\eqref{eq:RIOneAndTwoPointFunctions}}{=} &\!\!\!\! \left(Q_{0}[0\notin\mathcal{I}^{u}]\right)^{2}e^{2\frac{u\delta}{g(0)}}+cue^{-cs^{\Cr{couplingerr}}}
 \\[1ex]
 &\!\!\!\! \le & \!\!\!\!\left(Q_{0}[0\notin\mathcal{I}^{u}]\right)^{2}(1+cs^{-\Cr{TPerr}})+cue^{-cs^{\Cr{couplingerr}}}.
\end{array}\label{eq:RWTPFuncPart1}
\end{equation}
But if $\Cr{TPerr}$ is chosen small enough $cue^{-cs^{\Cr{couplingerr}}}\le ce^{-cs^{\Cr{couplingerr}}}\le cs^{-\Cr{TPerr}}e^{-cs^{c}}\le cs^{-\Cr{TPerr}}e^{-\frac{2u}{g(0)}}=cs^{-\Cr{TPerr}}(Q_{0}[0\notin\mathcal{I}^{u}])^{2}$,
so \prettyref{eq:RWTPFuncUpperBoundFar} follows.

To prove \prettyref{eq:RWTPFuncUpperBoundClose} we let $x_{1}=x$,
$n=1$ (so that the separation $s$ is $N$). We further let $\varepsilon=\frac{1}{2}$
so that the box $A+x_{1}=B(x_{1},s^{1-\varepsilon})$ contains $y$,
and note that $u\ge s^{-\Cr{couplingerr}}=N^{-\Cr{couplingerr}}$. Thus letting $\delta=\Cr{couplingerr}^{-1}s^{-\Cr{couplingerr}}=\Cr{couplingerr}^{-1}N^{-\Cr{couplingerr}}$
it follows from \prettyref{thm:Coupling} that
\[
P[x,y\notin Y(0,uN^{d})]\le Q_{0}[0,y-x\notin\mathcal{I}^{u(1-\delta)}]+cue^{-cN^{\Cr{couplingerr}}}.
\]

\medskip\n
Now similarly to above we find that the right-hand side is bounded
above by $(1+cN^{-\Cr{TPerr}})Q_{0}[0,y-x\notin\mathcal{I}^{u}]$ (provided
$\Cr{TPerr}$ is chosen small enough), so \prettyref{eq:RWTPFuncUpperBoundClose}
follows.
\end{proof}
We are now in a position to prove \prettyref{lem:G_GE_GoodEvent}.
We will show that $E[|F_{\rho}|]$ is close to $|F|^{\rho}$, so that
proving that the probability of $F_{\rho}\notin\mathcal{G}$ is small
reduces (via Chebyshev's inequality) to bounding the variance of $\mbox{Var}[|F_{\rho}|]$
from above and bounding the probability that  $\inf_{x,y\in F_{\rho},x\ne y}$ \, $d_{\infty}(x,y)$
is small from above (recall \prettyref{eq:GDefGoodSets} and \prettyref{eq:G_GE_GoodEvent}).
But both $\mbox{Var}[|F_{\rho}|]$ and the probability that $\inf_{x,y\in F_{\rho},x\ne y}d_{\infty}(x,y)$
is small can be bounded above in terms of sums, over pairs $x,y$
of vertices, of the probability $P[x,y\notin Y(0,t(\rho))]$, and
these sums can be controlled by \prettyref{lem:TwoPointCalculation},
via \prettyref{eq:RWTPFuncUpperBoundFar} and \prettyref{eq:RWTPFuncUpperBoundClose}.
\begin{proof}[Proof of \prettyref{lem:G_GE_GoodEvent}]
Let 
\begin{equation}
u=g(0)(1-\rho)\log|F|\mbox{ so that }t(\rho)=uN^{d},\label{eq:G_GE_DefofU}
\end{equation}
and record for the sequel that
\begin{equation}
Q_{0}[x\notin\mathcal{I}^{u}]\overset{\eqref{eq:RIOneAndTwoPointFunctions},\prettyref{eq:G_GE_DefofU}}{=}|F|^{\rho-1}\mbox{ for all }x\in\mathbb{Z}^{d}.\label{eq:GE_OnePointProb}
\end{equation}
By summing over $x\in F$ in \prettyref{eq:RWOPFuncMultError} (note
that $|F|\le N^{d}$ so that we have 
\begin{equation}
1\overset{|F|\ge c(\rho)}{\le}u\le c\log N\overset{N\ge|F|^{1/d}\ge c}{\le}N^{\Cr{OPu}}\label{eq:ustuff}
\end{equation}

\medskip\n
by \eqref{eq:G_GE_DefofU}) and using $|F|Q_{0}[0\notin\mathcal{I}^{u}]\overset{\eqref{eq:GE_OnePointProb}}{=}|F|^{\rho}$
we get 
\begin{equation}
(1-cN^{-c})|F|^{\rho}\le E[|F_{\rho}|]\le(1+cN^{-c})|F|^{\rho}.\label{eq:ExpectationofFRho}
\end{equation}

\medskip\n
Therefore $\left||F_{\rho}|-|F|^{\rho}\right|>|F|^{\frac{2}{3}\rho}$ implies 
\begin{equation*}
||F_{\rho}|-E[|F_{\rho}|]|>|F|^{\frac{2}{3}\rho}-cN^{-c}|F|^{\rho}\overset{|F|\le N^{d}}{\ge}|F|^{\frac{2}{3}\rho}-c|F|^{\rho-c}\overset{\rho\le c,|F|\ge c}{\ge}\mbox{\f $\dis\frac{|F|^{\frac{2}{3}\rho}}{2}$}.
\end{equation*}
Thus the Chebyshev inequality gives
\[
P\big[\bigl||F_{\rho}|-|F|^{\rho}]\big|>|F|^{\frac{2}{3}\rho}\big]\le P\big[\bigl||F_{\rho}|-E[|F_{\rho}|]\big|>\mbox{\f $\dis\frac{1}{2}$} |F|^{\frac{2}{3}\rho}\big]\le4\mbox{\f $\dis\frac{\mbox{Var}[|F_{\rho}|]}{|F|^{\frac{4}{3}\rho}}$}.
\]
Note that $\mbox{Var}[|F_{\rho}|]=\sum_{x,y\in F}q_{x,y}\le E[|F_{\rho}|]+\sum_{x\ne y\in F}q_{x,y}$,
where $q_{x,y}=P[x,y\notin Y(0,uN^{d})]-P[x\notin Y(0,uN^{d})]P[y\notin Y(0,uN^{d})]$.
Therefore using \prettyref{eq:ExpectationofFRho}, and splitting the
sum between ``far and close pairs of vertices'', we get 
\begin{equation}
P\big[\bigl||F_{\rho}|-|F^{\rho}|\big|>|F|^{\frac{2}{3}\rho}\big]\le c|F|^{-\frac{1}{3}\rho}+c\dsl_{\{x,y\}\in V}P[x,y\notin Y(0,uN^{d})]+c\dsl_{\{x,y\}\in W}q_{x,y},\label{eq:ResOfVarianceBound}
\end{equation}
where $V=\{\{x,y\}\subset F:0<d_{\infty}(x,y)\le|F|^{\frac{1}{2d}}\}$
and $W=\{\{x,y\}\subset F:d_{\infty}(x,y)>|F|^{\frac{1}{2d}}\}$.
Furthermore note that

\begin{equation}
P[\inf_{x,y\in F_{\rho},x\ne y}d_{\infty}(x,y)<|F|^{\frac{1}{2d}}]\le\dsl_{\{x,y\}\in V}P[x,y\notin Y(0,uN^{d})],\label{eq:G_GE_Part1UnionBound}
\end{equation}
and thus by \prettyref{eq:GDefGoodSets}
\begin{equation}
P[F_{\rho}\notin\mathcal{G}]\le c|F|^{-c\rho}+c\dsl_{\{x,y\}\in V}P[x,y\notin Y(0,uN^{d})]+c\dsl_{\{x,y\}\in W}q_{x,y}.\label{eq:EnGrej}
\end{equation}

\medskip\n
We seek to bound the sums $\sum_{\{x,y\}\in V}P[x,y\notin Y(0,uN^{d})]$
and $\sum_{\{x,y\}\in W}q_{x,y}$.
To this end note that if $\{x,y\} \in V$ then by \prettyref{eq:RWTPFuncUpperBoundClose} we have $P[x,y\notin Y(0,uN^{d})]  \le \; cQ_{0}[0,y-x\notin\mathcal{I}^{u}]$,
where \prettyref{eq:RWTPFuncUpperBoundClose} applies because $d_{\infty}(x,y)\le|F|^{1/2d}\le N^{1/2}$ (note $|F| \le N^d$), and $1 \le u \le N^{\Cr{TPerr}}$
(cf. \prettyref{eq:ustuff}). Thus
\begin{equation}
\begin{split}
\dsl_{\{x,y\}\in V}P[x,y\notin Y(0,uN^{d})]  \le &\; c\dsl_{\{x,y\}\in V}Q_{0}[0,y-x\notin\mathcal{I}^{u}]\\
  \le & \; c\dsl_{x\in F}\sum_{y\in K_{x}}Q_{0}[0,y-x\notin\mathcal{I}^{u}],
\end{split}\label{eq:G_GE_Part1UnionBound2}
\end{equation}
where $K_{x}=F\cap B(x,|F|^{\frac{1}{2d}})\backslash\{x\}$. Now using
\prettyref{eq:TwoPointCalculation} on the inner sum of the right-hand
side with $K_{x}-x$ in place of $K$, we get
\begin{equation}
\begin{array}{ccl}
\underset{\{x,y\}\in V}{\sum}P[x,y\notin Y(0,uN^{d})] &\!\!\!\!\! \overset{u\ge1}{\le} &\!\!\!\!\!  c\underset{x\in F}{\sum}\big\{ |K_{x}|Q_{0}[0\notin\mathcal{I}^{u}]{}^{2}u+e^{-\Cr{TPC}\frac{u}{g(0)}}\big\} \\
 &\!\!\!\!\! \overset{\eqref{eq:G_GE_DefofU}}{\le} &\!\!\!\!\! c\underset{x\in F}{\sum}\big\{ |K_{x}|Q_{0}[0\notin\mathcal{I}^{u}]{}^{2}\log|F|+|F|^{-\Cr{TPC}(1-\rho)}\big\} \\
 &\!\!\!\!\! \le &\!\!\!\!\!  c|F|^{\frac{3}{2}}Q_{0}[0\notin\mathcal{I}^{u}]{}^{2}\log|F|+c|F|^{-c},
\end{array}\label{eq:G_GE_ApplicationOfTwoPointCalc}
\end{equation}

\medskip\n
where we have used that $\sum_{x\in F}|K_{x}|\le\sum_{x\in F}c(|F|^{\frac{1}{2d}})^{d}=c|F|^{\frac{3}{2}}$,
and choose $\Cr{rhobound}$ small enough so that $\Cr{TPC}(1-\Cr{rhobound})>1$ (recall that $\Cr{TPC}>1$).
Thus from \prettyref{eq:GE_OnePointProb} we have
\begin{equation}
\underset{\{x,y\}\in V}{\dsl}P[x,y\notin Y(0,uN^{d})]\le c|F|^{2\rho-\frac{1}{2}}\log|F|+c|F|^{-c}\overset{\rho\le c}{\le}c|F|^{-c}.\label{eq:CPSZ}
\end{equation}

\n
We now turn to the sum $\sum_{\{x,y\}\in W}q_{x,y}$. Using \prettyref{eq:RWOPFuncMultError}
again we obtain 
\begin{equation}
P[x\notin Y(0,uN^{d})]P[y\notin Y(0,uN^{d})]\ge(1-cN^{-c})Q_{0}[0\notin\mathcal{I}^{u}]^{2}.\label{eq:GE_ProductOfOPProbsLowerBound}
\end{equation}

\medskip\n
Also by \prettyref{eq:RWTPFuncUpperBoundFar} we have that if $x$
and $y$ are such that $v=d_{\infty}(x,y)\ge|F|^{3\Cr{TPerr}^{-1}\rho}$
then (similarly to \prettyref{eq:ustuff} we have
$1\le u \le g(0)\log|F| \le |F|^{3\rho} \le v^{\Cr{TPerr}}$, see \eqref{eq:G_GE_DefofU} and note $|F|\ge c(\rho)$)

\begin{equation}
P[x,y\notin Y(0,uN^{d})]\le(1+c|F|^{-3\rho})Q_{0}[0\notin\mathcal{I}^{u}]^{2}.\label{eq:MOSZ}
\end{equation}
Combining \eqref{eq:GE_ProductOfOPProbsLowerBound} and \eqref{eq:MOSZ}
we have
\begin{equation}
\begin{array}{rc}
\underset{x,y\in F,d_{\infty}(x,y)\ge|F|^{3\Cr{TPerr}^{-1}\rho}}{\sum}q_{x,y}\le c(|F|^{-3\rho}+cN^{-c})\underset{x,y\in F}{\dsl}Q_{0}[0\notin\mathcal{I}^{u}]^{2}\le c|F|^{-\rho},\end{array}\label{eq:GE_SumFarPoints}
\end{equation}
since (possibly decreasing $\Cr{rhobound}$) 
\begin{equation*}
N^{-c}\overset{|F|\le N^{d}}{\le}|F|^{-c}\le|F|^{-3\Cr{rhobound}}\le|F|^{-3\rho} \;
\mbox{and} \;  \dsl_{x,y\in F}Q_{0}[0\notin\mathcal{I}^{u}]^{2}\overset{\eqref{eq:GE_OnePointProb}}{=}|F|^{2\rho}.
\end{equation*}
Possibly drecasing $\Cr{rhobound}$ once again, we have that all $ 0 < \rho \le \Cr{rhobound}$ satisfy $3\Cr{TPerr}^{-1}\rho \le \frac{1}{2d}$. Then $|F|^{3\Cr{TPerr}^{-1}\rho}\le|F|^{\frac{1}{2d}}$
so that from the definition of $W$ and \prettyref{eq:GE_SumFarPoints}
\begin{equation}
\underset{\{x,y\}\in W}{\dsl}q_{x,y}\le c|F|^{-\rho}.\label{eq:qxyUpperBound}
\end{equation}

\medskip\n
Now using \prettyref{eq:CPSZ} and \prettyref{eq:qxyUpperBound} on
the right-hand side of \prettyref{eq:EnGrej} gives \prettyref{eq:G_GE_GoodEvent}.\end{proof}

We have now completely reduced the proof of \prettyref{thm:G_Gumbel} to the coupling
result \prettyref{thm:Coupling}. We end this section with a remark on the use of the
coupling \prettyref{thm:Coupling} as a general tool, and on the possibility of
extending \prettyref{thm:G_Gumbel} to other families of graphs.

\begin{rem}
\label{rem:EndOfSec3Remarks}
(1) As mentioned in the introduction, a coupling like \prettyref{thm:Coupling} is a very powerful tool to
study the trace of random walk. To prove the cover time result \prettyref{thm:G_Gumbel} we used the
coupling to study certain properties of the trace of the random walk; namely the probabilities that points,
pairs of points, and sets of ``distant'' points are contained in the trace 
(see \prettyref{lem:RWOnePoint}, \prettyref{lem:RWTwoPoint} and \prettyref{lem:RWNPointFunc} respectively).
When studying the percolative properties of the so called vacant set (the complement of the trace of random walk), similar couplings have been used,
and there the properties of the trace studied are certain connnectivity properties of its complement (see e.g. (2.4), the display after (2.14) or (2.20)-(2.22)
in \cite{TeixeiraWindischOnTheFrag}, or (7.18) in \cite{Sznitman2009-OnDOMofRWonDCbyRI}). The generality
of the couping \prettyref{thm:Coupling} ensures that it can be used in the future to study further
unrelated properties of the trace of random walk in the torus.

\medskip\n
(2) The method used in this section to prove Gumbel fluctuations essentially
consists of considering the set of ``late points'' (recall \eqref{eq:G_DefOfUncoveredSet})
and proving that it concentrates and is separated (i.e. \eqref{lem:G_GE_GoodEvent}).
It has already been used to prove Gumbel fluctuations in related models
in \cite{Belius2010} and \cite{BeliusCTinDC}, and could potentially
apply to prove Gumbel fluctuations for many families of graphs, as
long as one can obtain good enough control of entrance times to replace
\eqref{eq:RWOPFuncMultError}, \eqref{lem:RWTwoPoint} and \eqref{eq:RWNPointFunc}
(in a more general context the latter estimate may be difficult to
obtain but could be replaced with an estimate on how close $H_{\{x_{1},x_{2},...,x_{m}\}}$
is to being exponential when $x_{1},...,x_{m}$ are ``separated'',
since the cover time of a set $\{x_{1},...,x_{m}\}$ consisting of
separated points is essentially the sum of $m$ entrance times for
sets consisting of $m,m-1,m-2,...$ and finally $1$ points; from
this one can derive something similar to \prettyref{eq:GSPGumbelForSeparatedStartingFromPoint}).
In a forthcoming work Roberto Oliveira and Alan Paula obtain such a
generalization of \prettyref{thm:G_Gumbel}.\qed
\end{rem}

\section{\label{sec:Coupling}Coupling}

We now turn to the proof of the coupling result \prettyref{thm:Coupling}.
The proof has three main steps: first the trace of random walk in
the union of the boxes $A+x_{i},i=1,...,n,$ (recall \prettyref{eq:DefOfSeparation}
and \prettyref{eq:DefOfA}) is coupled with a certain Poisson process
on the space of trajectories $\Gamma(\mathbb{T}_{N})$ (see below
\prettyref{eq:DefOfBoundary}). From this Poisson point process we
then construct $n$ independent Poisson point processes, one for each
box $A+x_{i}$, which are coupled with the trace of random walk in
the corresponding box. Lastly we construct from each of these Poisson
processes a random interlacement which is coupled with the trace of
the random walk in the corresponding box $A+x_{i}$. Essentially speaking
the three steps are contained in the three propositions \ref{pro:CoupleRWandPPoE},
\ref{pro:CouplePPoEandUFreePPoE} and \ref{pro:CouplePPoEandRI},
which we state in this section and use to prove \prettyref{thm:Coupling}.
The proofs of the propositions are postponed until the subsequent
sections.

For the rest of the paper we assume that we are given centres of boxes
\begin{equation}
x_{1},..,x_{n}\in\mathbb{T}_{N}\mbox{ whose separation is }s\mbox{ (see \eqref{eq:DefOfSeparation}), and }\varepsilon\in(0,1).\label{eq:StandingSeparationAssumption}
\end{equation}
We also define the concentric boxes $B\subset C$ around $A$ by 
\begin{equation}
A\overset{\eqref{eq:DefOfA}}{=}B(0,s^{1-\varepsilon})\subset B=B(0,s^{1-\frac{\varepsilon}{2}})\subset C=B(0,s^{1-\frac{\varepsilon}{4}}).\label{eq:ABCNotation}
\end{equation}
For convenience we introduce the notation 
\begin{equation}
\bar{F}=\mbox{\f $\dis\bigcup\limits_{i=1}^{n}$} F_{i}\mbox{ where }F_{i}=F+x_{i}\mbox{ for any }F\subset\mathbb{T}_{N}.\label{eq:SetUnionTranslateNotation}
\end{equation}
Note that if $s\ge c(\varepsilon)$ then the $C_{i}$ are disjoint.
To state \prettyref{pro:CoupleRWandPPoE} we introduce $U$, the first
time random walk spends a ``long'' time outside of $\bar{C}$ (roughly
speaking long enough time to mix, see \prettyref{pro:QSConvToQuasiStat}), defined by 
\begin{equation}
U=\inf\{t\ge t^{\star}:Y(t-t^{\star},t)\cap\bar{C}=\emptyset\}\mbox{ where }t^{\star}=N^{2+\frac{\varepsilon}{100}}.\label{eq:DefOfUandtStar}
\end{equation}
We also define the intensity measure $\kappa_{1}$ on $\Gamma(\mathbb{T}_{N})$
by (recall \prettyref{eq:DefOfeKandCap})
\begin{equation}
\kappa_{1}(dw)=P_{e}[Y_{\cdot\wedge U}\in dw]\mbox{ where }e(x)=\dsl_{i=1}^{n}e_{A}(x-x_{i}).\label{eq:DefOfKappa1}
\end{equation}
For parameters $u\ge0$ and $\delta>0$ (satisfying suitable conditions),
\prettyref{pro:CoupleRWandPPoE} constructs a coupling of $Y_{\cdot}$
with two independent Poisson point processes $\mu_{1}$ and $\mu_{2}$
on $\Gamma(\mathbb{T}_{N})$ of intensities $u(1-\delta)\kappa_{1}$
and $2u\delta\kappa_{1}$ respectively such that (recall the notation
from \prettyref{eq:DefOfPPPTrace}) 
\begin{equation}
\{\mathcal{I}(\mu_{1})\cap\bar{A}\subset Y(0,uN^{d})\cap\bar{A}\subset\mathcal{I}(\mu_{1}+\mu_{2})\cap\bar{A}\}\mbox{ with high probability.}\label{eq:CouplingInformalFirstInclusion}
\end{equation}

\medskip\n
\prettyref{pro:CouplePPoEandUFreePPoE}, the second ingredient in
the proof of \prettyref{thm:Coupling}, couples Poisson processes
like $\mu_{1}$ and $\mu_{2}$ with Poisson processes with intensity
a multiple of
\begin{equation}
\kappa_{2}(dw)=P_{e}[Y_{\cdot\wedge T_{\bar{B}}}\in dw].\label{eq:DefOfKappa2-1}
\end{equation}

\medskip\n
More precisely if $\nu$ is a Poisson process of intensity $u\kappa_{1},u>0,$
and $\delta>0$ then (under appropriate conditions) \prettyref{pro:CouplePPoEandUFreePPoE}
will construct Poisson point processes $\nu_{1}$ and $\nu_{2}$ of
intensities $u(1-\delta)\kappa_{2}$ and $2u\delta\kappa_{2}$ respectively
such that
\begin{align}
 & \{\mathcal{I}(\nu_{1})\cap\bar{A}\subset\mathcal{I}(\nu)\cap\bar{A}\}\mbox{ almost surely and}\label{eq:SecondCouplingASInc}\\
 & \{\mathcal{I}(\nu)\cap\bar{A}\subset\mathcal{I}(\nu_{1}+\nu_{2})\cap\bar{A}\}\mbox{ with high probability.}\label{eq:SecondInclusionHighProbInc}
\end{align}

\medskip\n
Note that, in contrast to the situation for $\mu_{1}$ and $\mu_{2}$
from \prettyref{eq:CouplingInformalFirstInclusion}, each ``excursion''
in the support of $\nu_{1}$ and $\nu_{2}$ never returns to $\bar{A}$
after it has left $\bar{B}$. Under the law induced on it from the
intensity measure $\kappa_{2}$, an excursion therefore, conditionally
on its starting point, has the law of a random walk in $\mathbb{Z}^{d}$
stopped upon leaving $B$ (up to translation). Furthermore it leaves
a trace in only one of the boxes $A_{1},...,A_{n}$. This will allow
us (in \prettyref{cor:CouplePPoEandIndepPPoEs}) to ``split'' the
Poisson point processes $\nu_{1}$ and $\nu_{2}$ into independent
Poisson point processes $\nu_{1}^{i},\nu_{2}^{i},i=1,...,n,$ (on
$\Gamma(\mathbb{Z}^{d})$) such that the $\nu_{1}^{i}$ have intensity
$u(1-\delta)\kappa_{3}$ and the $\nu_{2}^{i}$ have intensity $2u\delta\kappa_{3}$,
where
\begin{align}
 & \kappa_{3}(dw)=P_{e_{A}}^{\mathbb{Z}^{d}}[Y_{\cdot\wedge T_{B}}\in dw],\mbox{ and such that}\label{eq:DefOfKappa3}\\
 & \{\mathcal{I}(\nu_{1}^{i})\cap A\subset(\mathcal{I}(\nu)-x_{i})\cap A\mbox{ for all }i\}\mbox{ almost surely and}\label{eq:CompPPPropCorLowerInclusion-1}\\
 & \{(\mathcal{I}(\nu)-x_{i})\cap A\subset\mathcal{I}(\nu_{1}^{i}+\nu_{2}^{i})\cap A\mbox{ for all }i\}\mbox{ with high probability.}\label{eq:SecondInclusionCorollHighProbInc}
\end{align}

\medskip\n
\prettyref{pro:CouplePPoEandRI}, the third ingredient in the proof
of \prettyref{thm:Coupling}, constructs independent random subsets
of $\mathbb{Z}^{d}$ that have the law of random interlacements intersected
with $A$, from Poisson processes like $\nu_{j}^{i}$. More precisely
if $\eta$ is a Poisson point process of intensity $u\kappa_{3},u\ge0,$
and $\delta>0$ then under appropriate conditions it constructs independent
random sets $\mathcal{I}_{1},\mathcal{I}_{2}\subset\mathbb{Z}^{d}$
such that $\mathcal{I}_{1}$ has the law of $\mathcal{I}^{u(1-\delta)}\cap A$
under $Q_{0}$, $\mathcal{I}_{2}$ has the law of $\mathcal{I}^{2\delta u}\cap A$
under $Q_{0}$, and
\begin{equation}
\{\mathcal{I}_{1}\cap A\subset\mathcal{I}(\eta)\cap A\subset(\mathcal{I}_{1}\cup\mathcal{I}_{2})\cap A\}\mbox{ with high probability.}\label{eq:ThirdCouplingInclusion}
\end{equation}

\medskip\n
But essentially speaking because of \prettyref{eq:AdditiveInterlacements}
we will be able to easily construct a random interlacement $(\mathcal{I}^{u})_{u\ge0}$
from such a pair $\mathcal{I}_{1},\mathcal{I}_{2}$.

We now state the propositions. Recall the standing assumption \prettyref{eq:StandingSeparationAssumption}.
\begin{prop}
\label{pro:CoupleRWandPPoE}($d\ge3$,$N\ge3$,$x_{1},...,x_{n}\in\mathbb{T}_{N}$)

\medskip
If $s\ge c(\varepsilon)$, $u\ge s^{-c(\varepsilon)}$, $\frac{1}{2}\ge\delta\ge cs^{-c(\varepsilon)}$
and $n\le s^{c(\varepsilon)}$ we can construct a coupling $(\Omega_{2},\mathcal{A}_{2},Q_{2})$
of the random walk $Y_{\cdot}$ with law $P$ and independent Poisson
point processes $\mu_{1}$ and $\mu_{2}$ on $\Gamma(\mathbb{T}_{N})$,
such that $\mu_{1}$ has intensity $u(1-\delta)\kappa_{1}$, $\mu_{2}$
has intensity $2u\delta\kappa_{1}$, and $Q_{2}[I_{1}]\ge1-cue^{-cs^{c(\varepsilon)}}$,
where $I_{1}$ is the event in \eqref{eq:CouplingInformalFirstInclusion}.
\end{prop} 

\prettyref{pro:CoupleRWandPPoE} will be proved in \prettyref{sec:Poissonization}. 
\begin{prop} \label{pro:CouplePPoEandUFreePPoE}($d\ge3$,$N\ge3$,$x_{1},...,x_{n}\in\mathbb{T}_{N}$)

\medskip
Assume $s\ge c(\varepsilon)$ and that $\nu$ is a Poisson point process
on $\Gamma(\mathbb{T}_{N})$ with intensity measure $u\kappa_{1}$,
$u\ge s^{-c(\varepsilon)}$, constructed on some probability space
$(\Omega,\mathcal{A},Q)$. Then if $1\ge\delta\ge cs^{-c(\varepsilon)}$
and $n\le s^{c(\varepsilon)}$ we can extend the space to get independent
Poisson point processes $\nu_{1},\nu_{2},$ on $\Gamma(\mathbb{T}_{N})$
such that $\nu_{1}$ has intensity $u(1-\delta)\kappa_{2}$, $\nu_{2}$
has intensity $2u\delta\kappa_{2}$, \eqref{eq:SecondCouplingASInc}
holds and $Q[I_{2}]\ge1-ce^{-cs^{c(\varepsilon)}}$, where $I_{2}$
is the event in \eqref{eq:SecondInclusionHighProbInc}.
\end{prop}
The proof of \prettyref{pro:CouplePPoEandUFreePPoE} is contained
in \prettyref{sec:OutOfTorus}. In the proof of \prettyref{thm:Coupling}
we will actually use the following corollary.
\begin{cor}
\label{cor:CouplePPoEandIndepPPoEs}Under the conditions of \prettyref{pro:CouplePPoEandUFreePPoE}
we can construct independent Poisson point processes $\nu_{1}^{i},\nu_{2}^{i},i=1,2,...,n,$
such that $\nu_{1}^{i}$ has intensity $u(1-\delta)\kappa_{3}$ and
$\nu_{2}^{i}$ has intensity $2u\delta\kappa_{3}$ for $i=1,...,n,$
\prettyref{eq:CompPPPropCorLowerInclusion-1} holds and $Q[I_{3}]\ge1-ce^{-cs^{c(\varepsilon)}}$,
where $I_{3}$ is the event in \eqref{eq:SecondInclusionCorollHighProbInc}.\end{cor}
\begin{proof}
For $i=1,...,n,$ and $j=1,2,$ let $\nu_{j}^{i}$ be the image of
$1_{\{Y_{0}\in A_{i}\}}\nu_{j}$ under the map which sends $w(\cdot)\in\Gamma(B_{i})\subset\Gamma(\mathbb{T}_{N})$
to $w(\cdot)-x_{i}\in\Gamma(\mathbb{Z}^{d})$ (recall that $\Gamma(B_{i})$
for $B_{i}\subset\mathbb{T}_{N}$ denotes the set of paths in $\mathbb{T}_{N}$
that never leave $B_{i}$, and note that $B=B_{i}-x_{i}\subset\mathbb{T}_{N}$
may be identified with a subset of $\mathbb{Z}^{d}$, so that $w(\cdot)-x_{i}$
can be identified with an element of $\Gamma(\mathbb{Z}^{d}$)). Since
the sets $\{Y_{0}\in A_{i}\},i=1,...,n,$ are disjoint we have that
$\nu_{i}^{1},\nu_{i}^{2},i=1,...,n,$ are independent Poisson point
processes of the required intensities (see the \prettyref{eq:DefOfKappa1},
\prettyref{eq:DefOfKappa2-1} and \prettyref{eq:DefOfKappa3}). Now
\eqref{eq:CompPPPropCorLowerInclusion-1} and the required bound on
$Q[I_{3}]$ follows from \prettyref{pro:CouplePPoEandUFreePPoE} (see
\prettyref{eq:SecondCouplingASInc} and \eqref{eq:SecondInclusionHighProbInc}),
since $(\mathcal{I}(\nu_{j})-x_{i})\cap A=\mathcal{I}(\nu_{j}^{i})\cap A$
for all $i$ and $j$.
\end{proof}

We now state the proposition which couples processes like $\nu_{j}^{i}$
with random interlacements. Note that we will apply this proposition
after ``decoupling'' the boxes $A_{1},...,A_{n},$ using \prettyref{cor:CouplePPoEandIndepPPoEs},
and that the statement of \prettyref{pro:CouplePPoEandRI} therefore
does not refer to these boxes or the centres $x_{1},...x_{n},$ except
through their separation $s$ (recall \prettyref{eq:DefOfSeparation}),
which goes into the definition of the radii of the boxes $A$ and
$B$ (see \eqref{eq:DefOfA} and \eqref{eq:ABCNotation}). The interpretation
of $s$ as separation is therefore irrelevant, and for the purposes
of the following proposition it can be simply considered as a parameter
that (together with $\varepsilon$) determines the radii of $A$ and
$B$.
\begin{prop} \label{pro:CouplePPoEandRI}($d\ge3$) 

\medskip
Let $\eta$ be a Poisson point
process on $\Gamma(\mathbb{Z}^{d})$ with intensity measure $u\kappa_{3}$,
$u\ge0$, constructed on some probability space $(\Omega,\mathcal{A},Q)$.
If $s\ge c(\varepsilon)$ and $1\ge\delta\ge cs^{-c(\varepsilon)}$
then we can construct a probability space $(\Omega',\mathcal{A}',Q')$
and (on the product space) independent $\sigma(\eta)\times\mathcal{A}-$measurable
random sets $\mathcal{I}_{1},\mathcal{I}_{2}\subset\mathbb{Z}^{d}$
such that $\mathcal{I}_{1}$ has the law of $\mathcal{I}^{u(1-\delta)}\cap A$
under $Q_{0}$, $\mathcal{I}_{2}$ has the law of $\mathcal{I}^{2u\delta}\cap A$
under $Q_{0}$, and $Q\otimes Q'[I_{4}]\ge1-ce^{-cs^{c(\varepsilon)}}$,
where $I_{4}$ is the event in \eqref{eq:ThirdCouplingInclusion}.
\end{prop}

\prettyref{pro:CouplePPoEandRI} will be proved in \prettyref{sec:ToRI}.

We are now ready to start the proof of \prettyref{thm:Coupling}.
We will apply \prettyref{pro:CoupleRWandPPoE} to the random walk,
then apply \prettyref{cor:CouplePPoEandIndepPPoEs} to the resulting
Poisson point processes, and finally apply \prettyref{pro:CouplePPoEandRI}
to the Poisson point processes resulting from \prettyref{cor:CouplePPoEandIndepPPoEs}.
This gives us random subsets of $\mathbb{Z}^{d}$ from which we will
construct random interlacements.
\begin{proof}[Proof of \prettyref{thm:Coupling}.]
Throughout the proof we decrease $\Cr{couplingerr}(\varepsilon)$ whenever
necessary so that the conditions on $u,\delta$ and $n$ needed for
\prettyref{pro:CoupleRWandPPoE}, \prettyref{cor:CouplePPoEandIndepPPoEs}
or \prettyref{pro:CouplePPoEandRI} to hold are fulfilled. We first
apply \prettyref{pro:CoupleRWandPPoE} with $\frac{\delta}{14}$ in
place of $\delta$ to get the space $(\Omega_{2},\mathcal{A}_{2},Q_{2})$
and independent Poisson point processes $\mu_{1}$ and $\mu_{2}$
such that $\mu_{1}$ has intensity $u(1-\frac{\delta}{14})\kappa_{1}$,
$\mu_{2}$ has intensity $u\frac{\delta}{7}\kappa_{1}$ and
\begin{equation}
Q_{2}[\mathcal{I}(\mu_{1})\cap\bar{A}\subset Y(0,uN^{d})\cap\bar{A}\subset\mathcal{I}(\mu_{1}+\mu_{2})\cap\bar{A}]\ge1-cue^{-cs^{c(\varepsilon)}}\mbox{ for }s\ge c(\varepsilon).\label{eq:CouplingRWPPoE}
\end{equation}

\medskip\n
Next we apply \prettyref{cor:CouplePPoEandIndepPPoEs} once with $\mu_{1}$
in place of $\nu$, $u(1-\frac{\delta}{14})$ in place of $u$ and
$\frac{\delta}{14}$ in place of $\delta$ and extend the space $(\Omega_{2},\mathcal{A}_{2},Q_{2})$
to get the space $(\Omega_{1},\mathcal{A}_{1},Q_{1})$ with Poisson
point processes $\nu_{1}^{i},\nu_{2}^{i}$ of intensities $u(1-\frac{\delta}{14})^{2}\kappa_{3}$
and $u(1-\frac{\delta}{14})\frac{\delta}{7}\kappa_{3}$ such that
$\nu_{1}^{i},\nu_{2}^{i},i=1,...,n,\mu_{2}$ are mutually independent
and (for $s\ge c(\varepsilon)$)
\begin{equation}
Q_{1}[\mathcal{I}(\nu_{1}^{i})\cap A\subset(\mathcal{I}(\mu_{1})-x_{i})\cap A\subset\mathcal{I}(\nu_{1}^{i}+\nu_{2}^{i})\cap A\mbox{ for all }i]\ge1-ce^{-cs^{c(\varepsilon)}}.\label{eq:CouplingSecondStep}
\end{equation}

\smallskip\n
For convenience we may ``thicken'' each $\nu_{2}^{i}$ so that they
have intensity $u\frac{\delta}{7}\kappa_{3}$, while preserving the
independence of $\nu_{1}^{i},\nu_{2}^{i},i=1,...,n,\mu_{2}$ and the
validity \eqref{eq:CouplingSecondStep} (by extending the space with independent
Poisson point processes of intensities given by an appropriate multiple of $\kappa_3$, and adding
them to the original $\nu_2^i$). Repeating this extension
but with $\mu_{2}$ in place of $\mu$, $u\frac{\delta}{7}$ in place
of $u$ and $1$ in place of $\delta$, we furthermore get processes
$\nu_{3}^{i}$ of intensity $u\frac{2\delta}{7}\kappa_{3}$ (arising
from the $\nu_{2}^{i}$ in the statement of \prettyref{cor:CouplePPoEandIndepPPoEs},
the $\nu_{1}^{i}$ in the statement of \prettyref{cor:CouplePPoEandIndepPPoEs}
are zero since $u(1-\delta)=0$) such that $\nu_{1}^{i},\nu_{2}^{i},\nu_{3}^{i},i=1,...,n$
are mutually independent and 
\begin{equation}
Q_{1}[(\mathcal{I}(\mu_{2})-x_{i})\cap A\subset\mathcal{I}(\nu_{3}^{i})\cap A\mbox{ for all }i]\ge1-ce^{-cs^{c(\varepsilon)}}\mbox{ for }s\ge c(\varepsilon).\label{eq:CouplingSecondStepSecondPart}
\end{equation}

\smallskip\n
Now apply \prettyref{pro:CouplePPoEandRI} with $u(1-\frac{\delta}{14})^{2}$
in place of $u$, $\frac{\delta}{14}$ in place of $\delta$, and
$\nu_{1}^{1}$ in place of $\mu$ and extend the space with mutually
independent sets $\mathcal{I}_{1,1},\mathcal{I}_{2,1}$ (independent
of $\nu_{1}^{j},\nu_{2}^{i},\nu_{3}^{i},j\ge2,i\ge1$) such that $\mathcal{I}_{1,1}$
has the law of $\mathcal{I}^{u(1-\frac{\delta}{14})^{3}}\cap A$ under
$Q_{0}$, $\mathcal{I}_{2,1}$ has the law of $\mathcal{I}^{u(1-\frac{\delta}{14})^{2}\frac{\delta}{7}}\cap A$
under $Q_{0}$, and $Q_{1}[\mathcal{I}_{1,1}\cap A\subset\mathcal{I}(\nu_{1}^{1})\cap A\subset(\mathcal{I}_{1,1}\cup\mathcal{I}_{2,1})\cap A]\ge1-cue^{-cs^{c}}$
for $s\ge c(\varepsilon)$. Then apply \prettyref{pro:CouplePPoEandRI}
once again with $u\frac{3\delta}{7}$ in place of $u$, $1$ in place
of $\delta$ and $\nu_{2}^{1}+\nu_{3}^{1}$ (which is a Poisson point
process of intensity $u\frac{3\delta}{7}\kappa_{3}$) in place of $\eta$, to
extend the space with a random set $\mathcal{I}_{3,1}$ (independent
of $\mathcal{I}_{1,1},\mathcal{I}_{2,1},\nu_{1}^{i},\nu_{2}^{i},\nu_{3}^{i},i\ge2$)
such that $\mathcal{I}_{3,1}$ has the law of $\mathcal{I}^{u\frac{6\delta}{7}}$
under $Q_{0}$, and such that $Q_{1}[\mathcal{I}(\nu_{2}^{1}+\nu_{3}^{1})\cap A\subset\mathcal{I}_{3,1}\cap A]\ge1-cue^{-cs^{c}}$ for $s\ge c(\varepsilon)$ (similarly to before $\mathcal{I}_{3,1}$
arises from the $\mathcal{I}_{2}$ of the statement of \prettyref{pro:CouplePPoEandRI},
$\mathcal{I}_{1}$ is empty since $u(1-\delta)=0$). We can repeat
this for $i=2,3,...,n,$ each time extending the space, to get mutually
independent sets $\mathcal{I}_{1,i},\mathcal{I}_{2,i},\mathcal{I}_{3,i},i=1,...,n,$
such that for each $j=1,2,3$ the $\mathcal{I}_{j,i},i=1,...,n,$
have the same law, and for all $i$ and $s\ge c(\varepsilon)$ 
\begin{equation}
Q_{1}[\mathcal{I}_{1,i}\cap A\subset\mathcal{I}(\nu_{1}^{i})\cap A\subset(\mathcal{I}_{1,i}\cup\mathcal{I}_{2,i})\cap A,\mathcal{I}(\nu_{2}^{i}+\nu_{3}^{i})\cap A\subset\mathcal{I}_{3,i}\cap A]\ge1-cue^{-cs^{c}}.\label{eq:CouplingPPoERI}
\end{equation}

\smallskip\n
By \prettyref{eq:CouplingRWPPoE}, \prettyref{eq:CouplingSecondStep},
\prettyref{eq:CouplingSecondStepSecondPart} and \eqref{eq:CouplingPPoERI}
we have for all $i$ and (possibly decreasing $\Cr{couplingerr}$ and recalling
that $u\ge s^{-c(\varepsilon)}$) that for $s\ge c(\varepsilon)$
\begin{equation}
Q_{1}[\mathcal{I}_{1,i}\cap A\subset(Y(0,uN^{d})-x_{i})\cap A\subset(\mathcal{I}_{1,i}\cup\mathcal{I}_{2,i}\cup\mathcal{I}_{3,i})\cap A]\ge1-\Cr{couplingerr}^{-1}ue^{-\Cr{couplingerr}^{-1}s^{\Cr{couplingerr}}}.\label{eq:CouplingAllPropsTogether}
\end{equation}

\smallskip\n
It now only remains to construct ``proper'' random interlacements
from $\mathcal{I}_{1,i},\mathcal{I}_{2,i},\mathcal{I}_{3,i},i=1,...,n.$

By \prettyref{eq:AdditiveInterlacements} the $\mathcal{I}_{2,i}\cup\mathcal{I}_{3,i}$
have the law of $\mathcal{I}^{u_{2}}\cap A$ under $Q_{0}$, where
$u_{2}=u(1-\frac{\delta}{14})^{2}\frac{\delta}{7}+u\frac{6\delta}{7}$.
Once again by \prettyref{eq:AdditiveInterlacements} the pair $(\mathcal{I}_{1,i}\cap A,(\mathcal{I}_{1,i}\cup\mathcal{I}_{2,i}\cup\mathcal{I}_{3,i})\cap A)$
has the law of $(\mathcal{I}^{u_{1}}\cap A,\mathcal{I}^{u_{1}+u_{2}}\cap A)$
under $Q_{0}$, where $u_{1}=u(1-\frac{\delta}{14})^{3}$. But this
pair takes only finitely many values (so that the set of values that
are taken with positive probability together have probability one),
so we can, by ``sampling from the conditional law (under $Q_{0}$)
of $(\mathcal{I}^{u})_{u\ge0}$ given $(\mathcal{I}^{u_{1}}\cap A,\mathcal{I}^{u_{1}+u_{2}}\cap A)$'',
construct for each $i=1,...,n,$ a family $(\mathcal{I}_{i}^{u})_{u\ge0}$
with the law of $(\mathcal{I}^{u})_{u\ge0}$ under $Q_{0}$ such that
$\mathcal{I}_{i}^{u(1-\delta)}\cap A\subset\mathcal{I}_{i}^{u_{1}}\cap A=\mathcal{I}_{1,i}\cap A$
(recall \eqref{eq:IncreasingInterlacements} and note $u(1-\delta)\le u(1-\frac{3\delta}{14})\le u_{1}$)
almost surely and $(\mathcal{I}_{1,i}\cup\mathcal{I}_{2,i}\cup\mathcal{I}_{3,i})\cap A=\mathcal{I}_{i}^{u_{1}+u_{2}}\cap A\subset\mathcal{I}_{i}^{u(1+\delta)}$
(note $u_{1}+u_{2}\le u+u\frac{\delta}{7}+u\frac{6\delta}{7}=u(1+\delta)$)
almost surely, which combined with \eqref{eq:CouplingAllPropsTogether}
implies \eqref{eq:CouplingInclusion} for $s\ge c(\varepsilon)$.
But by adjusting constants \eqref{eq:CouplingInclusion} holds for
all $s$, so the proof of \prettyref{thm:Coupling} is complete.
\end{proof}
\prettyref{thm:Coupling} (and therefore also \prettyref{thm:G_Gumbel})
has now been reduced to \prettyref{pro:CoupleRWandPPoE}, \prettyref{pro:CouplePPoEandUFreePPoE}
and \prettyref{pro:CouplePPoEandRI}.

\section{\label{sec:Quasistationary}Quasistationary distribution}

In this section we introduce the \emph{quasistationary distribution},
which is a probability distribution on $\mathbb{T}_{N}\backslash\bar{C}$
(recall our standing assumption \prettyref{eq:StandingSeparationAssumption},
\prettyref{eq:ABCNotation} and \prettyref{eq:SetUnionTranslateNotation})
denoted by $\sigma(\cdot)$ and which will be an essential tool when
we prove (in \prettyref{sec:Poissonization}) the coupling \prettyref{pro:CoupleRWandPPoE}.

The main result is \prettyref{pro:QSConvToQuasiStat}, which says
that for all $x,y\in\mathbb{T}_{N}\backslash\bar{C}$ the probability
$P_{x}[Y_{t}=y|H_{\bar{C}}>t^{\star}]$ (recall the definition of
$t^{\star}$ from \prettyref{eq:DefOfUandtStar}) is very close to
$\sigma(y)$ (and thus almost independent of $x$). The result will
allow us to show, in \prettyref{sec:Poissonization}, that regardless
of where the random walk $Y_{\cdot}$ starts, $Y_{U}$ (where $U$
was defined in \prettyref{eq:DefOfUandtStar}) is very close in distribution
to the quasistationary distribution, and this in turn will let us
``cut the random walk'' $Y_{\cdot}$ into almost independent excursions,
each with law close to $P_{\sigma}[Y_{\cdot\wedge U}\in dw]$ (cf.
\prettyref{eq:DefOfKappa1}). This will be the main step in constructing
the Poisson processes $\mu_{1}$ and $\mu_{2}$ from the statement
of \prettyref{pro:CoupleRWandPPoE}.

At the end of this section we also give a result that says that the
hitting distribution on $\partial_{i}\bar{A}$ when starting random
walk from the quasistationary distribution is approximately the normalized
sum of the equilibrium distributions on $A_{1},A_{2},...,A_{n}$ (see
\prettyref{eq:QSEQQuasiStatEquilMeas}). This result will be used
several times in the subsequent sections.

Let us now formally introduce the quasistationary distribution. We
define the $(N^{d}-|\bar{C}|)\times(N^{d}-|\bar{C}|)$ matrix $(P^{\bar{C}})_{x,y\in\mathbb{T}_{N}\backslash\bar{C}}=\frac{1}{2d}1_{\{x\sim y\}},$where
$x\sim y$ means that $x$ and $y$ share an edge in $\mathbb{T}_{N}$.
When $s\ge c(\varepsilon)$ so that $\mathbb{T}_{N}\backslash\bar{C}$
is connected the Perron-Frobenius theorem (Theorem 8.2, p. 151 in
\cite{SerreMatrices}) implies that this (real symmetric, non-negative
and irreducible) matrix has a unique largest eigenvalue $\lambda_{1}^{\bar{C}}$
with a non-negative normalized eigenvector $v_{1}$. We let $\lambda_{2}^{\bar{C}}$
denote the second largest eigenvalue of $P^{\bar{C}}$. The quasistationary
distribution $\sigma$ on $\mathbb{T}_{N}\backslash\bar{C}$ is then defined by
\begin{equation}
\sigma(x)=\frac{(v_{1})_{x}}{v_{1}^{T}{\bf 1}}\mbox{ for }x\in\mathbb{T}_{N}\backslash\bar{C}.\label{eq:DefOfQuasiStat}
\end{equation}
Since $\mathbb{T}_{N}\backslash\bar{C}$ is connected (when $s\ge c(\varepsilon)$)
it holds that (see (6.6.3), p. 91 in \cite{keilson1979markov})
\begin{equation}
\lim_{t\rightarrow\infty}P_{x}[Y_{t}=y|H_{\bar{C}}>t]=\sigma(y)\mbox{ for all }x,y\in\mathbb{T}_{N}\backslash\bar{C}.\label{eq:QSNonQuantConv}
\end{equation}
\prettyref{pro:QSConvToQuasiStat}, the main result of the section,
is a quantitative version of \prettyref{eq:QSNonQuantConv}, which
we now state. Recall once again the assumption \prettyref{eq:StandingSeparationAssumption},
and the definition of $t^{\star}$ in \prettyref{eq:DefOfUandtStar}.
\begin{prop} \label{pro:QSConvToQuasiStat}($d\ge3$,$N\ge3$) 

\medskip
If $n\le s^{c(\varepsilon)}$ and $s\ge c(\varepsilon)$ then
\begin{equation}
\sup_{x,y\in\mathbb{T}_{N}\backslash\bar{C}}\left|P_{x}[Y_{t^{\star}}=y|H_{\bar{C}}>t^{\star}]-\sigma(y)\right|\le ce^{-cN^{c(\varepsilon)}}.\label{eq:QSConvToQuasiStat}
\end{equation}
\end{prop}

To prove \prettyref{pro:QSConvToQuasiStat} we will express $P_{x}[Y_{t}=y|H_{\bar{C}}>t^{\star}]$
in terms of the matrix $P^{\bar{C}}$, and then use the spectral expansion
of $P^{\bar{C}}$ to prove that $P_{x}[Y_{t}=y|H_{\bar{C}}>t^{\star}]$
is close to $\frac{(v_{1})_{y}}{v_{1}^{T}\bm{1}}$. To control the
error we will need an estimate of the spectral gap of $P^{\bar{C}}$,
which we obtain in \prettyref{lem:QSSpectralGap}, and a lower bound
on the minimum of $\sigma(\cdot)$, which we obtain in \prettyref{lem:QSMQuasiStatMinimum}.
This is the approach taken to prove Lemma 3.9 in \cite{TeixeiraWindischOnTheFrag},
which is essentially the same result when $n=1$ so that $\bar{C}$
consists of only one box. Since for us $\bar{C}$ consists of many
boxes, bounding the minimum of $\sigma(\cdot)$ is harder, and achieving
a good enough bound will consume most of our efforts in this section.

We prove \prettyref{pro:QSConvToQuasiStat} after introducing \prettyref{lem:QSSpectralGap}
and \prettyref{lem:QSMQuasiStatMinimum}. To prove \prettyref{lem:QSSpectralGap}
we will need the following lemma, which roughly speaking says that
$E[H_{V}]\approx\frac{N^{d}}{\mbox{cap}(V)}$ for appropriate sets
$V\subset\mathbb{T}_{N}$.
\begin{lem} \label{lem:ETTEntraceTimesTorus}($d\ge3$,$N\ge3$) 

\medskip
For any (non-empty) $V\subset\bar{C}$ let $V^{i}=(V\cap C_{i})-x_{i}\subset\mathbb{Z}^{d},i=1,...,n$. Then if $s\ge c(\varepsilon)$ we have
\begin{equation}
\frac{N^{d}}{E[H_{V}]\sum_{i=1}^{n}\,{\rm cap}(V^{i})}\le1+c(\varepsilon)s^{-c(\varepsilon)}.\label{eq:ETTEntranceTimeLowerBound}
\end{equation}
Furthermore if $V\subset\bar{B}$ and $n\le s^{c(\varepsilon)}$ then
\begin{eqnarray}
 & 1-c(\varepsilon)s^{-c(\varepsilon)}\le\frac{N^{d}}{E[H_{V}]\sum_{i=1}^{n} {\rm cap}(V^{i})},\mbox{ and }\label{eq:ETTEntraceTimeUpperBound}\\
 & (1-c(\varepsilon)s^{-c(\varepsilon)})E[H_{V}]\le\underset{x\notin\bar{C}}{\inf}E_{x}[H_{V}]\le\underset{x\in\mathbb{T}_{N}}{\sup}E_{x}[H_{V}]\le(1+cN^{-c(\varepsilon)})E[H_{V}].\label{eq:ETTStartingFromPointandUniformComparison}
\end{eqnarray}
\end{lem}

The proof of \prettyref{lem:ETTEntraceTimesTorus} is contained in
the appendix. We are now ready to prove \prettyref{lem:QSSpectralGap}
about the spectral gap of $P^{\bar{C}}$.
\begin{lem} \label{lem:QSSpectralGap}($d\ge3$,$N\ge3$) 

\medskip
If $n\le s^{c(\varepsilon)}$ and $s\ge c(\varepsilon)$ we have
\begin{equation}
\lambda_{1}^{\bar{C}}-\lambda_{2}^{\bar{C}}\ge cN^{-2}.\label{eq:QSSpectralGap}
\end{equation}
\end{lem}

\begin{proof}
Lemma A.3 of \cite{TeixeiraWindischOnTheFrag} contains a proof for
$n=1$ (note that $B$ in that lemma plays the role of $\bar{C}$
in this lemma). The proof for $n>1$ is almost identical; one replaces
$B$ with $\bar{C}$ and the inequality $E[H_{B}]\ge c(\varepsilon)N^{2+\frac{\varepsilon(d-2)}{2}}$
with
\begin{equation}\label{eq:reqonn}
E[H_{\bar{C}}]\overset{\eqref{eq:AsymptoticsOfCapOfBox},\eqref{eq:ABCNotation},\eqref{eq:ETTEntranceTimeLowerBound}}{\ge}\frac{c(\varepsilon)N^{d}}{s^{(1-\frac{\varepsilon}{4})(d-2)}n}\overset{s\le N}{\ge}c(\varepsilon)s^{\frac{\varepsilon(d-2)}{8}}\frac{N^{2+\frac{\varepsilon(d-2)}{8}}}{n}\overset{s\ge c(\varepsilon),n\le N^{\frac{\varepsilon(d-2)}{16}}}{\ge}cN^{2+\frac{\varepsilon(d-2)}{16}}.
\end{equation}
We omit the details.
\end{proof}
The bound on the minimum of $\sigma(\cdot)$ comes from the following lemma.
\begin{lem}
\label{lem:QSMQuasiStatMinimum}($d\ge3$,$N\ge3$) 

\medskip
If $s\ge c(\varepsilon)$ we have
\begin{equation}
\inf_{x\in\mathbb{T}_{N}\backslash\bar{C}}\sigma(x)\ge N^{-cn}.\label{eq:QSMQuasiStatMinimum}
\end{equation}
\end{lem}

\smallskip
We will prove \prettyref{lem:QSMQuasiStatMinimum} after finishing
the proof of \prettyref{pro:QSConvToQuasiStat}.
\begin{proof}[Proof of \prettyref{pro:QSConvToQuasiStat}.]
Note that $P_{x}[Y_{t^{\star}}=y,H_{\bar{C}}>t^{\star}]=\delta_{x}^{T}e^{-t^{\star}(I-P^{\bar{C}})}\delta_{y}$
for $x,y\in\mathbb{T}_{N}\backslash\bar{C}$, so
\begin{equation}
P_{x}[Y_{t^{\star}}=y|H_{\bar{C}}>t^{\star}]=\frac{\delta_{x}^{T}e^{-t^{\star}(I-P^{\bar{C}})}\delta_{y}}{\delta_{x}^{T}e^{-t^{\star}(I-P^{\bar{C}})}\bm{1}},\label{eq:QSCTQSCondProbInTermsOfSemiGroup}
\end{equation}

\smallskip\n
where $\bm{1}$ denotes the vector $(1,...,1)\in\mathbb{R}^{N^{d}-|\bar{C}|}$.
By the spectral theorem we have
\begin{eqnarray*}
e^{-t^{\star}(I-P^{\bar{C}})} & = & e^{-t^{\star}(1-\lambda_{1}^{\bar{C}})}v_{1}v_{1}^{T}+e^{-t^{\star}(1-\lambda_{2}^{\bar{C}})}R,
\end{eqnarray*}
where $R$ is an operator onto the space orthogonal to $v_{1}$ with
operator norm $1$ (we use the Euclidean norm on $\mathbb{R}^{N^{d}-|\bar{C}|}$).
We thus see from \prettyref{eq:QSCTQSCondProbInTermsOfSemiGroup}
that if we let $R'=\frac{e^{-t^{\star}(\lambda_{1}^{\bar{C}}-\lambda_{2}^{\bar{C}})}}{\sigma(x)(v_{1}^{T}\bm{1})^{2}}R$
then
\begin{equation}
P_{x}[Y_{t^{\star}}=y|H_{\bar{C}}>t^{\star}]=\frac{(v_{1})_{x}(v_{1})_{y}+\delta_{x}^{T}e^{-t^{\star}(\lambda_{1}^{\bar{C}}-\lambda_{2}^{\bar{C}})}R\delta_{y}}{(v_{1})_{x}(v_{1}^{T}\bm{1})+\delta_{x}^{T}e^{-t^{\star}(\lambda_{1}^{\bar{C}}-\lambda_{2}^{\bar{C}})}R\bm{1}}\overset{\prettyref{eq:DefOfQuasiStat}}{=}\frac{\sigma(y)+\delta_{x}^{T}R'\delta_{y}}{1+\delta_{x}^{T}R'\bm{1}}.\label{eq:QSCTQSAlmostThere}
\end{equation}

\smallskip\n
Now since we require $n\le s^{c(\varepsilon)}$ and $s\ge c(\varepsilon)$
both \prettyref{lem:QSSpectralGap} and \prettyref{lem:QSMQuasiStatMinimum} hold. Therefore 
\begin{equation*}
|R'\delta_{y}|,|R'\bm{1}|\le N^{d}\frac{e^{-t^{\star}(\lambda_{1}^{\bar{C}}-\lambda_{2}^{\bar{C}})}}{\sigma(x)|v_{1}^{T}\bm{1}|^{2}}\le ce^{-cN^{c(\varepsilon)}},
\end{equation*}
since 
\begin{equation*}
e^{-t^{\star}(\lambda_{1}^{\bar{C}}-\lambda_{2}^{\bar{C}})}\overset{\eqref{eq:DefOfUandtStar},\prettyref{eq:QSSpectralGap}}{\le}e^{-cN^{\frac{\varepsilon}{100}}},
\sigma(x)\overset{\prettyref{eq:QSMQuasiStatMinimum}}{\ge}N^{-cn}\overset{n\le s^{\frac{\varepsilon}{200}}\le N^{\frac{\varepsilon}{200}}}{\ge}e^{-cN^{\frac{\varepsilon}{200}}\log N},|v_{1}^{T}\bm{1}|\ge|v_{1}|^{2}=1
\end{equation*}

\smallskip\n
and $N\ge c(\varepsilon)$ (since $s\ge c(\varepsilon)$). Thus \prettyref{eq:QSConvToQuasiStat} follows
from \prettyref{eq:QSCTQSAlmostThere} (using again that $N\ge c(\varepsilon)$).
This completes the proof of \prettyref{pro:QSConvToQuasiStat}.
\end{proof}

It still remains to prove \prettyref{lem:QSMQuasiStatMinimum}. The
proof will involve further concentric boxes $D$, $E$ and $F$ such
that $A\subset B\subset C\subset D\subset E\subset F$ defined by
\begin{equation}
D=B(0,s^{1-\frac{\varepsilon}{8}})\subset E=B(0,s^{1-\frac{\varepsilon}{16}})\subset F=B(0,s^{1-\frac{\varepsilon}{32}}).\label{eq:DefOfD}
\end{equation}

\begin{proof}[Proof of \prettyref{lem:QSMQuasiStatMinimum}.]
Let $y$ be the maximum of $\sigma(\cdot)$. Since $\sigma(\cdot)$
is a probability distribution we have $\sigma(y)\ge N^{-d}$. Also
by reversibility we have for any $x\notin\bar{C}$ and $t\ge0$ that
$P_{x}[Y_{t}=y,H_{\bar{C}}>t]=P_{y}[Y_{t}=x,H_{\bar{C}}>t]$ and thus
\begin{equation}
P_{y}[Y_{t}=x|H_{\bar{C}}>t]=P_{x}[Y_{t}=y|H_{\bar{C}}>t]\mbox{\f $\dis\frac{P_{x}[H_{\bar{C}}>t]}{P_{y}[H_{\bar{C}}>t]}$}.\label{eq:QSMRatioOfProbs}
\end{equation}

\smallskip\n
Since by the Markov property $P_{x}[H_{\bar{C}}>t]\ge P_{x}[H_{y}<H_{\bar{C}}]P_{y}[H_{\bar{C}}>t]$
we see, by taking the limit $t\rightarrow\infty$ in \prettyref{eq:QSMRatioOfProbs}
and using \prettyref{eq:QSNonQuantConv}, that $\sigma(x)\ge\sigma(y)P_{x}[H_{y}<H_{\bar{C}}]\ge N^{-d}P_{x}[H_{y}<H_{\bar{C}}]$.
To prove \prettyref{eq:QSMQuasiStatMinimum} it thus suffices to show that
\begin{equation}
P_{x}[H_{y}<H_{\bar{C}}]\ge N^{-cn}\mbox{ for all }x,y\notin\bar{C}.\label{eq:QSMSufficeToShow1}
\end{equation}
For $i=1,...,n$, and $x\in D_{i}\backslash C_{i}$ (recall \prettyref{eq:SetUnionTranslateNotation},
\prettyref{eq:DefOfD}) it follows from a one-dimensional random walk
estimate that $P_{x}[T_{D_{i}}<H_{\bar{C}}]\ge N^{-1}$, so that by
the Markov property $P_{x}[H_{y}<H_{\bar{C}}]\ge N^{-1}\inf_{x'\in\partial_{e}D_{i}}P_{x'}[H_{y}<H_{\bar{C}}]$.
If $x\notin\bar{D}$ and $y\in\bar{F}\backslash\bar{C}$ then we can
use that by reversibility $P_{x}[H_{y}<H_{\bar{C}}]=P_{y}[H_{x}<H_{\bar{C}}]$
and another one-dimensional random walk estimate to show $P_{x}[H_{y}<H_{\bar{C}}]\ge N^{-1}\inf_{y'\notin\partial_{e}\bar{F}}P_{x}[H_{y'}<H_{\bar{C}}]$.
To prove \prettyref{eq:QSMSufficeToShow1} and thus \prettyref{eq:QSMQuasiStatMinimum}
it therefore suffices to show (recall $y\notin\bar{F}$) 
\begin{equation}
P_{x}[H_{y}<H_{\bar{C}}]\ge N^{-cn}\mbox{ for all }x\notin\bar{D},y\notin\bar{F}.\label{eq:QSMSufficeToShow2}
\end{equation}

\smallskip\n
Now fix $y\notin\bar{F}$ and note that $P_{x}[H_{y}<H_{\bar{C}}]\ge cs^{-c}$
for all $x\in\partial_{e}(D+y)$, since $P_{x}[H_{y}<T_{E+y}]\ge cs^{-c}$
(e.g. by Proposition 1.5.9, p. 35 in \cite{LawlersLillaGrona}) and
$(E+y)\cap\bar{C}=\emptyset$, since $s\ge c(\varepsilon)$. Therefore
to prove \prettyref{eq:QSMSufficeToShow2} and thus \prettyref{eq:QSMQuasiStatMinimum} it suffices to show 
\begin{equation}
P_{x_{1}}[H_{y}<H_{\bar{C}}]\ge N^{-cn}P_{x_{2}}[H_{y}<H_{\bar{C}}]\mbox{ for all }x_{1},x_{2}\notin\bar{D}\cup(y+D).\label{eq:QSMSufficeToShow3}
\end{equation}

\smallskip\n
Consider the function $x\rightarrow P_{x}[H_{y}<H_{\bar{C}}]$. This
function is non-negative and harmonic on $\left(\bar{C}\cup\{y\}\right)^{c}$.
Thus by the Harnack inequality (Theorem 1.7.2, p. 42 in \cite{LawlersLillaGrona})
we have, for any $z\in\mathbb{T}_{N}$ and $r\ge0$ for which $B(z,2(r+1))\cap(\bar{C}\cup\{y\})=\emptyset$, that 
\begin{equation}
\inf_{x\in B(z,r+1)}P_{x}[H_{y}<H_{\bar{C}}]\ge c\sup_{x\in B(z,r+1)}P_{x}[H_{y}<H_{\bar{C}}].\label{eq:QSMHarnack}
\end{equation}
To iterate this inequality we need the following lemma.
\begin{lem} \label{lem:QSMCoveringLemma}($d\ge3$,$N\ge3$) 

\medskip
If $s\ge c(\varepsilon)$ one can cover $\left(\bar{D}\cup(y+D)\right)^{c}$ by $m\le cn\log N$
balls $B(z_{i},r_{i}),i=1,...,m$, that satisfy $B(z_{i},2(r_{i}+1))\cap(\bar{C}\cup\{y\})=\emptyset$. 
\end{lem}

Before proving \prettyref{lem:QSMCoveringLemma} we use it to show
\prettyref{eq:QSMQuasiStatMinimum}. The balls $B(z_{i},r_{i}+1)$
``overlap'' and cover the connected set $\left(\bar{D}\cup(y+D)\right)^{c}$
(recall $s\ge c(\varepsilon)$), so for any $x_{1},x_{2}\in\left(\bar{D}\cup(y+D)\right)^{c}$
we can find a ``path'' of at most $cn\log N$ balls $B(z,r+1)$
satisfying \prettyref{eq:QSMHarnack}, such that any two consecutive
balls intersect, and such that $x_{1}$ is in the first ball and $x_{2}$
is in the last. Applying \prettyref{eq:QSMHarnack} at most $cn\log N$
times along these ``paths'' yields \prettyref{eq:QSMSufficeToShow3}.
This completes the proof of \prettyref{eq:QSMQuasiStatMinimum}, so
only the proof of \prettyref{lem:QSMCoveringLemma} remains.
\begin{proof}[Proof of \prettyref{lem:QSMCoveringLemma} ]
By a standard argument $\left(B(0,\frac{N}{4})\right)^{c}$ can be
covered by a bounded number of balls $B(z,r)$ that satisfy $B(z,4r+2)\cap B(0,\frac{N}{8})=\emptyset$
(when $N\ge c$, which we may assume since we require $s\ge c$).
Furthermore each of the annuli $B(0,2^{l+1}s^{1-\frac{\varepsilon}{8}})\backslash B(0,2^{l}s^{1-\frac{\varepsilon}{8}}),l=0,1,2,3,...$
can be covered by a bounded number of balls $B(z,r)$ that satisfy
$B(z,4r+2))\cap B(0,2^{l-1}s^{1-\frac{\varepsilon}{8}})=\emptyset$
(since $s^{1-\frac{\varepsilon}{8}}\ge c$). Now combining at most
$c\log N$ coverings of annuli with the covering of $\left(B(0,\frac{N}{4})\right)^{c}$
we get that (provided $s\ge c(\varepsilon)$ so that $C\subset B(0,\frac{1}{2}s^{1-\frac{\varepsilon}{8}})$)
\begin{equation}
\mbox{one can cover }D^{c}\mbox{ by at most }c\log N\mbox{ balls }B(z,r)\mbox{ with }B(z,4r+2)\cap C=\emptyset.\label{eq:QSMCLOneBoxCase}
\end{equation}
We can now use \prettyref{eq:QSMCLOneBoxCase} to get for each $x\in\{x_{1},...,x_{n},y\}$
a covering $\mathcal{B}_{x}$ of $\left(x+D\right)^{c}$ consisting
of at most $c\log N$ balls such that if $B(v,r)\in\mathcal{B}_{x}$
then $B(v,4r+2)\cap(x+C)=\emptyset$. We now combine the coverings
by picking for every $x\notin\bar{D}\cup(y+D)$ a ball from $\mathcal{B}_{z(x)}$
that contains $x$, where $z(x)$ is a member of $\{x_{1},...,x_{n},y\}$
of minimal $d_{\infty}$ distance to $x$. This gives a covering of
$(\bar{D}\cup(y+D))^{c}$ by at most $c(n+1)\log N\le cn\log N$ balls.
Also if $x\notin\bar{D}\cup(y+D)$ and $x \in B(v,r)\in\mathcal{B}_{z(x)}$
then $B(v,4r+2)\cap(z(x)+C)=\emptyset$, which implies
$d_{\infty}(v,z(x)) > s^{1-\frac{\varepsilon}{4}}+4r+2$.
This in turn implies
$d_{\infty}(v,\{x_{1},...,x_{n},y\}) \ge d_{\infty}(x,z(x))-d_{\infty}(v,x) \ge d_{\infty}(v,z(x))
- 2d_{\infty}(v,x) > s^{1-\frac{\varepsilon}{4}}+2(r+1)$ (using that $z(x)$ is of minimal distance to
$x$ and $d_{\infty}(x,v)\le r$), so that $B(v,2(r+1))\cap(\bar{C}\cup\{y+C\})=\emptyset$. Thus we
have the desired covering and the proof of \prettyref{lem:QSMCoveringLemma} is complete.
\end{proof}
This completes the proof of \prettyref{lem:QSMQuasiStatMinimum}.
\end{proof}

Finally we state and prove \prettyref{lem:QSEQQuasiStatEquilMeas},
which says that the hitting distribution on $\partial_{i}\bar{A}$
when starting from $\sigma$ is approximately the normalized sum of
the equilibrium distributions on $\partial_{i}A_{1},...,\partial_{i}A_{n}$.
\begin{lem}
\label{lem:QSEQQuasiStatEquilMeas}($d\ge3$,$N\ge3$) 

\medskip
If $s\ge c(\varepsilon)$ and $n\le s^{c(\varepsilon)}$ then for all $i=1,...,n,$
\begin{equation}
\frac{e_{A}(x-x_{i})}{n\, {\rm cap}(A)}(1-cs^{-c(\varepsilon)})\le P_{\sigma}[Y_{H_{\bar{A}}}=x]\le\frac{e_{A}(x-x_{i})}{n\, {\rm cap}(A)}(1+cs^{-c(\varepsilon)})\mbox{ for all }x\in\partial_{i}A_{i}.\label{eq:QSEQQuasiStatEquilMeas}
\end{equation}
\end{lem}

\begin{proof}
The proof is very similar to the proof of Lemma 3.10 in \cite{TeixeiraWindischOnTheFrag}.
By redefining $t^{\star}$ and $U$ from \cite{TeixeiraWindischOnTheFrag}
to agree with our definition in \eqref{eq:DefOfUandtStar}, replacing
$A$ with $\bar{A}$, $B$ with $\bar{C}$, and the application of
Lemma 3.9 from \cite{TeixeiraWindischOnTheFrag} with an application
of \prettyref{pro:QSConvToQuasiStat} (which is allowed since we assume
$n\le s^{c(\varepsilon)}$ and $s\ge c(\varepsilon)$) the argument
leading up to (3.42) in \cite{TeixeiraWindischOnTheFrag} becomes a proof of 
\begin{equation}
|P_{x}[\tilde{H}_{\bar{A}}>U]-P_{\sigma}[Y_{H_{\bar{A}}}=x]\dsl_{y\in\partial_{i}\bar{A}}P_{y}[\tilde{H}_{\bar{A}}>U]|\le ce^{-cN^{c(\varepsilon)}}\mbox{ for all }x\in\partial_{i}\bar{A}.\label{eq:ThingByReversibility}
\end{equation}
Furthermore note that by \eqref{eq:DefOfRelativeEquilAndCap}, \eqref{eq:DefOfUandtStar}
and the strong Markov property applied at time $T_{\bar{D}}$ we have
$e_{A,D}(x-x_{i})\inf_{x\notin\bar{D}}P_{x}[H_{\bar{C}}>U]\le P_{x}[\tilde{H}_{\bar{A}}>U]$.
But also
\begin{equation}
\sup_{x\notin\bar{D}}P_{x}[H_{\bar{C}}<U]\overset{\eqref{eq:DefOfUandtStar}}{=}\sup_{x\notin\bar{D}}P_{x}[H_{\bar{C}}<N^{2+\frac{\varepsilon}{100}}]\overset{\eqref{eq:BHBBallHittingBound}}{\le}nc(\varepsilon)s^{-c\varepsilon}\le c(\varepsilon)s^{-c\varepsilon}\mbox{ if }n\le s^{c\varepsilon},\label{eq:CDBallHitting}
\end{equation}
and thus by \eqref{eq:RelativeEquilLowerBound} we have $(1-cs^{-c(\varepsilon)})e_{A}(x-x_{i})\le P_{x}[\tilde{H}_{\bar{A}}>U]$
(recall $s\ge c(\varepsilon)$). Now 
\begin{equation*}
T_{C}\overset{\eqref{eq:DefOfUandtStar}}{\le}U \; \mbox{so} \; P_{x}[\tilde{H}_{\bar{A}}>U]\overset{\eqref{eq:DefOfRelativeEquilAndCap}}{\le}e_{A,C}(x-x_{i}),
\end{equation*}
and using \eqref{eq:RelativeEquilUpperBound} with $r=s^{1-\varepsilon}$
and $\lambda$ such that $s^{(1-\varepsilon)(1+\lambda)}=s^{1-\frac{\varepsilon}{4}}$
we get $e_{A,C}(x-x_{i})\le(1+cs^{-c(\varepsilon)})e_{A}(x-x_i)$ (since $s\ge c(\varepsilon)$). Thus 
\begin{equation}
(1-cs^{-c(\varepsilon)})e_{A}(x-x_{i})\le P_{x}[\tilde{H}_{\bar{A}}>U]\le e_{A}(x-x_{i})(1+cs^{-c(\varepsilon)}).\label{eq:CloseToEquil}
\end{equation}
But plugging \eqref{eq:CloseToEquil} into \eqref{eq:ThingByReversibility},
and using \eqref{eq:DefOfeKandCap} and \eqref{eq:EquilDistLowerBound}
yields \eqref{eq:QSEQQuasiStatEquilMeas}, cf. below (3.44) in \cite{TeixeiraWindischOnTheFrag}.
We omit the details.
\end{proof}

\section{\label{sec:Poissonization}Poissonization}

In this section the goal is to construct the coupling of random walk
in the torus with Poisson point processes of intensity a multiple
of $\kappa_{1}$, i.e. prove \prettyref{pro:CoupleRWandPPoE}. We
recall the standing assumption \prettyref{eq:StandingSeparationAssumption}.

First we define $R_{k},k\ge1$, the successive returns to $\bar{A}$
(see \eqref{eq:SetUnionTranslateNotation} for the notation) and $U_{k},k\ge0,$
the successive ``departures'' from $\bar{C}$, by
\begin{equation}
U_{0}=0,U_{k}=U\circ\theta_{R_{k}}+R_{k},k\ge1,R_{1}=H_{\bar{A}},R_{k}=H_{\bar{A}}\circ\theta_{U_{k-1}}+U_{k-1},k\ge2.\label{eq:DefOfRkUk}
\end{equation}
We call the segments $(Y_{(R_{k}+\cdot)\wedge U_{k}})_{k\ge1}$ the
\emph{excursions} of the random walk. The first step in the proof
of \prettyref{pro:CoupleRWandPPoE} is to couple the random walk $Y_{\cdot}$
when it starts from the \emph{quasistationary }distribution\emph{
}with i.i.d. processes $\tilde{Y}_{\cdot}^{1},\tilde{Y}_{\cdot}^{2},...$
with law $P_{\sigma}[Y_{\cdot\wedge U_{1}}\in dw]$, such that with
high probability $Y(U_{i-1},U_{i})\cap\bar{C}=\tilde{Y}^{i}(0,\infty)\cap\bar{C}$.
This is done in \prettyref{lem:ExcwithIIDExc} using \prettyref{pro:QSConvToQuasiStat}
from the previous section.

The second step in the proof of \prettyref{pro:CoupleRWandPPoE} is
to relate the stopping times $R_{k},U_{k}$, to deterministic times,
roughly speaking showing that $U_{[un\,{\rm cap}(A)]}\approx uN^{d}$.
This is done in \prettyref{lem:ETRTExcursionTimeandRealTime} using
large deviation estimates.

The third step in the proof of \prettyref{pro:CoupleRWandPPoE} is
to use the relation $U_{[un\,{\rm cap}(A)]}\approx uN^{d}$ to modify
the coupling from \prettyref{lem:ExcwithIIDExc} so that, very roughly,
$Y(0,uN^{d})\cap\bar{C}\approx\cup_{i=1}^{[un\,{\rm cap}(A)]}\tilde{Y}^{i}(0,\infty)\cap\bar{C}$
with high probability. This is done in \prettyref{pro:UnifRWIIDExc},
where we also use a mixing argument to ensure that the coupling, as
opposed to that from \prettyref{lem:ExcwithIIDExc}, has $Y_{\cdot}$
starting from the \emph{uniform }distribution\emph{.}

Finally at the end of the section we use \prettyref{pro:UnifRWIIDExc}
to prove \prettyref{pro:CoupleRWandPPoE} essentially by constructing
a point process from $\tilde{Y}_{\cdot}^{1},\tilde{Y}_{\cdot}^{2},...Y_{\cdot}^{J}$,
where $J$ is a Poisson random variable. We will see that this gives
rise to a Poisson point process which we can modify, using \prettyref{lem:QSEQQuasiStatEquilMeas}
to ``change'' the intensity from a multiple of $P_{\sigma}[Y_{\cdot\wedge U}\in dw]$
to a multiple of $P_{e}[Y_{\cdot\wedge U_{1}}\in dw]$ (i.e. of $\kappa_{1}$).

We now state and prove \prettyref{lem:ExcwithIIDExc} which couples
$Y_{\cdot}$ under $P_{\sigma}$ with i.i.d. excursions.
\begin{lem} \label{lem:ExcwithIIDExc}($d\ge3$,$N\ge3$)

\medskip
 If $n\le s^{c(\varepsilon)}$
and $s\ge c(\varepsilon)$ we can construct a coupling $(\Omega_{3},\mathcal{A}_{3},Q_{3})$
of a process $Y_{\cdot}$ with law $P_{\sigma}$ and an i.i.d. sequence
$\tilde{Y}_{\cdot}^{1},\tilde{Y}_{\cdot}^{2},...$ with law $P_{\sigma}[Y_{\cdot\wedge U_{1}}\in dw]$
such that
\begin{equation}
Q_{3}[Y(U_{i-1},U_{i})\cap\bar{C}\ne\tilde{Y}^{i}(0,\infty)\cap\bar{C}]\le e^{-cN^{c(\varepsilon)}}\mbox{ for all }i\ge1.\label{eq:ExcwithIIDExc}
\end{equation}
\end{lem}
\begin{proof}
For each $i$ we define $L_{i}$, the last time that $Y_{\cdot}$
leaves $\bar{C}$ before $U_{i}$, by
\[
L_{i}=L\circ\theta_{U_{i-1}}+U_{i-1},i\ge1,\mbox{ where }L=\sup\{t\le U_{1}:Y_{t}\in\bar{C}\},
\]
so that $0=U_{0}\le L_{1}\le U_{1}\le L_{2}\le U_{2}\le...$. Define
for convenience
\begin{equation}
\mathcal{L}_{i}=Y_{L_{i}},\mathcal{U}_{i}=Y_{U_{i}}\mbox{ and }\hat{Y}_{\cdot}^{i}=Y_{(U_{i-1}+\cdot)\wedge L_{i}},i\ge1.
\label{eq:forconvenience}
\end{equation}
Note that $\mathcal{L}_{i}\in\partial_{e}\bar{C}$ almost surely since
$Y_{\cdot}$ is cadlag. Also note that the $L_{i}$ are \emph{not}
stopping times. However, if we let $\sigma_{y}(z)=P_{y}[Y_{t^{\star}}=z|H_{\bar{C}}>t^{\star}]$
(cf. \prettyref{eq:QSConvToQuasiStat}) we have
\begin{lem}
For any $k\ge1,x\in\mathbb{T}_{N},y\in\partial_{e}\bar{C},z\in\bar{C}^{c},$
and $E,F\subset\Gamma(\mathbb{T}_{N})$ measurable 
\begin{equation}
P_{x}[Y_{\cdot\wedge L_{k}}\in E,\mathcal{L}_{k}=y,\mathcal{U}_{k}=z,Y_{U_{k}+\cdot}\in F]=P_{x}[Y_{\cdot\wedge L_{k}}\in E,\mathcal{L}_{k}=y]\sigma_{y}(z)P_{z}[F].\label{eq:EIIDEKindOfMarkov}
\end{equation}
\end{lem}
\begin{proof}
Let $E'=E\cap\{w:w\mbox{ constant eventually},w(\infty)=y\}$ (see
below \prettyref{eq:DefOfBoundary} for the notation) and $F'=F\cap\{w:w(0)=z\}$,
so that the left-hand side of \eqref{eq:EIIDEKindOfMarkov} equals
\begin{equation}
P_{x}[Y_{\cdot\wedge L_{k}}\in E',Y_{U_{k}+\cdot}\in F']=\dsl_{i\ge1}P_{x}[Y_{\cdot\wedge\tau_{i}}\in E',L_{k}=\tau_{i},Y_{\tau_{i}+t^{\star}+\cdot}\in F'],\label{eq:Decomp}
\end{equation}
(note that $U_{i}\overset{\eqref{eq:DefOfUandtStar}}{=}L_{i}+t^{\star}$,
and recall \prettyref{eq:DefOfJumpTimes}). Now $G=\{Y_{\tau_{i-1}}\in\bar{C},U_{k-1}\le\tau_{i}\le U_{k}\}$
is $\sigma(Y_{\cdot\wedge\tau_{i}})-$measurable, and $G\cap\{H_{\bar{C}}\circ\theta_{\tau_{i}}>t^{\star}\}=\{L_{k}=\tau_{i}\}$,
so that by the strong Markov property applied at time $\tau_{i}$
and the definition of $E'$ we have
\[
P_{x}[Y_{\cdot\wedge\tau_{i}}\in E',L_{k}=\tau_{i},Y_{\tau_{i}+t^{\star}+\cdot}\in F']=P_{x}[\{Y_{\cdot\wedge\tau_{i}}\in E'\}\cap G]P_{y}[H_{\bar{C}}>t^{\star},Y_{t^{\star}+\cdot}\in F'].
\]
But $P_{y}[Y_{t^{\star}+\cdot}\in F',H_{\bar{C}}>t^{\star}]=P_{y}[H_{\bar{C}}>t^{\star}]\sigma_{y}(z)P_{z}[F]$
(by the Markov property applied at time $t^{\star}$ and the definitions
of $\sigma_{y}$ and $F'$) so in fact 
\begin{align*}
P_{x}[Y_{\cdot\wedge\tau_{i}}\in E',L_{k}=\tau_{i},Y_{\tau_{i}+t^{\star}+\cdot}\in F']  = & \;P_{x}[\{Y_{\cdot\wedge\tau_{i}}\in E'\}\cap G]P_{y}[H_{\bar{C}}>t^{\star}]\sigma_{y}(z)P_{z}[F]\\
  = &\; P_{x}[\{Y_{\cdot\wedge\tau_{i}}\in E'\}\cap G\cap\{H_{\bar{C}}\circ\theta_{\tau_{i}}>t^{\star}\}]\sigma_{y}(z)P_{z}[F]\\
  = &\; P_{x}[Y_{\cdot\wedge\tau_{i}}\in E',L_{k}=\tau_{i}]\sigma_{y}(z)P_{z}[F],
\end{align*}
by an application of the strong Markov property and the definition
of $G$. Plugging this into \prettyref{eq:Decomp} and using the definition
of $E'$ gives \prettyref{eq:EIIDEKindOfMarkov}.
\end{proof}
We now continue with the proof of \prettyref{lem:ExcwithIIDExc}.
Because of \prettyref{eq:QSConvToQuasiStat} (together with Proposition
4.7, p. 50 in \cite{LevPerWilMarkovChainsandMixingTimes}) we can
construct for each $y\in\bar{C}^{c}$ a measure $q_{y}(\cdot,\cdot)$
on $\bar{C}^{c}\times\bar{C}^{c}$ coupling $\sigma$ and $\sigma_{y}$,
such that 
\begin{equation}
\mbox{the first marginal is }\sigma(\cdot),\mbox{ the second is }\sigma_{y}(\cdot),\mbox{ and }\dsl_{z\in\mathbb{T}_{N}}q_{y}(z,z)\ge1-ce^{-cN^{c(\varepsilon)}}.\label{eq:CouplingIsGood}
\end{equation}
Let $q_{y}(\cdot|\cdot)$ denote the conditional distribution of the
first argument given the second (note that $\sigma_{y}(z)>0$ for
all $z\in\bar{C}^{c}$, provided $s\ge c(\varepsilon)$ so that $\bar{C}$
consists of disjoint boxes).

We now construct $(\Omega_{3},\mathcal{A}_{3},Q_{3})$ as a space
with the following \emph{mutually independent} families of random
variables
\begin{align}
 & Y_{\cdot}\mbox{ with law }P_{\sigma},\label{eq:DefOfRW}\\
 & (V_{y,z,i})_{y,z\in\bar{C}^{c},i\ge1}\mbox{ independent }\bar{C}^{c}-\mbox{valued, where }V_{y,z,i}\mbox{ has law }q_{y}(dw|z),\label{eq:DefOfS}\\
 & (Z_{\cdot}^{v,i})_{v\in\bar{C}^{c},i\ge1}\mbox{ independent }\Gamma(\mathbb{T}_{N})-\mbox{valued, where }Z_{\cdot}^{v,i}\mbox{ has law }P_{v}[\hat{Y}_{\cdot\wedge L_{1}}^{1}\in dw],\label{eq:DefOfZs}
\end{align}

\smallskip\n
(recall \prettyref{eq:forconvenience} for the definition of $\hat{Y}_{\cdot\wedge L_{1}}^{1}$). We define on $\Omega_{3}$ starting points of excursions $\Sigma_{i}$
by
\begin{equation}
\Sigma_{i}=V_{\mathcal{L}_{i},\mathcal{U}_{i},i}\mbox{ for }i\ge1.\label{eq:DefofSigmas}
\end{equation}
We will see that the $\Sigma_{i}$ are i.i.d. with law $\sigma$,
but coincide with the $\mathcal{U}_{i}$ with high probability. Furthermore
define the excursions $\bar{Y}^{i}$, with starting points $\Sigma_{i-1}$
for $i\ge2$, by
\begin{equation}
\bar{Y}_{\cdot}^{1}=\hat{Y}_{\cdot}^{1}\mbox{ and }\bar{Y}_{\cdot}^{i+1}=\begin{cases}
\hat{Y}_{\cdot}^{i+1} & \mbox{if }\mathcal{U}_{i}=\Sigma_{i}\\
Z_{\cdot}^{\Sigma_{i},i} & \mbox{if }\mathcal{U}_{i}\ne\Sigma_{i}
\end{cases}\mbox{ for }i\ge1.\label{eq:DefOfYBars}
\end{equation}
(By a slight abuse of notation we view $\hat{Y}_{\cdot}^{i},\mathcal{U}_{i},\mathcal{L}_{i}$
and $Y_{\cdot}$ as being defined on $\Omega_{3}$ as well as on $\Gamma(\mathbb{T}_{N})$).
We will see that the $\bar{Y}_{\cdot}^{i},i\ge1,$ are i.i.d. with
law $P_{\sigma}[Y_{\cdot\wedge L_{1}}\in dw]$, essentially because
the starting points $\Sigma_{i}$ are i.i.d. with law $\sigma$. Furthermore
we will see that $\bar{Y}_{\cdot}^{i+1}$ coincides with $\hat{Y}_{\cdot}^{i+1}$
with high probability, because $\Sigma_{i}$ coincides with $\mathcal{U}_{i}$
(the starting point of $\hat{Y}_{\cdot}^{i+1}$) with high probability.
To this end note that for all $i\ge2$ 
\[
Q_{3}[\bar{Y}_{\cdot}^{i}=\hat{Y}_{\cdot}^{i}]\overset{\eqref{eq:DefofSigmas},\eqref{eq:DefOfYBars}}{=}Q_{3}[\mathcal{U}_{i-1}=V_{\mathcal{L}_{i-1},\mathcal{U}_{i-1},i-1}]=\dsl_{y,z\in\mathbb{T}_{N}}Q_{3}[\mathcal{L}_{i-1}=y,\mathcal{U}_{i-1}=z,V_{y,z,i-1}=z].
\]
Now by independence and \eqref{eq:DefOfS} we have $Q_{3}[\mathcal{L}_{i-1}=y,\mathcal{U}_{i-1}=z,V_{y,z,i-1}=z]=P_{\sigma}[\mathcal{L}_{i-1}=y,\mathcal{U}_{i-1}=z]q_{y}(z|z)$.
Furthermore $P_{\sigma}[\mathcal{L}_{i-1}=y,\mathcal{U}_{i-1}=z]=P_{\sigma}[\mathcal{L}_{i-1}=y]\sigma_{y}(z)$
by \eqref{eq:EIIDEKindOfMarkov} so
\begin{equation}
Q_{3}[\bar{Y}_{\cdot}^{i}=\hat{Y}_{\cdot}^{i}]=\dsl_{y\in\mathbb{T}_{N}}\{P_{\sigma}[\mathcal{L}_{i-1}=y]\dsl_{z\in\mathbb{T}_{N}}\sigma_{y}(z)q_{y}(z|z)\}\overset{\eqref{eq:CouplingIsGood}}{\ge}1-ce^{-cN^{c(\varepsilon)}}\mbox{ for all }i\ge1,\label{eq:CouplingIsGood2}
\end{equation}
where we have used that $\sigma_{y}(z)q_{y}(z|z)=q_{y}(z,z)$ and
that $\bar{Y}_{\cdot}^{1}=\hat{Y}_{\cdot}^{1}$ almost surely.

The next lemma will be used to to show that the $\bar{Y}_{\cdot}^{i}$
are i.i.d. with law $P_{\sigma}[\hat{Y}_{\cdot\wedge L_{1}}^{1}\in dw]$.
\begin{lem}
\label{lem:ApplicationOfAlmostMarkov}For any measurable $E_{1},....,E_{k}\subset\Gamma(\mathbb{T}_{N})$,
let $F_{k}=\{\bar{Y}_{\cdot}^{i}\in E_{i},1\le i\le k\}$. Then for
all $y,z\in\bar{C}^{c}$ and measurable $F\subset\Gamma(\mathbb{T}_{N})$
we have 
\begin{equation}
Q_{3}[F_{k}\cap\{\mathcal{L}_{k}=y,\mathcal{U}_{k}=z,Y_{U_{k}+\cdot}\in F\}]=Q_{3}[F_{k}\cap\{\mathcal{L}_{k}=y\}]\sigma_{y}(z)P_{z}[F],\label{eq:ApplicationOfAlmostMarkov}
\end{equation}
\end{lem}
\begin{proof}
Essentially speaking \eqref{eq:ApplicationOfAlmostMarkov} follows
directly from \eqref{eq:EIIDEKindOfMarkov} because the event $F_{k}$
only depends on $Y_{\cdot\wedge L_{k}}$, $Z_{\cdot}^{v,i}$ and $V_{y,z,i}$
for $i\le k$ (see \eqref{eq:DefOfYBars}), where the $Z_{\cdot}^{v,i}$
and $V_{y,z,i}$ are independent of $Y_{U_{k}+\cdot}$ by construction.
We omit the details (which involve ``conditioning on $\mathcal{L}_{i},\mathcal{U}_{i},\Sigma_{i},i\le k-1$'',
i.e. considering $Q_{3}[F_{k}\cap\{\mathcal{L}_{k}=y,\mathcal{U}_{k}=z,Y_{U_{k}+\cdot}\in F\}\cap K(\bar{y},\bar{z},\bar{v})]$,
where $K(\bar{y},\bar{z},\bar{v})=\{(\mathcal{L}_{i})_{i=1}^{k-1}=\bar{y},(\mathcal{U}_{i})_{i=1}^{k-1}=\bar{z},(V_{\bar{y}_{i},\bar{z}_{i},i})_{i=1}^{k-1}=\bar{v}\}$
for vectors $\bar{y},\bar{z},\bar{v}$ in $(\bar{C}^{c})^{k-1}$).
\end{proof}
We now continue with the proof of \prettyref{lem:ExcwithIIDExc}
by showing that the $\bar{Y}_{\cdot}^{i}$ are i.i.d with law $P_{\sigma}[Y_{\cdot\wedge L_{1}}\in dw]$.
For any measurable $E_{1},...,E_{k},E_{k+1}\subset\Gamma(\mathbb{T}_{N})$
let $F_{k}$ be defined as in \prettyref{lem:ApplicationOfAlmostMarkov},
let $F=\{Y_{\cdot\wedge L_{1}}\in E_{k+1}\}$ and note that by \prettyref{eq:DefOfYBars},
$Q_{3}[F_{k}\cap\{\bar{Y}_{\cdot}^{k+1}\in E_{k+1}\}]$ equals
\begin{align}
 & \dsl_{y,z\in\bar{C}^{c}}Q_{3}[F_{k}\cap\{\mathcal{L}_{k}=y,\mathcal{U}_{k}=z,V_{y,z,k}=z,Y_{U_{k}+\cdot}\in F\}]\label{eq:FirstSum}\\
 & +\dsl_{y,z,v\in\bar{C}^{c},v\ne z}Q_{3}[F_{k}\cap\{\mathcal{L}_{k}=y,\mathcal{U}_{k}=z,V_{y,z,k}=v,Z_{\cdot}^{v,k}\in F\}].\label{eq:SecondSum}
\end{align}
By independence and \eqref{eq:DefOfS} we have that the probability
in \prettyref{eq:FirstSum} equals
\[
Q_{3}[F_{k}\cap\{\mathcal{L}_{k}=y,\mathcal{U}_{k}=z,Y_{U_{k}+\cdot}\in F\}]q_{y}(z|z)\overset{\eqref{eq:ApplicationOfAlmostMarkov}}{=}Q_{3}[F_{k}\cap\{\mathcal{L}_{k}=y\}]\sigma_{y}(z)P_{z}[F]q_{y}(z|z),
\]
and similarly by independence, \eqref{eq:DefOfS} and \prettyref{eq:DefOfZs},
the probability in \prettyref{eq:SecondSum} equals
\[
Q_{3}[F_{k}\cap\{\mathcal{L}_{k}=y,\mathcal{U}_{k}=z\}]q_{y}(v|z)P_{v}[F]\overset{\eqref{eq:ApplicationOfAlmostMarkov}}{=}Q_{3}[F_{k}\cap\{\mathcal{L}_{k}=y\}]\sigma_{y}(z)q_{y}(v|z)P_{v}[F].
\]

\smallskip\n
But $\sigma_{y}(z)P_{z}[F]q_{y}(z|z)=q_{y}(z,z)P_{z}[F]$ and $\sigma_{y}(z)q_{y}(v|z)P_{v}[F]=q_{y}(v,z)P_{v}[F]$,
so 
\begin{eqnarray*}
Q_{3}[F_{k}\cap\{\bar{Y}_{\cdot}^{k+1}\in E_{k+1}\}] & = & \dsl_{y,z,v\in\bar{C}^{c}}Q_{3}[F_{k}\cap\{\mathcal{L}_{k}=y\}]q_{y}(v,z)P_{v}[F]\\
 & = & Q_{3}[F_{k}]P_{\sigma}[F]=Q_{3}[F_{k}]P_{\sigma}[Y_{\cdot\wedge L_{1}}\in E_{k+1}],
\end{eqnarray*}
where we have used \eqref{eq:CouplingIsGood} and the definition of
$F$. But applying this recursively we get that for all $k\ge1$ and
measurable $E_{1},...,E_{k+1}\subset\Gamma(\mathbb{T}_{N})$
\[
Q_{3}[\bar{Y}_{\cdot}^{i}\in E_{i},1\le i\le k+1]=\mbox{\f $\dis\prod_{i=1}^{k+1}$} P_{\sigma}[Y_{\cdot\wedge L_{1}}\in E_{i}],
\]
and thus the $\bar{Y}_{\cdot}^{i},i\ge1,$ are i.i.d. with the same
law as $Y_{\cdot\wedge L_{1}}$ under $P_{\sigma}$.

\smallskip
We can now extend the space $(\Omega_{3},\mathcal{A}_{3},Q_{3})$
by independently ``appending a piece'' with law $P_{y}[Y_{\cdot\wedge t^{\star}}\in dw|H_{\bar{C}}>t^{\star}]$
to each $\bar{Y}_{\cdot}^{i}$, conditionally on the event $\{\bar{Y}_{\infty}^{i}=y\}$
for every $y\in\partial_{e}\bar{C}$, to obtain i.i.d. processes $\tilde{Y}_{\cdot}^{1},\tilde{Y}_{\cdot}^{2},...$
with law $P_{\sigma}[Y_{\cdot\wedge U_{1}}\in dw]$ such that $\tilde{Y}_{\cdot}^{i}(0,\infty)\cap\bar{C}=\bar{Y}_{\cdot}^{i}(0,\infty)\cap\bar{C}$
almost surely. Then \eqref{eq:ExcwithIIDExc} is satisfied by \eqref{eq:CouplingIsGood2}
and since $\hat{Y}_{\cdot}^{i}(0,\infty)\cap\bar{C}=Y(U_{i-1},U_{i})\cap\bar{C}$.
This completes the proof of \prettyref{lem:ExcwithIIDExc}.
\end{proof}
The next step is to to relate the stopping times $U_{k}$ to deterministic times.
\begin{lem} \label{lem:ETRTExcursionTimeandRealTime}($d\ge3$,$N\ge3$) 

\medskip
If $s\ge c(\varepsilon)$, $u\ge s^{-c(\varepsilon)}$, $\frac{1}{2}\ge\delta\ge cs^{-c(\varepsilon)}$ and $n\le s^{c(\varepsilon)}$ then
\begin{eqnarray}
P_{\sigma}[U_{[u(1+\delta)n\,{\rm cap}(A)]}\le uN^{d}] & \le & ce^{-s^{c(\varepsilon)}}\mbox{ and }\label{eq:ETRTUpperBound}\\
P_{\sigma}[U_{[u(1-\delta)n\,{\rm cap}(A)]}\ge uN^{d}] & \le & ce^{-s^{c(\varepsilon)}}.\label{eq:ETRTLowerBound}
\end{eqnarray}
\end{lem}

\begin{proof}
Note that 
\begin{equation}
U_{k}=\dsl_{i=0}^{k-1}H_{\bar{A}}\circ\theta_{U_{i}}+\dsl_{i=1}^{k}U\circ\theta_{R_{i}}\mbox{ for all }k\ge0.\label{eq:ETRTRDecomp}
\end{equation}
Let $k^{+}=[u(1+\delta)n\,{\rm cap}(A)]$ and $k^{-}=[u(1-\delta)n\,{\rm cap}(A)]$.
By \prettyref{eq:ETRTRDecomp} both \prettyref{eq:ETRTUpperBound}
and \prettyref{eq:ETRTLowerBound} follow from (note that $\delta uN^c,\delta^{2}un\mbox{cap}(A)\overset{\prettyref{eq:AsymptoticsOfCapOfBox}}{\ge}s^{c(\varepsilon)}$
if $u,\delta\ge s^{-c(\varepsilon)}$) 
\begin{eqnarray}
P_{\sigma}\Big[\dsl_{i=0}^{k^{+}-1}H_{\bar{A}}\circ\theta_{U_{i}}\le uN^{d}\Big] & \le & ce^{-c\delta^{2}un\mbox{cap}(A)},\label{eq:ETRTHA1}\\
P_{\sigma}\Big[\dsl_{i=0}^{k^{-}-1}H_{\bar{A}}\circ\theta_{U_{i}}\ge u(1-\frac{\delta}{2})N^{d}\Big] & \le & e^{-c\delta^{2}un\mbox{cap}(A)},\label{eq:ETRTHA2}\\
P_{\sigma}\Big[\dsl_{i=1}^{k^{-}}U\circ\theta_{R_{i}}\ge\frac{\delta}{2}uN^{d}\Big] & \le & ce^{-c\delta uN^c}.\label{eq:ETRTU}
\end{eqnarray}
One shows \prettyref{eq:ETRTHA1}, \prettyref{eq:ETRTHA2} and \prettyref{eq:ETRTU}
using large deviations bounds. Since the proofs are very similar to
those in Lemma 4.3 in \cite{TeixeiraWindischOnTheFrag} we omit the
details. Let us simply state that to prove \prettyref{eq:ETRTHA1}
one estimates the small exponential moments of $\delta\frac{\sum_{i=0}^{k^{+}-1}H_{\bar{A}}\circ\theta_{U_{i}}}{\inf_{x\notin\bar{C}}E_{x}[H_{\bar{A}}]}$
and to prove \prettyref{eq:ETRTHA2} one estimates the small exponential
moments of $\delta\frac{\sum_{i=0}^{k^{-}-1}H_{\bar{A}}\circ\theta_{U_{i}}}{\sup_{x\in\mathbb{T}_{N}}E_{x}[H_{\bar{A}}]}.$
This is possible because by the strong Markov property $E_{\sigma}[\exp(\lambda\sum_{i=0}^{k}H_{\bar{A}}\circ\theta_{U_{k}})]\le(\sup_{x\notin\bar{C}}E_{x}[\exp(\lambda H_{\bar{A}})])^{k}$
for all $\lambda\in\mathbb{R}$, and by elementary bounds on the function
$x\rightarrow e^{x}$ and Khasminskii\textquoteright{}s lemma (Lemma
3.8 in \cite{TeixeiraWindischOnTheFrag}) we have
\begin{align}
 & \sup_{x\notin\bar{C}}E_{x}[\exp(\lambda H_{\bar{A}})]\le1+\lambda\inf_{x\notin\bar{C}}E_{x}[H_{\bar{A}}]+c\lambda^{2}(\sup_{x\in\mathbb{T}_{N}}E[H_{\bar{A}}])^{2}\mbox{ for }\lambda<0\label{eq:TheFirstIneq}\\
 & \mbox{and }\sup_{x\notin\bar{C}}E_{x}[\exp(\lambda H_{\bar{A}})]\le\sum_{m\ge0}\lambda^{m}(\sup_{x\in\mathbb{T}_{N}}E_{x}[H_{\bar{A}}])^{m}\mbox{ for }\lambda>0.\label{eq:TheSecondIneq}
\end{align}
To show \eqref{eq:ETRTHA1} one sets $\lambda=-\frac{c\delta}{\inf_{x\notin\bar{C}}E_{x}[H_{\bar{A}}]}$
in \eqref{eq:TheFirstIneq}, for a small enough constant $c$, uses
\prettyref{eq:ETTStartingFromPointandUniformComparison} to show that
the term involving $\lambda^{2}$ is at most $c\delta^{2}$ (for a
small enough constant $c$) and \prettyref{eq:ETTEntranceTimeLowerBound}
and \prettyref{eq:ETTStartingFromPointandUniformComparison} to show
that $\frac{N^{d}}{\inf_{x\notin\bar{C}}E_{x}[H_{\bar{A}}]}\le(1+c\delta)n\mbox{cap}(A)$,
for any small constant $c$, as long as we require $\delta\ge c(\varepsilon)s^{-c(\varepsilon)}$.
To show \eqref{eq:ETRTHA2} one sets $\lambda=\frac{c\delta}{\sup_{x\in\mathbb{T}_{N}}E_{x}[H_{\bar{A}}]}>0$
in \eqref{eq:TheSecondIneq}, for a small enough constant $c$, and
uses \prettyref{eq:ETTEntraceTimeUpperBound} and \prettyref{eq:ETTStartingFromPointandUniformComparison}
to show that $\frac{N^{d}}{\sup_{x\in\mathbb{T}_{N}}E_{x}[H_{\bar{A}}]}\ge(1-c\delta)n\mbox{cap}(A)$,
for any small constant $c$, as long as we require $\delta\ge c(\varepsilon)s^{-c(\varepsilon)}$.
(Note that \prettyref{eq:ETTEntraceTimeUpperBound} and \prettyref{eq:ETTStartingFromPointandUniformComparison}
hold because we require $n\le s^{c(\varepsilon)}$.)

To prove \prettyref{eq:ETRTU} one estimates $E_{\sigma}[\exp(\lambda\sum_{i=1}^{k^{-}}U\circ\theta_{R_{i}})]$
for $\lambda=(t^{\star})^{-1}$, first by similarly bounding it above
by $(\sup_{x\in\bar{C}}E_{x}[\exp(\lambda U)])^{k^{-}}$. By noting
that if $U\le T_{\bar{D}}+t^{\star}$ does not hold, then $H_{\bar{C}}\circ\theta_{T_{\bar{D}}}\le t^{\star}$
and $U\le U\circ\theta_{H_{\bar{C}}}\circ\theta_{T_{\bar{D}}}+T_{\bar{D}}+t^{\star}$
(recall \prettyref{eq:DefOfUandtStar} and \eqref{eq:DefOfD}), we obtain the inequality

\begin{equation}
\sup_{x\in\bar{C}}E_{x}[\exp(\lambda U)]\le\sup_{x\in\bar{C}}E_{x}[\exp(\lambda(T_{\bar{D}}+t^{\star})]\left(1+\sup_{x\in\bar{C}}E_{x}[\exp(\lambda U)]\sup_{x\notin\bar{D}}P_{x}[H_{\bar{C}}<U]\right).\label{eq:expUupperbound}
\end{equation}
Using once again Khasminskii's lemma and the elementary  $\sup_{x\in\mathbb{T}_{N}}E_{x}[T_{\bar{D}}]=\sup_{x\in\mathbb{T}_{N}}E_{x}[T_{D}]\le cs^{2(1-\frac{\varepsilon}{8})}\le ct^{\star}$
(recall that $D$ has radius $s^{1-\frac{\varepsilon}{8}}$, $s\le N$
and \prettyref{eq:DefOfUandtStar})) one obtains $\sup_{x\in\bar{C}}E_{x}[\exp(\lambda(T_{\bar{D}}+t^{\star})]\le e^{c\lambda t^{\star}}=c$,
since $\lambda t^{\star}=1$. Using this inequality together with
\prettyref{eq:CDBallHitting} in \prettyref{eq:expUupperbound}, and rearranging terms, one obtains
that $\sup_{x\in\bar{C}}E_{x}[\exp(\lambda U)]\le e^{c\lambda t^{\star}}$
(provided $s\ge c(\varepsilon)$). To prove \prettyref{eq:ETRTU}
using the exponential Chebyshev one must then check that $ct^{\star}k^{-} - \frac{\delta}{2}uN^{d} \le -cu\delta N^d$
for $s\ge c(\varepsilon)$, which follows by noting that $\mbox{cap}(A)\le N^{(1-\varepsilon)(d-2)}$ (see \prettyref{eq:AsymptoticsOfCapOfBox}),
using \prettyref{eq:DefOfUandtStar}) and requiring $\delta\ge s^{-c(\varepsilon)}$
and $n\le s^{c(\varepsilon)}$ for small enough exponents $c(\varepsilon)$.
\end{proof}

We can now combine \prettyref{lem:ExcwithIIDExc} and \prettyref{lem:ETRTExcursionTimeandRealTime}
to construct a coupling of a random walk $Z_{\cdot}$ with law $P$
and a sequence of i.i.d. excursions with law $P_{\sigma}[Y_{\cdot\wedge U_{1}}\in dw]$
such that, roughly speaking, $Z(0,uN^{d})\cap\bar{A}$ coincides with
the traces of the i.i.d. excursions.
\begin{prop}\label{pro:UnifRWIIDExc}($d\ge3$,$N\ge3$) 

\medskip
If $s\ge c(\varepsilon)$,
$u\ge s^{-c(\varepsilon)}$, $\frac{1}{2}\ge\delta\ge cs^{-c(\varepsilon)}$
and $n\le s^{c(\varepsilon)}$, then we can construct a coupling $(\Omega_{4},\mathcal{A}_{4},Q_{4})$
of a process $Z_{\cdot}$ with law $P$ and an i.i.d. sequence $\tilde{Y}_{\cdot}^{1},\tilde{Y}_{\cdot}^{2},...$
with law $P_{\sigma}[Y_{\cdot\wedge U_{1}}\in dw]$ such that
\begin{equation}
\begin{array}{rc}
Q_{4}[\overset{[u(1-\delta)n\mbox{cap}(A)]}{\underset{i=2}{\cup}}\tilde{Y}^{i}(0,\infty)\cap\bar{C}\subset Z(0,uN^{d})\cap\bar{C}\subset\overset{[u(1+\delta)n\,{\rm cap}(A)]}{\underset{i=1}{\bigcup}}\tilde{Y}^{i}(0,\infty)\cap\bar{C}]\\
\ge1-c(\varepsilon)ue^{-cs^{c(\varepsilon)}}.
\end{array}\label{eq:UnifRWIIDExc}
\end{equation}
(Note that the first union is over $i\ge2$, see the remark after
the proof.)\end{prop}
\begin{proof}
We first use \prettyref{lem:ExcwithIIDExc} to construct the space
$(\Omega_{3},\mathcal{A}_{3},Q_{3})$. We will now extend it to get
$(\Omega_{4},\mathcal{A}_{4},Q_{4})$. By \prettyref{eq:TorusMixing}
(with $\lambda=N^{-2}t^{\star}\overset{\eqref{eq:DefOfUandtStar}}{=}N^{\frac{\varepsilon}{100}}$)
and a standard coupling argument we can construct a process $Z_{\cdot}$
with law $P$ such that $Z_{\cdot}$ agrees with $Y_{t^{\star}+\cdot}$
with probability at least $1-ce^{-N^{c(\varepsilon)}}$, and in particular
\begin{equation}
Q_{4}[Z(0,uN^{d})=Y(t^{\star},uN^{d}+t^{\star})]\ge1-ce^{-cN^{c(\varepsilon)}}.\label{eq:URWIIDExcZY}
\end{equation}

\smallskip\n
Now letting $k^{-}=[u(1-\delta)n\mbox{cap}(A)]$ and $k^{+}=[u(1+\delta)n\mbox{cap}(A)]$
we have
\begin{equation}
Q_{4}[Y(U_{1},U_{k^{-}})\subset Y(t^{\star},uN^{d}+t^{\star})\subset Y(0,U_{k^{+}})]\ge1-ce^{-s^{c}},\label{eq:UEQIIDExcYU}
\end{equation}
since $U_{1}\overset{\eqref{eq:DefOfRkUk}}{\ge}U\circ\theta_{R_{1}}\overset{\eqref{eq:DefOfUandtStar}}{\ge}t^{\star}$,
$Q_{4}[U_{k^{-}}\le uN^{d}]\ge1-ce^{-s^{c}}$ by \prettyref{eq:ETRTLowerBound},
and $uN^{d}+t^{\star}\overset{u,\delta\ge N^{-c}}{\le}N^d u(1+\frac{\delta}{4})<U_{[u(1+\frac{\delta}{4})^{2}n\,{\rm cap}(A)]}\le U_{k^{+}}$
with probability at least $1-ce^{-s^{c}}$ by \prettyref{eq:ETRTUpperBound}
(applied with $u(1+\frac{\delta}{4})$ in place of $u$ and $\frac{\delta}{4}$
in place of $\delta$). Finally by \prettyref{eq:ExcwithIIDExc} we have
\[
Q_{4}\Big[\,\mbox{\f $\dis\bigcup\limits^{k^{-}}_{i=2}$} \,\tilde{Y}^{i}(0,\infty)\cap\bar{C}=Y(U_{1},U_{k^{-}})\cap\bar{C},Y(0,U_{k^{+}})\cap\bar{C}=\mbox{\f $\dis\bigcup\limits^{k^{+}}_{i=1}$}\, \tilde{Y}^{i}(0,\infty)\cap\bar{C}\Big]\ge1-cue^{-cs^{c(\varepsilon)}},
\]
where we have used that $k^{+}\le cun\mbox{cap}(A)\overset{\prettyref{eq:AsymptoticsOfCapOfBox}}{\le}cus^{c(\varepsilon)}s^{(1-\varepsilon)(d-2)}\le cus^{c(\varepsilon)}$.
Combining this with \prettyref{eq:URWIIDExcZY} and \prettyref{eq:UEQIIDExcYU}
now gives \eqref{eq:UnifRWIIDExc} (using also that $u\ge s^{-c(\varepsilon)}$).
\end{proof}
Note that to ensure that also the first excursion $\tilde{Y}_{\cdot}^{1}$
has law $P_{\sigma}$, we have generated the law $P$ of $Z$ by modifying
$Y_{t^{\star}+\cdot}$ under $P_{\sigma}$, and getting the i.i.d.
excursions $\tilde{Y}_{\cdot}^{i},i\ge1,$ from $Y_{\cdot}$ via \prettyref{lem:ExcwithIIDExc}.
The first union is \prettyref{eq:UnifRWIIDExc} is over $i\ge2$,
since with this construction $\tilde{Y}_{\cdot}^{1}$ corresponds
to a piece of $Y_{\cdot}$ that is not fully included in $Z_{\cdot}$.

We are ready to finish the proof of \prettyref{pro:CoupleRWandPPoE}
using the previous \prettyref{pro:UnifRWIIDExc}.
\begin{proof}[Proof of \prettyref{pro:CoupleRWandPPoE}]
We apply \prettyref{pro:UnifRWIIDExc} with $\frac{\delta}{4}\ge cs^{-c(\varepsilon)}$
in place of $\delta$ to construct the space $(\Omega_{4},\mathcal{A}_{4},Q_{4})$
which we will extend to get $(\Omega_{2},\mathcal{A}_{2},Q_{2})$.
First of all we rename the process $Z_{\cdot}$ to $Y_{\cdot}$, so
that we have from \prettyref{eq:UnifRWIIDExc} that
\begin{equation}
\begin{array}{rc}
Q_{4}\Big[\, \mbox{\f $\dis\bigcup\limits^{[u(1-\frac{\delta}{4})n\,{\rm cap}(A)]}_{i=2}$}\tilde{Y}^{i}(0,\infty)\cap\bar{C}\subset Y(0,uN^{d})\cap\bar{C}\subset\mbox{\f $\dis\bigcup\limits^{[u(1+\frac{\delta}{4})n\mbox{cap}(A)]}_{i=1}$}\tilde{Y}^{i}(0,\infty)\cap\bar{C}]\\
\ge1-c(\varepsilon)ue^{-cs^{c(\varepsilon)}},
\end{array}\label{eq:RWPPOEFirstStep}
\end{equation}
where $\tilde{Y}_{1},\tilde{Y}_{2},...$ are i.i.d. with law $P_{\sigma}[Y_{\cdot\wedge U_{1}}\in dw]$.
We now add independent Poisson random variables $J_{1}$ and $J_{2}$
to the space, where $J_{1}$ has parameter $u(1-\frac{\delta}{2})n\mbox{cap}(A)$,
$J_{2}$ has parameter $u\delta n\mbox{cap}(A)$ and define the following
point processes on $\Gamma(\mathbb{T}_{N})$
\begin{equation}
\mu_{1}^{'}=\dsl_{i=2}^{J_{1}+1}\delta_{\tilde{Y}_{H_{\bar{A}}+\cdot}^{i}}\mbox{ and }\mu_{2}^{'}=1_{\{J_{2}\ne0\}}\delta_{\tilde{Y}_{H_{\bar{A}}+\cdot}^{1}}+\dsl_{i=J_{1}+2}^{J_{1}+J_{2}}\delta_{\tilde{Y}_{H_{\bar{A}}+\cdot}^{i}}.\label{eq:RWPPOEDefOfPPPs}
\end{equation}
Then $\mu_{1}^{'}$ and $\mu_{2}^{'}$ are independent Poisson point
processes such that $\mu_{1}^{'}$ has intensity $u(1-\frac{\delta}{2})n\mbox{cap}(A)P_{q}[Y_{\cdot\wedge U}\in dw]$
and $\mu_{2}^{'}$ has intensity $u\delta n\mbox{cap}(A)P_{q}[Y_{\cdot\wedge U}\in dw]$,
where $q$ denotes the measure defined by $q(x)=P_{\sigma}[Y_{H_{\bar{A}}}=x]$.
By a standard large deviations bound one can show that 
\[
Q_{4}[J_{1}+1\le\Big[u\Big(1-\mbox{\f $\dis\frac{\delta}{4}$}\Big)n {\rm cap}(A) \Big]\le\Big[u\Big(1+\mbox{\f $\dis\frac{\delta}{4}$}\Big)n{\rm cap}(A)\Big]\le J_{1}+J_{2}]\ge1-ce^{-cun\mbox{cap}(A)\delta^{2}},
\]
so that from $u\mbox{cap}(A)\delta^{2}\overset{\prettyref{eq:AsymptoticsOfCapOfBox},u,\delta\ge s^{-(1-\varepsilon)\frac{d-2}{8}}}{\ge}s^{(1-\varepsilon)\frac{d-2}{2}}$,
\prettyref{eq:RWPPOEFirstStep} and \prettyref{eq:RWPPOEDefOfPPPs}
it follows that
\begin{equation}
Q_{4}[\mathcal{I}(\mu_{1}^{'})\cap\bar{A}\subset Y(0,uN^{d})\cap\bar{A}\subset\mathcal{I}(\mu_{1}^{'}+\mu_{2}^{'})\cap\bar{A})\ge1-cue^{-cs^{c(\varepsilon)}}.\label{eq:RWPPOEResultPrime}
\end{equation}
By \prettyref{eq:QSEQQuasiStatEquilMeas} we have the following inequality
involving the intensity of $\mu_{1}^{'}$
\begin{equation}
u(1-\delta)\kappa_{1}\overset{\delta\ge cs^{-c(\varepsilon)}}{\le}u\Big(1-\mbox{\f $\dis\frac{\delta}{2}$}\Big)n\,{\rm cap}(A)P_{q}[Y_{\cdot\wedge U}\in dw]\overset{\delta\ge cs^{-c(\varepsilon)}}{\le}u\Big(1-\mbox{\f $\dis\frac{\delta}{3}$}\Big)\kappa_{1}.\label{eq:FirstIntensitySandwich}
\end{equation}
Using the lower bound we can thus construct, by means of a standard thinning procedure, a Poisson point process $\mu_{1}$ of intensity $u(1-\delta)\kappa_{1}$, such that $\mu_{1}\le\mu_{1}^{'}$
and $\mu_{1}$ and $\mu_{1}^{'}-\mu_{1}$ are independent (by placing each point $x \in \Gamma(\mathbb{T}_n)$ of $\mu_1'$ in $\mu_1$ independently with
probability $\frac{u(1-\delta)\kappa_1(x)}{u(1-\delta/2)n{\rm cap}(A)P_{q}[Y_{\cdot\wedge U} = x]} \in [0,1]$, where we extend the space with the appropriate Bernoulli random variables).
Using the upper bound we can furthermore thicken $\mu_{1}^{'}-\mu_{1}$ to get
a Poisson point process $\nu$ (independent of $\mu_{1}$) of intensity
$u\frac{2}{3}\delta\kappa_{1}$ such that $\mu_{1}^{'}\le\mu_{1}+\nu$
(by extending the space with with an independent Poisson point process of intensity
$u(1-\delta/3)\kappa_1 - u(1-\delta/2)n{\rm cap}(A)P_q[Y_{\cdot \wedge U} \in dw] \ge 0$ and adding this process
to $\mu_1'-\mu_1$ to form $\nu$).

Once again by \prettyref{eq:QSEQQuasiStatEquilMeas} we also have
the following inequality involving the intensity of $\mu_{2}^{'}$
\[
u\delta n\,{\rm cap}(A)P_{q}[Y_{\cdot\wedge U}\in dw]\overset{s\ge c(\varepsilon)}{\le}u\, \mbox{\f $\dis\frac{4}{3}$}\,  \delta\kappa_{1}.
\]
Thus we can (as above) thicken $\mu_{2}^{'}$ to get a Poisson point process
$\eta$ of intensity $u\frac{4}{3}\delta\kappa_{1}$, such that $\mu_{2}^{'}\le\eta$
and $\mu_{1},\nu,\eta$ are independent. We then define $\mu_{2}=\nu+\eta$,
and see that $\mu_{2}$ is a Poisson point process of intensity $u2\delta\kappa_{1}$
which is independent from $\mu_{1}$, and $\mu_{1}\le\mu_{1}^{'}\le\mu_{1}^{'}+\mu_{2}^{'}\le\mu_{1}+\mu_{2}$.
Thus it follows from \prettyref{eq:RWPPOEResultPrime} that the probability
of the event in \prettyref{eq:CouplingInformalFirstInclusion} is
at least $1-cue^{-s^{c(\varepsilon)}}$ and the proof of \prettyref{pro:CoupleRWandPPoE}
is complete.
\end{proof}
To prove \prettyref{thm:Coupling} it now remains to show \prettyref{pro:CouplePPoEandUFreePPoE}
and \prettyref{pro:CouplePPoEandRI}.

\section{\label{sec:OutOfTorus}From the torus to $\mathbb{Z}^{d}$, and decoupling
boxes.}

In this section we prove \prettyref{pro:CouplePPoEandUFreePPoE} which
dominates a Poisson point process $\nu$ of intensity a multiple of
$\kappa_{1}$ (and whose ``excursions'' therefore roughly speaking
``feel that they are in the torus'', and may visit several of the
boxes $A_{1},...,A_{n}$, see \eqref{eq:DefOfUandtStar} and \eqref{eq:DefOfKappa1}),
from above and below by Poisson point processes whose intensities
are multiples of $\kappa_{2}$ (and whose ``excursions'' thus, conditionally
on their starting point, ``behave'' like random walk in $\mathbb{Z}^{d}$
stopped upon leaving a box, and visit only a single box, see \eqref{eq:DefOfKappa2-1}).
First we will (roughly speaking) take all excursions in $\nu$ that
never return to $\bar{A}$ after leaving $\bar{B}$ (the great majority
of all excursions), truncate them upon leaving $\bar{B}$, and collect
them into a Poisson point process whose intensity can bounded from
above and below by multiples of $\kappa_{2}$ (see \prettyref{lem:KappaOneIntensitySandwich}).
This will allow us to dominate this Poisson point process, from above
and below, by Poisson point processes with intensities a multiple
of $\kappa_{2}$. We will then, in \prettyref{lem:Mu234etc}, use
an argument from the proof of Theorem 2.1 in \cite{SznitmanDecoupling}
to dominate the excursions that do return to $\bar{A}$ after leaving
$\bar{B}$ by a Poisson point process with intensity a small multiple
of $\kappa_{2}$. That is, we will, essentially speaking, decouple
the ``successive visits to $\bar{A}$ (after leaving $\bar{B}$)''
of a single excursion by dominating the hitting distribution on $\bar{A}$
when starting outside of $\bar{B}$ by a multiple of the measure $e$
from \prettyref{eq:DefOfKappa1} (see \prettyref{lem:CompoundPoissionIntMeasIneq}).
We then collect a number of the successive visits of all the excursions
(with high probability \emph{all the successive visits} of \emph{all
the excursions}, and in addition a number of extra independent visits)
into a Poisson point process of intensity a small multiple of $\kappa_{2}$.
This is done in \prettyref{lem:Mu234etc}. Though the number of excursions
that make several ``successive returns'' to $\bar{A}$ are small,
dominating them with high enough probability (namely, stretched exponential
in the separation $s$, see the statement of \prettyref{pro:CouplePPoEandUFreePPoE}
and \eqref{eq:Mu234etc}), so that what happens in the individual
boxes $A_{1},A_{2},...,A_{n},$ is independent, is not straight-forward.
Since \prettyref{lem:Mu234etc} achieves this, it should be considered
the heart of the proof of \ref{pro:CouplePPoEandUFreePPoE}.

We recall the standing assumption \prettyref{eq:StandingSeparationAssumption}.
Define for $w\in\Gamma(\mathbb{T}_{N})$ the successive returns $\hat{R}_{k}=\hat{R}_{k}(w)$
to $\bar{A}$ and departures $\hat{D}_{k}=\hat{D}_{k}(w)$ from $\bar{B}$
as follows
\begin{equation}
\hat{R}_{1}=H_{\bar{A}},\mbox{ }\hat{R}_{k}=H_{\bar{A}}\circ\theta_{\hat{D}_{k-1}}+\hat{D}_{k-1},k\ge2,\mbox{ }\hat{D}_{k}=T_{\bar{B}}\circ\theta_{\hat{R}_{k}}+\hat{R}_{k},k\ge1.\label{eq:RetDepABTorus}
\end{equation}
Note that $\hat{R}_{k}$ should not be confused with the $R_{k}$
used in \prettyref{sec:Poissonization} (see \prettyref{eq:DefOfRkUk}).
To extract the ``successive visits to $\bar{A}$'' of an excursion
we furthermore define for each $i\ge1$ the map $\phi_{i}$ from $\Gamma(\mathbb{T}_{N})$
into $\Gamma(\mathbb{T}_{N})^{i}$ by
\begin{equation}
(\phi_{i}(w))_{j}=w((\hat{R}_{j}+\cdot)\wedge\hat{D}_{j})\mbox{ for }j=1,...,i,w\in\{\hat{R}_{i}<U<\hat{R}_{i+1}\}\subset\Gamma(\mathbb{T}_{N}),i\ge1.\label{eq:PhiExcursionExtractor}
\end{equation}
For each $i\ge1$ we will apply this map to the Poisson point processes
$1_{\{\hat{R}_{i}<U<\hat{R}_{i+1}\}}\nu$ to get a Poisson point processes
$\mu_{i}$ of intensity $u\kappa_{1}^{i}$, where
\begin{equation}
\kappa_{1}^{i}=\phi_{i}\circ(1_{\{\hat{R}_{i}<U<\hat{R}_{i+1}\}}\kappa_{1}),i\ge1.\label{eq:DefOfKappaL}
\end{equation}
To dominate the $\mu_{1}$, the first of these Poisson point processes,
which contains most excursions, from above and below by Poisson point
processes of intensities that are multiples of $\kappa_{2}$ we will
use the following inequality.
\begin{lem}
\label{lem:KappaOneIntensitySandwich}($d\ge3$,$N\ge3$) 

\medskip
If $s\ge c(\varepsilon)$
and $n\le s^{c(\varepsilon)}$ we have $(1-cs^{-c(\varepsilon)})\kappa_{2}\le\kappa_{1}^{1}\le\kappa_{2}.$
\end{lem}
We postpone the proof of \prettyref{lem:KappaOneIntensitySandwich}
until after the proof of \prettyref{pro:CouplePPoEandUFreePPoE}.
The Poisson point processes $\mu_{2},\mu_{3},...$ will contain the
successive visits to $\bar{A}$ of the excursions that make $2,3,...$
such visits, respectively. There will only be a few of these and using
the following lemma we will be able to dominate them by a Poisson
point process $\theta$ of intensity a small multiple of $\kappa_{2}$.
\begin{lem} \label{lem:Mu234etc}$(d\ge3,N\ge3)$ 

\medskip
Let $(\Omega,\mathcal{A},Q)$
be a probability space with independent Poisson point processes $\mu_{2},\mu_{3},...$
such that $\mu_{i}$ has intensity $u\kappa_{1}^{i}$,$u\ge0$. Then
if $1 \ge \delta\ge s^{-c(\varepsilon)}$, $n\le s^{c(\varepsilon)}$ and
$s\ge c(\varepsilon)$ we can construct a space $(\Omega',\mathcal{A}',Q')$
and, on the product space $\Omega\times\Omega'$, a $\sigma(\mu_{i},i\ge2)\times\mathcal{A}'-$measurable
Poisson point process $\theta$ of intensity $u\delta\kappa_{2}$
such that (recalling the notation from \prettyref{eq:DefOfPPPTrace})
\begin{equation}
Q\otimes Q'[\cup_{i\ge2}\mathcal{I}(\mu_{i})\subset\mathcal{I}(\theta)]\ge1-ce^{-cu\delta\mbox{cap(A)}}.\label{eq:Mu234etc}
\end{equation}

\end{lem}
We postpone the proof of \prettyref{lem:Mu234etc} until later, and
instead use it together with \prettyref{lem:KappaOneIntensitySandwich}
to prove \prettyref{pro:CouplePPoEandUFreePPoE}.
\begin{proof}[Proof of \prettyref{pro:CouplePPoEandUFreePPoE}]
We let 
\begin{equation}
\mu_{i}=\phi_{i}(1_{\{\hat{R}_{i}<U<\hat{R}_{i+1}\}}\nu),i\ge1,\mbox{ so that }\cup_{i\ge1}\mathcal{I}(\mu_{i})\cap\bar{A}\overset{\eqref{eq:RetDepABTorus}\eqref{eq:PhiExcursionExtractor}}{=}\mathcal{I}(\nu)\cap\bar{A}.\label{eq:FirstUseOfPhil-1}
\end{equation}
Since the sets $\{\hat{R}_{i}<U<\hat{R}_{i+1}\},i\ge1,$ are disjoint
$\mu_{i},i\ge1,$ are independent Poisson point processes on the respective
spaces $\Gamma(\mathbb{T}_{N})^{i},i\ge1$, and by \prettyref{eq:DefOfKappaL}
they have respective intensities $u\kappa_{1}^{i},i\ge1$. By \prettyref{lem:KappaOneIntensitySandwich}
it follows, since we require $\delta\ge cs^{-c(\varepsilon)}$, that
\begin{equation}
u(1-\delta)\kappa_{2}\le u\kappa_{1}^{1}\le u\kappa_{2}.\label{eq:Kappa1DeltaIntensitySandwich-1}
\end{equation}
Now similarly to how we used \eqref{eq:FirstIntensitySandwich} to
construct the processes $\mu_{1}$ and $\nu$ from $\mu'$, we now
(extending our space appropriately) use \eqref{eq:Kappa1DeltaIntensitySandwich-1}
to construct processes $\nu_{1}$ and $\rho$, such that $\nu_{1},\rho,\mu_{i},i\ge2,$
are independent, $\nu_{1}$ has intensity $u(1-\delta)\kappa_{2}$,
$\rho$ has intensity $u\delta\kappa_{2}$ and 
\begin{equation}
\nu_{1}\le\mu_{1}\le\nu_{1}+\rho\mbox{ almost surely.}\label{eq:PPIneq}
\end{equation}
Thus \prettyref{eq:SecondCouplingASInc} holds, because $\mathcal{I}(\nu_{1})\cap\bar{A}\overset{\eqref{eq:PPIneq}}{\subset}\mathcal{I}(\mu_{1})\cap\bar{A}\overset{\eqref{eq:FirstUseOfPhil-1}}{\subset}\mathcal{I}(\nu)\cap\bar{A}$,
and since $\nu_{1}$ has intensity $u(1-\delta)\kappa_{2}$ it now
suffices to construct $\nu_{2}$ appropriately.

To this end we apply \prettyref{lem:Mu234etc}, once again extending
the space, to get a Poisson point process $\theta$ of intensity $u\delta\kappa_{2}$
such that $\nu_{1},\rho$ and $\theta$ are independent and 
\begin{equation}
Q[\cup_{i\ge2}\mathcal{I}(\mu_{i})\subset\mathcal{I}(\theta)]\ge1-ce^{-cu\delta\,{\rm cap}(A)}\ge1-ce^{-cs^{c(\varepsilon)}},\label{eq:DealingWithIt}
\end{equation}
where we use that we require $u,\delta\ge s^{-c(1-\varepsilon)} = s^{-c(\varepsilon)}$ so that $u\delta\,{\rm cap}(A)\overset{\eqref{eq:AsymptoticsOfCapOfBox},\eqref{eq:ABCNotation}}{\ge}u\delta s^{(1-\varepsilon)(d-2)}\ge s^{c(\varepsilon)}$.
Now set $\nu_{2}=\rho+\theta$ and note that $\nu_{1}$ and $\nu_{2}$
are independent, $\nu_{2}$ has intensity $2u\delta\kappa_{2}$ and
because of \prettyref{eq:FirstUseOfPhil-1}, \prettyref{eq:PPIneq}
and \prettyref{eq:DealingWithIt} we have 
\[
Q[\mathcal{I}(\nu)\cap\bar{A}\subset\mathcal{I}(\nu_{1}+\nu_{2})]\ge1-ce^{-cs^{c(\varepsilon)}}.
\]
Thus the proof of \prettyref{pro:CouplePPoEandUFreePPoE} is complete.
\end{proof}
The proof of \prettyref{pro:CouplePPoEandUFreePPoE} has thus been
reduced to \prettyref{lem:KappaOneIntensitySandwich} and \prettyref{lem:Mu234etc}.
We now prove \prettyref{lem:KappaOneIntensitySandwich}.
\begin{proof}[Proof of \prettyref{lem:KappaOneIntensitySandwich}]
Let $W\subset\Gamma(\mathbb{T}_{N})$ be measurable. Then $\kappa_{1}^{1}(W)\overset{\eqref{eq:DefOfKappa1},\eqref{eq:DefOfKappaL}}{=}P_{e}[Y_{\cdot\wedge T_{\bar{B}}}\in W,\mbox{ }U<\hat{R}_{2}]$
so the upper bound follows directly from \eqref{eq:DefOfKappa2-1}.
Furthermore $\kappa_{1}^{1}(W)\ge\kappa_{2}(W)\inf_{x\in\partial_{e}\bar{B}}P_{x}[H_{\bar{A}}>T_{\bar{D}}]\inf_{x\in\partial_{e}\bar{D}}P_{x}[H_{\bar{A}}>U]$
(recall \eqref{eq:DefOfD}) by the strong Markov property. But (if
$s\ge c(\varepsilon)$) $\inf_{x\in\partial_{e}\bar{B}}P_{x}[H_{\bar{A}}>T_{\bar{D}}]=\inf_{x\in\partial_{e}B}P_{x}^{\mathbb{Z}^{d}}[H_{A}>T_{D}]\overset{\eqref{eq:BHBZdBallHitting}}{\ge}1-cs^{-c(\varepsilon)}$,
and by \eqref{eq:CDBallHitting} we have $\inf_{x\in\partial_{e}\bar{D}}P_{x}[H_{\bar{A}}>U]\ge1-c(\varepsilon)s^{-c(\varepsilon)}$,
so the lower bound follows.
\end{proof}
It thus only remains to prove \prettyref{lem:Mu234etc} to complete
the proof of \prettyref{pro:CouplePPoEandUFreePPoE}. For this we
will use the following bound on the intensities $\kappa_{1}^{i}$
of the $\mu_{i}$, (this corresponds to (2.33) in \cite{SznitmanDecoupling}).
\begin{lem} \label{lem:CompoundPoissionIntMeasIneq}($d\ge3$,$N\ge3$) 

\medskip
If $s\ge c(\varepsilon)$ and $n\le s^{c(\varepsilon)}$ then for all $i\ge2$ 
\begin{equation}
\kappa_{1}^{i}\le\tilde{\kappa}_{1}^{i}\mbox{ where }\tilde{\kappa}_{1}^{i}(d(w_{1},...,w_{i}))=s^{-\frac{\varepsilon}{8}(i-1)}\mbox{cap}(A)\otimes_{k=1}^{i}P_{\bar{e}}[Y_{\cdot\wedge T_{\bar{B}}}\in dw_{k}],\label{eq:CompoundPoissonIntMeasIneq}
\end{equation}
and $\bar{e}=\frac{e}{n\,{\rm cap}(A)}$ denotes the normalisation
of the measure $e$ from \prettyref{eq:DefOfKappa1} (see \prettyref{eq:DefOfeKandCap}).
\end{lem}
In the proof of \prettyref{lem:Mu234etc} we will use \prettyref{lem:CompoundPoissionIntMeasIneq}
to dominate $\mu_{i},i\ge2,$ by Poisson point processes $\eta_{i}$
of intensity $u\tilde{\kappa}_{1}^{i}$. Since $\tilde{\kappa}_{1}^{i}$
is proportional to a product measure the ``points'' of $\eta_{i}$
will be vectors of \emph{independent} excursions with law $P_{\bar{e}}[Y_{\cdot\wedge T_{\bar{B}}}\in dw]$.
Thus we will have ``decoupled'' the excursions and we will be able
to use them (along with additional independent excursions) to construct
the Poisson point process $\theta$. We postpone the proof of \prettyref{lem:CompoundPoissionIntMeasIneq}
until after the proof of \prettyref{lem:Mu234etc}. In the proof of
\prettyref{lem:Mu234etc} we will use the following simple lemma about
Poisson random variables.
\begin{lem} \label{lem:SumOfPoissonsLemma} Let $N$ be a Poisson random variable
of intensity $\lambda>0$, and let $N_{i},i\ge2,$ be independent
Poisson random variables such that $N_{i}$ has intensity at most $\lambda r^{i-1}$. Then 
\begin{equation}
\mathbb{P}\Big[\dsl_{i\ge2}iN_{i}\le N\Big]\ge1-ce^{-c\lambda},\mbox{ if }0<r\le c.\label{eq:SumOfPoissonsLemma}
\end{equation}
\end{lem}
\begin{proof}
This follows from the standard Chebyshev bounds $\mathbb{P}[N<\frac{\lambda}{2}] \le \mathbb{E}[e^{-\frac{1}{2}N}]e^{\frac{\lambda}{4}}\le e^{-\frac{\lambda}{8}}$ and
\[
\mathbb{P}\Big[\dsl_{l\ge2}iN_{i}>\mbox{\f $\dis\frac{\lambda}{2}$}\Big]\le e^{-\frac{\lambda}{2}}\mathbb{E}[e^{\sum_{i\ge2}iN_{i}}] \le e^{-\frac{\lambda}{2}+\lambda\sum_{i\ge2}r^{i-1}(e^{i}-1)}\overset{re\le c<1}{\le}e^{-\frac{\lambda}{4}}\mbox{ for all }\lambda>0.
\]

\end{proof}
We now prove \prettyref{lem:Mu234etc}. The proof corresponds roughly
to (2.38)-(2.54) in \cite{SznitmanDecoupling}.
\begin{proof}[Proof of \prettyref{lem:Mu234etc}]
If we multiply \prettyref{eq:CompoundPoissonIntMeasIneq} by $u$
we get for each $i\ge2$ an inequality for the intensity measure of
$\mu_{i}$. Because of this inequality we can ``thicken'' each $\mu_{i}$,
by constructing $(\Omega',\mathcal{A}',Q')$ with the appropriate
random variables, to get (on $\Omega\times\Omega'$) $\sigma(\mu_{i},i\ge2)\times\mathcal{A}'-$measurable
independent Poisson point processes
\begin{equation}
\eta_{i},i\ge2,\mbox{ on }\Gamma(\mathbb{T}_{N})^{i}\mbox{ of intensities }u\tilde{\kappa}_{1}^{i}\mbox{ (respectively), such that }\mu_{i}\le\eta_{i},i\ge2,\label{eq:FirstDom}
\end{equation}
(analogously to below \prettyref{eq:FirstIntensitySandwich}).
We note that if we let $N_{i}=\eta_{i}(\Gamma(\mathbb{T}_{N})^{i}),i\ge2,$
then (see \prettyref{eq:CompoundPoissonIntMeasIneq}) 
\begin{equation}
N_{i},i\ge2,\mbox{ are independent and Poisson, where }N_{i}\mbox{ has intensity }us^{-\frac{\varepsilon}{8}(i-1)}\mbox{cap}(A).\label{eq:Nilaw}
\end{equation}

\smallskip\n
Now extend $(\Omega',\mathcal{A}',Q')$ to obtain $\sigma(\mu_{i},i\ge2)\times\mathcal{A}'-$measurable
vectors $v_{j}^{i},i\ge2,j\ge1$, such that $v_{j}^{i},i\ge2,j\ge1,N_{i},i\ge1,$ are independent,  
\begin{equation}
v_{j}^{i}\mbox{ has law }\otimes_{k=1}^{i}P_{\bar{e}}[Y_{\cdot\wedge T_{\bar{B}}}\in dw_{k}]\mbox{ (i.e. }u\tilde{\kappa}_{1}^{i}\mbox{ normalised)}\mbox{ and }\eta_{i}=\dsl_{j=1}^{N_{i}}\delta_{v_{j}^{i}},i\ge2,\label{eq:DefOfNlandEtal}
\end{equation}
(conditionally on $\eta_{i}$, we order the $N_{i}$ points in the
support of $\eta_{i}$ according to say the time until the first jump
of the first of the $i$ paths that make up a point of $\eta_{i}$,
and let $v_{1}^{i},\ldots,v_{N_{i}}^{i},$ be a permutation of these
points chosen uniformly at random; we then add i.i.d. vectors to form
$v_{j}^{i},j>N_{i}$).
Define $\bar{N}=\sum_{i\ge2}iN_{i}$ and construct on $(\Omega',\mathcal{A}',Q')$
a Poisson random variable $N$ of intensity $u\delta n\mbox{cap}(A)$,
and trajectories $\tilde{w}_{i},i\ge1,$ with law $P_{\bar{e}}[Y_{\cdot\wedge T_{\bar{B}}}\in dw]$,
such that $N,\tilde{w}_{i},i\ge1,v_{j}^{i},i\ge2,j\ge1,N_{i},i\ge2,$
are independent. Write $v_{j}^{i}=(w_{j,1}^{i},...,w_{j,i}^{i})$
and let  
\begin{equation}
\theta=\begin{cases}
\dsl_{i=2}^{\infty}\sum_{j=1}^{N_{i}}(\delta_{w_{j,1}^{i}}+...+\delta_{w_{j,i}^{i}})+\dsl_{i=1}^{N-\bar{N}}\delta_{\tilde{w}_{i}} & \mbox{if }N\ge\bar{N},
\\[2ex]
\dsl_{i=1}^{N}\delta_{\tilde{w}_{i}} & \mbox{if }N<\bar{N}.
\end{cases}\label{eq:DefOfTheta-1}
\end{equation}

\smallskip\n
The number of points $N$ of $\theta$ is a Poisson random variable,
and conditionally on $N$ the points of $\theta$ are i.i.d. with
law $P_{\bar{e}}[Y_{\cdot\wedge T_{\bar{B}}}\in dw]$ (see \prettyref{eq:DefOfNlandEtal}),
so that $\theta$ is (as claimed) a $\sigma(\mu_{i},i\ge2)\times\mathcal{A}'-$measurable
Poisson point process of intensity $u\delta n\mbox{cap}(A)P_{\bar{e}}[Y_{\cdot\wedge T_{\bar{B}}}\in dw]=u\delta\kappa_{2}$
(see \prettyref{eq:DefOfKappa2-1}).

It remains to show \prettyref{eq:Mu234etc}. We have $\cup_{i\ge2}\mathcal{I}(\mu_{i})\subset\cup_{i\ge2}\mathcal{I}(\eta_{i})$
(by \prettyref{eq:FirstDom}) and on the event $\{\bar{N}\le N\}$
we have $\cup_{i\ge2}\mathcal{I}(\eta_{i})\subset\mathcal{I}(\theta)$
by \prettyref{eq:DefOfNlandEtal} and \prettyref{eq:DefOfTheta-1}.
Thus by \prettyref{eq:SumOfPoissonsLemma} with $\lambda=u\delta n\mbox{cap}(A)$
and $r=\delta^{-1}s^{-\frac{\varepsilon}{8}}\le c$ (we require $\delta\ge cs^{-\varepsilon/8}$)
we get
\begin{equation}
Q\otimes Q'[\cup_{i\ge2}\mathcal{I}(\mu_{i})\subset\mathcal{I}(\theta)]\ge Q\otimes Q'[\bar{N}\le N]\ge1-ce^{-cu\delta\mbox{cap(\emph{A})}},\label{eq:FirstUseOfPoissonLemma}
\end{equation}
(see \prettyref{eq:Nilaw} and note that $us^{-\frac{\varepsilon}{8}(i-1)}\mbox{cap}(A)\le\lambda\delta^{-1}s^{-\frac{\varepsilon}{8}(i-1)}\le\lambda r^{i-1}$).
Thus \prettyref{eq:Mu234etc} holds.
\end{proof}
It remains to prove \prettyref{lem:CompoundPoissionIntMeasIneq}.
For the proof we will need the following upper bound on the probability
of hitting $\bar{A}$ in a given point before $U$, from outside of
$\bar{B}$.
\begin{lem} \label{lem:HittingLineEquilWeak}($s\ge c(\varepsilon),n\le s^{c(\varepsilon)})$

\medskip
For all $x\in\partial_{e}\bar{B}$ and $y\in\partial_{i}\bar{A}$ we have
\begin{equation}
h_{x}(y)\le s^{-\frac{\varepsilon}{4}}\bar{e}(y)\mbox{ where }h_{x}(y)=P_{x}[H_{\bar{A}}<U,Y_{H_{\bar{A}}}=y].\label{eq:HittingLikeEquilWeak}
\end{equation}

\end{lem}
Note that crucially $s^{-\frac{\varepsilon}{4}}\bar{e}(y)$ does not
depend on the starting point $x$. In the proof of \prettyref{lem:CompoundPoissionIntMeasIneq},
which we now start, this is what will allow us to bound $\kappa_{1}^{i}$
from above by an intensity that is proportional to a product measure
(namely $\tilde{\kappa}_{1}^{i}$). The proof of \prettyref{lem:HittingLineEquilWeak}
will follow after the proof of \prettyref{lem:CompoundPoissionIntMeasIneq}.
\begin{proof}[Proof of \prettyref{lem:CompoundPoissionIntMeasIneq}]
Fix a $i\ge2$. Let $W_{1},...,W_{i}\subset\Gamma(\mathbb{T}_{N})$
be measurable. Then
\begin{eqnarray*}
\kappa_{1}^{i}(W_{1}\times...\times W_{i}) &\!\!\!\! \overset{\eqref{eq:DefOfKappa1},\eqref{eq:DefOfKappaL}}{=} & \!\!\!\!P_{e}[\hat{R}_{i}<U<\hat{R}_{i+1},Y_{(\hat{R}_{j}+\cdot)\wedge\hat{D}_{j}}\in W_{j},1\le j\le i]\\
 &\!\!\!\! \le & \!\!\!\!P_{e}[\hat{R}_{i}<U,Y_{(\hat{R}_{j}+\cdot)\wedge\hat{D}_{j}}\in W_{j},1\le j\le i]\\
 &\!\!\!\! \underset{\rm{Markov}}{\overset{\eqref{eq:RetDepABTorus},\eqref{eq:HittingLikeEquilWeak}}{=}} & \!\!\!\!E_{e}[1_{\{\hat{R}_{i-1}<U,Y_{(\hat{R}_{j}+\cdot)\wedge\hat{D}_{j}}\in W_{j},1\le j\le i-1\}}P_{h_{Y_{\hat{D}_{i-1}}}}[Y_{\cdot\wedge T_{\bar{B}}}\in W_{j}]],\\
 &\!\!\!\! \overset{\eqref{eq:HittingLikeEquilWeak}}{\le} & \!\!\!\!s^{-\frac{\varepsilon}{4}}P_{e}[\hat{R}_{i-1}<U,Y_{(\hat{R}_{j}+\cdot)\wedge\hat{D}_{j}}\in W_{j},1\le j\le i-1]P_{\bar{e}}[Y_{\cdot\wedge T_{\bar{B}}}\in W_{i}].
\end{eqnarray*}
Now iterating a similar inequality we get
\begin{align*}
\kappa_{1}^{i}(W_{1}\times...\times W_{i})  \le & s^{-\frac{\varepsilon}{4}(i-1)}P_{e}[Y_{\cdot\wedge T_{\bar{B}}}\in W_{1}]\mbox{\f $\dis\prod_{j=2}^{i}$}P_{\bar{e}}[Y_{\cdot\wedge T_{\bar{B}}}\in W_{j}]
\\
  \le & s^{-\frac{\varepsilon}{8}(i-1)}{\rm cap}(A)\mbox{\f $\dis\prod_{j=1}^{i}$}P_{\bar{e}}[Y_{\cdot\wedge T_{\bar{B}}}\in W_{j}]\overset{\eqref{eq:CompoundPoissonIntMeasIneq}}{=}\tilde{\kappa}_{1}^{i}(W_{1}\times...\times W_{i}),
\end{align*}
where we have used that $s^{-\frac{\varepsilon}{4}(i-1)}P_{e}[Y_{\cdot\wedge T_{\bar{B}}}\in W_{1}]=s^{-\frac{\varepsilon}{4}(i-1)}n\,\mbox{cap}(A)P_{\bar{e}}[Y_{\cdot\wedge T_{\bar{B}}}\in W_{1}]$
and $s^{-\frac{\varepsilon}{4}(i-1)}n\le s^{-\frac{\varepsilon}{8}(i-1)}$
(we require $n\le s^{\frac{\varepsilon}{8}}$). 

\smallskip
Thus $\kappa_{1}^{i}(W)\le\tilde{\kappa}_{1}^{i}(W)$ for all $W\in\Gamma(\mathbb{T}_{N})^{i}$
that are products of measurable sets. This implies that $\kappa_{1}^{i}(W)\le\tilde{\kappa}_{1}^{i}(W)$
for all $W\in\Gamma(\mathbb{T}_{N})^{i}$ that are finite unions of
such sets ($W$ need not be a disjoint union, since ``overlapping''
unions of products of measurable sets may be turned into disjoint
unions of such sets by further ``subdividing'' the ``overlapping''
sets). By a monotone class argument, this implies that $\kappa_{1}^{i}(W)\le\tilde{\kappa}_{1}^{i}(W)$
for \emph{all} measurable $W\in\Gamma(\mathbb{T}_{N})^{i}$ (see Theorem
3.4, p. 39 in \cite{Billingsley}), so \prettyref{eq:CompoundPoissonIntMeasIneq}
follows.
\end{proof}
Finally, we prove \prettyref{lem:HittingLineEquilWeak}, using the
Harnack inequality and \eqref{eq:EquilDistHitDistFromFar}.
\begin{proof}[Proof of \prettyref{lem:HittingLineEquilWeak}]
If $j\ne k$, $x\in\partial_{e}B_{j}$ and $y\in\partial_{i}A_{k}$
then by the Markov property
\begin{equation}
h_{x}(y)=P_{x}[H_{\bar{A}}<U,Y_{H_{\bar{A}}}=y]\le\sup_{x\in\partial_{e}B_{k}}P_{x}[H_{\bar{A}}<U,Y_{H_{\bar{A}}}=y]=\sup_{x\in\partial_{e}B_{k}}h_{x}(y),\label{eq:WLOG}
\end{equation}
(provided $s\ge c(\varepsilon)$ so that $B_{j}$ and $B_{k}$ are
disjoint), so without loss of generality we may assume $x\in\partial_{e}B_{k}$.
We have (recall from \eqref{eq:DefOfD} that $C\subset D=B(0,s^{1-\frac{\varepsilon}{8}})$)
\begin{equation}
h_{x}(y)\le P_{x}[H_{\bar{A}}<T_{\bar{D}},Y_{H_{\bar{A}}}=y]+P_{x}[T_{\bar{D}}<H_{\bar{A}}<U,Y_{H_{\bar{A}}}=y].\label{eq:zerlegung}
\end{equation}
By the strong Markov property, \prettyref{eq:CDBallHitting} and \eqref{eq:WLOG},
we have for $s\ge c(\varepsilon)$ 
\begin{equation}
P_{x}[T_{\bar{D}}<H_{\bar{A}}<U,Y_{H_{\bar{A}}}=y]\le\sup_{z\in\partial_{e}\bar{D}}P_{z}[H_{\bar{C}}<U]\sup_{z\in\partial_{e}\bar{B}}h_{z}(y)\le cs^{-c(\varepsilon)}\sup_{z\in\partial_{e}B_{k}}h_{z}(y).\label{eq:strongmarkovused}
\end{equation}
The function $z\rightarrow h_{z}(y)$ is non-negative harmonic on
$C_{k}\backslash A_{k}$ (which can be identified with a subset of
$\mathbb{Z}^{d}$) so by the Harnack inequality (Proposition 1.7.2,
p. 42, \cite{LawlersLillaGrona}) and a standard covering argument
we have $h_{z}(y)\le ch_{x}(y)$ for all $z\in\partial_{e}B_{k}$.
Applying this inequality to the right-hand side of \eqref{eq:strongmarkovused},
plugging the result into \eqref{eq:zerlegung} and rearranging we
find that 
\begin{equation}
h_{x}(y)\le cP_{x}[H_{\bar{A}}<T_{\bar{D}},Y_{H_{\bar{A}}}=y]\le c\sup_{z\in\partial_{i}B}P_{z}^{\mathbb{Z}^{d}}[H_{A}<\infty,Y_{H_{A}}=y]\mbox{ for }s\ge c(\varepsilon).\label{eq:TheFirstStep}
\end{equation}
Thus using \eqref{eq:EquilDistHitDistFromFar} with $K=A,r=s^{1-\varepsilon}$
(recall \eqref{eq:DefOfA}) we have that if $s\ge c(\varepsilon)$
(so that $z\notin B(0,\Cr{eqhitdist}s^{1-\varepsilon})$ if $z\in\partial_{i}B$)
then 
\[
h_{x}(y)\le \Cr{eqhitupper}\mbox{\f $\dis\frac{e_{A}(y)}{{\rm cap}(A)}$} \sup_{z\in\partial_{i}B}P_{z}^{\mathbb{Z}^{d}}[H_{A}<\infty]\overset{\eqref{eq:BHBZdBallHitting}}{\le}cn\bar{e}(y)s^{-\frac{\varepsilon}{2}}\overset{n\le s^{\frac{\varepsilon}{8}},s\ge c(\varepsilon)}{\le}\bar{e}(y)s^{-\frac{\varepsilon}{4}}.
\]
\end{proof}

\vspace{-3ex}
Now all components used in the proof of \prettyref{pro:CouplePPoEandUFreePPoE}
have been proved.

\section{\label{sec:ToRI}Coupling with random interlacements}

In this section we prove \prettyref{pro:CouplePPoEandRI}. We use
essentially the same techniques that were used to prove \prettyref{pro:CouplePPoEandUFreePPoE}
in the previous section, but speaking very roughly we use them ``in
reverse'' to reconstruct from the excursions in the Poisson point
process $\eta$ of intensity $u\kappa_{3}$, which all end upon leaving
$B$, excursions with law $P_{e_{A}}^{\mathbb{Z}^{d}}$ (or rather,
the successive visits to $A$ after departures from $B$ of such excursions).
$P_{e_{A}}^{\mathbb{Z}^{d}}$ gives positive measure to excursions
that return to $A$ even after leaving $B$, and to construct such
excursions we will in \prettyref{lem:Mu234etc-1} take a ``small
number'' of excursions from $\eta$ and ``glue them together'',
essentially reversing the argument from \prettyref{lem:Mu234etc}.

Let $\tilde{R}_{1}\le\tilde{D}_{1}\le\tilde{R}_{2}\le\tilde{D}_{2}\le...$
on $\Gamma(\mathbb{Z}^{d})$ denote the successive returns of $Y_{\cdot}$
to $A$ and successive departures from $B$,
\[
\tilde{R}_{1}=H_{A},\mbox{ }\tilde{R}_{k}=H_{A}\circ\theta_{\tilde{D}_{k-1}}+\tilde{D}_{k-1},k\ge2,\mbox{ }\tilde{D}_{k}=T_{B}\circ\theta_{\tilde{R}_{k}}+\tilde{R}_{k},k\ge1.
\]
These should not be confused with the $\hat{R}_{k},\hat{D}_{k},$
which were defined on $\Gamma(\mathbb{T}_{N})$ and used in \prettyref{sec:OutOfTorus}
(see \prettyref{eq:RetDepABTorus}), or the $R_{k}$ from \prettyref{sec:Poissonization}
(see \prettyref{eq:DefOfRkUk}). Furthermore similarly to \prettyref{eq:PhiExcursionExtractor}
define maps $\phi_{i}^{\mathbb{Z}^{d}},i\ge2,$ from $\Gamma(\mathbb{Z}^{d})$
to $\Gamma(\mathbb{Z}^{d})^{i}$ extracting the excursions between
$A$ and $B$, 
\begin{equation}
(\phi_{i}^{\mathbb{Z}^{d}}(w))_{j}=w((\tilde{R}_{j}+\cdot)\wedge\tilde{D}_{j})\mbox{ for }j=1,...,i,w\in\{\tilde{R}_{i}<\infty=\tilde{R}_{i+1}\}\subset\Gamma(\mathbb{Z}^{d}),i\ge1.\label{eq:PhiZd}
\end{equation}

\smallskip\n
To construct the random set $\mathcal{I}_{1}$ from the statement
of \prettyref{pro:CouplePPoEandRI} we will construct Poisson point
processes of intensities $u(1-\delta)\kappa_{4}^{i}$, where (cf.
\prettyref{eq:DefOfKappaL}) 
\begin{equation}
\kappa_{4}^{i}=\phi_{i}\circ(1_{\{\tilde{R}_{i}<\infty = \tilde{R}_{i+1}\}}P_{e_{A}}^{\mathbb{Z}^{d}}),i\ge1.\label{eq:DefOfKappaL-1}
\end{equation}
This will be enough to construct $\mathcal{I}_{1}$ because if $\mu_{i},i\ge1$,
are i.i.d. Poisson point processes of intensity $u(1-\delta)\kappa_{4}^{i}$
then by \prettyref{eq:LawOfRIInFiniteSet}, \prettyref{eq:DefOfMuK},
\prettyref{eq:PhiZd} and \prettyref{eq:DefOfKappaL-1} (recalling
the notation from \prettyref{eq:DefOfPPPTrace})
\begin{equation}
\mathcal{I}^{u(1-\delta)}\cap A\overset{\mbox{law}}{=}\cup_{i\ge1}\mathcal{I}(\mu_{i})\cap A.\label{eq:LawOfUnionIsRI}
\end{equation}
To construct a Poisson point process $\mu_{1}$ of intensity $u(1-\delta)\kappa_{4}^{1}$
we will, in the proof of \prettyref{pro:CouplePPoEandRI} , ``extract''
a Poisson point process of intensity $u(1-\delta)\kappa_{3}$ from
$\eta$, and ``thin'' it to get $\mu_{1}$. This will be possible
because of the following inequality.
\begin{lem} \label{lem:IntensitySandwhich1Zd}($N\ge3,d\ge3$) 

\medskip
If $s\ge c(\varepsilon)$
then \textup{$\kappa_{4}^{1}\le\kappa_{3}\le(1+cs^{-c(\varepsilon)})\kappa_{4}^{1}$.}\end{lem}
\begin{proof}
This is a consequence of \prettyref{eq:BHBZdBallHitting}, \prettyref{eq:DefOfKappa3}
and \prettyref{eq:DefOfKappaL-1}. The proof is a very similar to
that of \prettyref{lem:KappaOneIntensitySandwich}, so we omit it.
\end{proof}
After constructing $\mu_{1}$ in the proof of \prettyref{pro:CouplePPoEandRI}
we will take ``what is left of $\eta$ after thinning using \prettyref{lem:IntensitySandwhich1Zd}'',
namely a Poisson point process $\theta$ of intensity $u\delta\kappa_{3}$,
and ``extract from it'' the Poisson point processes $\mu_{2},\mu_{3},...$
of respective intensities $u(1-\delta)\kappa_{4}^{i}$. This will
be done using the following lemma.
\begin{lem} \label{lem:Mu234etc-1}$(d\ge3,N\ge3)$ 

\medskip
Let $u\ge0$, $\delta\ge cs^{-c(\varepsilon)}$
and $s\ge c(\varepsilon)$, and let $(\Omega,\mathcal{A},Q)$ be a
probability space with a Poisson point processes $\theta$ of intensity
$u\delta\kappa_{3}$. Then we can construct a space $(\Omega',\mathcal{A}',Q')$
and, on the product space, independent $\sigma(\theta)\times\mathcal{A}'-$measurable
Poisson point processes $\mu_{i},i\ge2,$ and $\rho_{2}$ such that
$\mu_{i}$ has intensity $u(1-\delta)\kappa_{4}^{i}$, $\rho_{2}$
has intensity $u\frac{3\delta}{2}\kappa_{3}$ and
\begin{equation}
Q\otimes Q'\Big[\mbox{\f $\dis\bigcup\limits_{i\ge2}$}\; \mathcal{I}(\mu_{i})\subset\mathcal{I}(\theta)\subset\cup_{i\ge2}\mathcal{I}(\mu_{i})\, \mbox{\f $\dis\bigcup$} \;\mathcal{I}(\rho_{2})\Big]\ge1-ce^{-cu\delta\mbox{cap(A)}}.\label{eq:Mu234etc-1}
\end{equation}

\end{lem}
The ``residual'' Poisson point process $\rho_{2}$, as well as a
``residual'' Poisson point process left after ``thinning'' to
obtain $\mu_{1}$, will be used to construct $\mathcal{I}_{2}$. We
postpone the proof of \prettyref{lem:Mu234etc-1} until later, and
instead use it to prove \prettyref{pro:CouplePPoEandRI}.
\begin{proof}[Proof of \prettyref{pro:CouplePPoEandRI}]
We start by constructing $(\Omega',\mathcal{A}',Q')$ appropriately,
to obtain, by ``thinning'' $\eta$, a $\sigma(\eta)\times\mathcal{A}'-$measurable
Poisson point process $\theta$ on the product space of intensity
$u\delta\kappa_{3}$, such that $\eta-\theta$ and $\theta$ are independent,
and $\eta-\theta$ is a Poisson point process of intensity $u(1-\delta)\kappa_{3}$
(similarly to below \prettyref{eq:FirstIntensitySandwich}, note that of course $u\delta\kappa_3 \le u\kappa_3$).

Since we require $\delta\ge cs^{-c(\varepsilon)}$ and $s\ge c(\varepsilon)$
we have $u(1-\delta)(1+s^{-c(\varepsilon)})\le u(1-\frac{\delta}{2})$,
and thus $u(1-\delta)\kappa_{4}^{1}\le u(1-\delta)\kappa_{3}\le u(1-\frac{\delta}{2})\kappa_{4}^{1}$
by \prettyref{lem:IntensitySandwhich1Zd}. But $u(1-\frac{\delta}{2})\kappa_{4}^{1}\le u(1-\delta)\kappa_{4}^{1}+\frac{\delta}{2}\kappa_{3}$
since $\kappa_{4}^{1}\le\kappa_{3}$, so that 
\[
u(1-\delta)\kappa_{4}^{1}\le u(1-\delta)\kappa_{3}\le u(1-\delta)\kappa_{4}^{1}+u\mbox{\f $\dis\frac{\delta}{2}$}\,\kappa_{3}.
\]
Therefore we can, similarly to under \prettyref{eq:Kappa1DeltaIntensitySandwich-1},
construct (extending $(\Omega',\mathcal{A}',Q')$ appropriately) $\sigma(\eta)\times\mathcal{A}'-$measurable
Poisson point processes $\mu_{1}$ and $\rho_{1}$ such that $\mu_{1},\rho_{1},\theta$
are independent, $\mu_{1}$ has intensity $u(1-\delta)\kappa_{4}^{1}$,
$\rho_{1}$ has intensity $u\frac{\delta}{2}\kappa_{3}$, 
\begin{equation}
\mu_{1}\le\eta-\theta\le\mu_{1}+\rho_{1},\mbox{ and thus }\mathcal{I}(\mu_{1})\subset\mathcal{I}(\eta-\theta)\subset\mathcal{I}(\mu_{1})\mbox{\f $\dis\bigcup$} \mathcal{I}(\rho_{1}).\label{eq:mu1etathetamuarho1}
\end{equation}
We then apply \prettyref{lem:Mu234etc-1} (once again extending $(\Omega',\mathcal{A}',Q')$)
to $\theta$ to get $\sigma(\eta)\times\mathcal{A}'-$measurable Poisson
point processes $\mu_{i},i\ge2,\rho_{2}$ such that $\rho_{1},\rho_{2},\mu_{i},i\ge1,$
are independent, $\mu_{i}$ has intensity $u(1-\delta)\kappa_{4}^{i}$,
$\rho_{2}$ has intensity $u\frac{3\delta}{2}\kappa_{3}$, 
\begin{equation}
Q\otimes Q'\Big[\mbox{\f $\dis\bigcup_{i\ge2}$}\; \mathcal{I}(\mu_{i})\subset\mathcal{I}(\theta)\subset\cup_{i\ge2}\mathcal{I}(\mu_{i})\; \mbox{\f $\dis\bigcup$}\; \mathcal{I}(\rho_{2})\Big]\ge1-ce^{-cu\delta\mbox{cap(\emph{A})}}.\label{eq:ConstructingTheMus}
\end{equation}
Note that $\rho_{1}+\rho_{2}$ is a Poisson point process of intensity
$2u\delta\kappa_{3}$, and that the ``points'' of this process are
pieces of random walk with law $\frac{1}{\mbox{cap}(A)}P_{e_{A}}[Y_{\cdot\wedge T_{B}}\in dw]$
(recall \prettyref{eq:DefOfeKandCap}). By constructing countably
many independent random walks on $(\Omega',\mathcal{A}',Q')$, and
``attaching'' a different one to each piece of random walk in $\rho_{1}+\rho_{2}$
we obtain a Poisson point process $\rho_{3}$ of intensity $2u\delta P_{e_{A}}$
(the ``points'' of $\rho_{3}$ have law $\frac{1}{{\rm cap}(A)}P_{e_{A}}$,
by the strong Markov property) such that 
\begin{equation}
\mathcal{I}(\rho_{1}+\rho_{2})\subset\mathcal{I}(\rho_{3})\mbox{ almost surely.}\label{eq:rho1rho2inrho3}
\end{equation}
Now let 
\[
\mathcal{I}_{1}=\cup_{i\ge1}\mathcal{I}(\mu_{i})\cap A\mbox{ and }\mathcal{I}_{2}=\mathcal{I}(\rho_{3})\cap A
\]

\smallskip\n
and note that $\mathcal{I}_{1}$ has the law of $\mathcal{I}^{u(1-\delta)}\cap A$
under $Q_{0}$, by \prettyref{eq:LawOfUnionIsRI}, $\mathcal{I}_{2}$
has the law of $\mathcal{I}^{2u\delta}$ under $Q_{0}$, by \prettyref{eq:LawOfRIInFiniteSet},
$\mathcal{I}_{1}$ and $\mathcal{I}_{2}$ are $\sigma(\eta)\times\mathcal{A}'-$measurable
and independent, and since $\mathcal{I}(\eta)\cap A=\mathcal{I}(\eta-\theta)\cap A\bigcup\mathcal{I}(\theta)\cap A$
we get from \prettyref{eq:mu1etathetamuarho1}, \prettyref{eq:ConstructingTheMus}
and \prettyref{eq:rho1rho2inrho3} that
\begin{equation}
Q\otimes Q'[\mathcal{I}_{1}\cap A\subset\mathcal{I}(\eta)\cap A\subset\mathcal{I}_{1}\cup\mathcal{I}_{2}]\ge1-ce^{-cu\delta\rm{cap(\emph{A})}}\ge1-ce^{-cs^{c(\varepsilon)}},\label{eq:FirstThing}
\end{equation}
were the second inequality holds because we require $u,\delta\ge s^{-c(1- \varepsilon)}=s^{-c(\varepsilon)}$,
similarly to in \prettyref{eq:DealingWithIt}. This completes the
proof of \prettyref{pro:CouplePPoEandRI}.
\end{proof}
It remains to prove \prettyref{lem:Mu234etc-1}. In the proof we will
extract from the Poisson point process $\theta$ of intensity $u\delta\kappa_{3}$
Poisson point processes of intensity $u(1-\delta)\tilde{\kappa}_{4}^{i}$
(see \prettyref{eq:Kappa4lUpperBound} below). This will be possible
because the ``points'' of a Poisson point process of intensity a
multiple of $\tilde{\kappa}_{4}^{i}$ is an \emph{i.i.d. vector} of
excursions with law $P_{\bar{e}_{A}}^{\mathbb{Z}^{d}}[Y_{\cdot\wedge T_{B}}\in dw]$
(see \prettyref{eq:Kappa4lUpperBound}), which is also the law of
the single excursions that make up the points of $\theta$ (and because
the number of ``points'' we need to construct the Poisson point
processes of intensity $u(1-\delta)\tilde{\kappa}_{4}^{i}$ will with
high probability not exceed the number of points in $\theta$). Once
we have these Poisson point processes we will use the following lemma
of intensity measures to ``thin'' them to obtain Poisson point processes
of intensity $u(1-\delta)\kappa_{4}^{i}$, and these will be the $\mu_{2},\mu_{3},..$
from the statement of \prettyref{lem:Mu234etc-1}.
\begin{lem} ($N\ge3,d\ge3$) 

\medskip
If $s\ge c(\varepsilon)$ and $i\ge2$
\begin{equation}
\kappa_{4}^{i}\le\tilde{\kappa}_{4}^{i}\mbox{ where }\tilde{\kappa}_{4}^{i}(d(w_{1},...,w_{i}))=s^{-\frac{\varepsilon}{4}(i-1)}\mbox{cap}(A)\otimes_{j=1}^{i}P_{\bar{e}_{A}}^{\mathbb{Z}^{d}}[Y_{\cdot\wedge T_{B}}\in dw_{j}],\label{eq:Kappa4lUpperBound}
\end{equation}
where $\bar{e}_{A}(\cdot)=\frac{e_{A}(\cdot)}{\mbox{cap}(A)}$ denotes
the normalisation of the measure $e_{A}(\cdot)$, see \prettyref{eq:DefOfeKandCap},
(and should not be confused with the measure $\bar{e}$ from \eqref{eq:DefOfKappa1}
and the proof of \prettyref{lem:Mu234etc}).\end{lem}
\begin{proof}
Similarly to how it (in the proof of \prettyref{lem:CompoundPoissionIntMeasIneq})
followed from \prettyref{eq:HittingLikeEquilWeak} that $\kappa_{3}^{i}(W)\le\tilde{\kappa}_{3}^{i}(W)$
for all $W\in\Gamma(\mathbb{T}_{N})^{i}$ that are products of measurable
sets, it follows from\linebreak{}
 $\sup_{x\in\partial_{e}B}P_{x}^{\mathbb{Z}^{d}}[Y_{H_{A}}=y,Y_{H_{A}}<\infty]\le cs^{-\frac{\varepsilon}{2}}\bar{e}_{A}(y)\le s^{-\frac{\varepsilon}{4}}\bar{e}_{A}(y)$
(see \eqref{eq:EquilDistHitDistFromFar} and \eqref{eq:BHBZdBallHitting}
and recall $s\ge c(\varepsilon)$) that $\kappa_{4}^{i}(W)\le\tilde{\kappa}_{4}^{i}(W)$
for all such $W$. But this implies \eqref{eq:Kappa4lUpperBound}
(by a monotone class argument, like at the end of the proof of \prettyref{lem:CompoundPoissionIntMeasIneq}).
We omit the details. 
\end{proof}
We now prove \prettyref{lem:Mu234etc-1}.
\begin{proof}[Proof of \prettyref{lem:Mu234etc-1}]
Note that 
\begin{equation}
N\stackrel{\rm def}{=}\theta(\Gamma(\mathbb{Z}^{d}))\mbox{ is Poisson with intensity }u\delta\mbox{cap}(A).\label{eq:NisPoiss}
\end{equation}
Similarly to in the proof of \prettyref{lem:Mu234etc} (see \prettyref{eq:DefOfNlandEtal})
we can construct $(\Omega',\mathcal{A}',Q')$ appropriately to obtain
i.i.d. $\sigma(\theta)\times\mathcal{A}'-$measurable trajectories
$w_{i},i\ge1,$ independent of $N$, 
\begin{equation}
\mbox{such that }w_{i}\mbox{ has law }P_{\bar{e}}[Y_{\cdot\wedge T_{B}}\in dw],\mbox{ and }\theta=\dsl_{i=1}^{N}\delta_{w_{i}}.\label{eq:ExcursionsOfTheta-1}
\end{equation}
 Extend $(\Omega',\mathcal{A}',Q')$ with independent Poisson random
variables $N_{i},i\ge2,$ of respective intensities $u(1-\delta)s{}^{-\frac{\varepsilon}{4}(i-1)}\mbox{cap}(A)$
(also independent of $w_{i},i\ge1,$ and $\theta$) and let
\begin{align}
\eta_{i}  = & \dsl_{j=1}^{N_{i}}\delta_{(w_{K_{j}^{i},},w_{K_{j}^{i}+1},..,w_{K_{j}^{i}+(i-1)})},\mbox{ where }K_{j}^{i}=\dsl_{k=1}^{i-1}kN_{k}+(j-1)i+1\mbox{ and}\label{eq:ConstrucOfEtas}\\
\rho  = & \dsl_{i=\bar{K}+1}^{\bar{K}+N}\delta_{w_{i}}\mbox{ where }\bar{K}=\dsl_{k=1}^{\infty}kN_{k}.\label{eq:ConstrucOfRho}
\end{align}
The $\eta_{i}$ ``use'' only $w_{1},...,w_{\bar{K}}$, so on the
event the event $\{\bar{K}\le N\}$ we have $\cup_{i\ge2}\mathcal{I}(\eta_{i})\subset\mathcal{I}(\theta)$
(see \prettyref{eq:ExcursionsOfTheta-1}). Recalling \prettyref{eq:NisPoiss}
and that the $N_{i},i\ge2,$ are independent Poisson random variables
of intensity less than $us^{-\frac{\varepsilon}{4}(i-1)}\mbox{cap}(A)\le u\delta s^{-\frac{\varepsilon}{8}(i-1)}\mbox{cap}(A)=\lambda r^{i-1}$
(we require $\delta\ge s^{-\frac{\varepsilon}{8}}$), where $\lambda=u\delta\mbox{cap}(A)$
and $r=s^{-\frac{\varepsilon}{8}}\le c$ (we require $s\ge c(\varepsilon)$),
we have by \prettyref{eq:SumOfPoissonsLemma}
\begin{equation}
Q\otimes Q'[\cup_{i\ge2}\mathcal{I}(\eta_{i})\subset\mathcal{I}(\theta)]\ge Q\otimes Q'[\bar{K}\le N]\ge1-ce^{-cu\delta\mbox{cap}(A)}.\label{eq:EtaContainedInTheta}
\end{equation}

\smallskip\n
Also, since the number of the $w_{i}$ ``used'' by $\rho$ is the
same as the number ``used by'' $\theta$ 
\begin{equation}
\mathcal{I}(\theta)\subset\cup_{i\ge2}\mathcal{I}(\eta_{i})\cup\mathcal{I}(\rho)\mbox{ almost surely.}\label{eq:ThetaContainedAlmostSurely}
\end{equation}

\medskip\n
Furthermore, because they ``use different $w_{i}$'' and $N,N_{i},i\ge2,w_{i},i\ge1,$
are independent $\rho,\eta_{i},i\ge2$ are independent $\sigma(\theta)\times\mathcal{A}'-$measurable
point processes, where $\rho$ has intensity $u\delta\mbox{cap}(A)P_{\bar{e}_{A}}[Y_{\cdot\wedge T_{B}}\in dw]=u\delta\kappa_{3}(dw)$
(see \prettyref{eq:DefOfKappa3}, \prettyref{eq:NisPoiss} and \prettyref{eq:ConstrucOfRho})
and $\eta_{i}$ has intensity $u(1-\delta)\tilde{\kappa}_{4}^{i}$
(see \eqref{eq:Kappa4lUpperBound}, \prettyref{eq:ConstrucOfEtas}
and recall that $N_{i}$ has intensity $(1-\delta)s^{-\frac{\varepsilon}{4}(i-1)}\mbox{cap}(A)$).
By the inequality in \eqref{eq:Kappa4lUpperBound} we have $u(1-\delta)\kappa_{4}^{i}\le u(1-\delta)\tilde{\kappa}_{4}^{i}$.
Together with the (very crude) bound $u(1-\delta)\tilde{\kappa}_{4}^{i}\le u(1-\delta)\kappa_{4}^{i}+u\tilde{\kappa}_{4}^{i}$
this allows us to (similarly to under \prettyref{eq:Kappa1DeltaIntensitySandwich-1})
construct independent Poisson point processes $\mu_{i},\mu_{i}^{'},i\ge1,$
such that $\mu_{i}$ has intensity $u(1-\delta)\kappa_{4}^{i}$, $\mu_{i}^{'}$
has intensity $u\tilde{\kappa}_{4}^{i}$ and 
\begin{equation}
\mu_{i}\le\eta_{i}\le\mu_{i}+\mu_{i}^{'}\mbox{ for }i\ge2.\label{eq:EtaIdominatedMui}
\end{equation}
By \prettyref{eq:EtaContainedInTheta} and \prettyref{eq:EtaIdominatedMui}
we have 
\begin{equation}
Q\otimes Q'[\cup_{i\ge2}\mathcal{I}(\mu_{i})\subset\mathcal{I}(\theta)]\ge1-ce^{-cu\delta\,\rm{cap(\emph{A})}}.\label{eq:TheLowerInv}
\end{equation}
Now $\mu_{i}^{'}(\Gamma(\mathbb{Z}^{d}))$ are independent Poisson
random variables of respective intensities \linebreak{}
$us^{-\frac{\varepsilon}{4}(i-1)}\mbox{cap}(A)\le\lambda r^{i-1}$,
where $\lambda=u\frac{\delta}{2}\mbox{cap}(A)$ and $r=s^{-\frac{\varepsilon}{8}}$
(we require $\delta\ge2s^{-\frac{\varepsilon}{8}}$). Thus using \prettyref{eq:SumOfPoissonsLemma}
we see that $\sum_{i\ge2}i\times\mu_{i}^{'}(\Gamma(\mathbb{Z}^{d}))\le N'$
with probability at least $1-ce^{-cu\delta\mbox{cap}(A)}$, where
$N'$ is a Poisson random variable of intensity $u\frac{\delta}{2}\mbox{cap}(A)$.
Therefore we can, similarly how the $\theta$ from \prettyref{lem:Mu234etc}
(not to be confused with the $\theta$ here) was constructed (see
\prettyref{eq:DefOfTheta-1}) construct a $\sigma(\theta)\times\mathcal{A}'-$measurable
Poisson point process $\rho'$ of intensity $u\frac{\delta}{2}\kappa_{3}$
such that $\mu_{i},i\ge2,\rho,\rho',$ are independent and
\begin{equation}
Q\otimes Q'[\cup_{i\ge2}\mathcal{I}(\mu_{i}^{'})\subset\mathcal{I}(\rho')]\ge1-ce^{-cu\delta\mbox{cap(\emph{A})}}.\label{eq:etaiminusmuiinrho}
\end{equation}
Now construct a $\sigma(\theta)\times\mathcal{A}'-$measurable Poisson
point process 
\begin{equation}
\rho_{2}\mbox{ of intensity }u\mbox{\f $\dis\frac{3\delta}{2}$}\, \kappa_{3}\mbox{ by setting }\rho_{2}=\rho+\rho'.\label{eq:defofrho2}
\end{equation}
Note that $\mu_{i},i\ge2,\rho_{2},$ are independent and and by \eqref{eq:ThetaContainedAlmostSurely},
\prettyref{eq:EtaIdominatedMui}, \eqref{eq:etaiminusmuiinrho} and
\prettyref{eq:defofrho2} 
\[
Q\otimes Q'[\mathcal{I}(\theta)\subset\mbox{\f $\dis\bigcup\limits_{i\ge2}$}\;\mathcal{I}(\mu_{i})\; \mbox{\f $\dis\bigcup$}\; \mathcal{I}(\rho_{2})]\ge1-ce^{-cu\delta\mbox{cap(A)}}.
\]
Together with \prettyref{eq:TheLowerInv} this implies \prettyref{eq:Mu234etc-1},
so the proof of \prettyref{lem:Mu234etc-1} is complete.
\end{proof}
Now all components used in the proof of \prettyref{pro:CouplePPoEandRI}
have been proved, and thus the last piece of the proof of \prettyref{thm:Coupling}
is complete. Since \prettyref{thm:G_Gumbel} was reduced to \prettyref{thm:Coupling}
in \prettyref{sec:ProofOfGumbel}, also the last part of the proof
of \prettyref{thm:G_Gumbel} is done. We finish with the following
remark.
\begin{rem}
\label{rem:endremark}(1) As mentioned in \prettyref{rem:EndOfSec3Remarks} (2),
a generalisation of the cover time result \prettyref{thm:G_Gumbel}
to other graphs may be possible. Roberto Oliveira and Alan Paula are
working on such a generalisation.

\medskip\n
(2) In \prettyref{cor:PPoVCL} we proved that the point process $\mathcal{N}_{N}^{z}$
of points of $\mathbb{T}_{N}$ not hit at time $g(0)N^{d}\{\log N^{d}+z\}$,
for $z\in\mathbb{R}$, converges in law to a Poisson point process.
It is an open question whether this can be generalised to show that
the point process $\mathcal{N}_{N}=\sum_{x\in\mathbb{T}_{N}}\delta_{(x/N,\frac{H_{x}}{g(0)N^{d}}-\log N^{d})}$
on the space $(\mathbb{R}/\mathbb{Z})^{d}\times\mathbb{R}$, from
which $\mathcal{N}_{N}^{z}$ can be recovered but which records also
when the vertices are hit, converges in law to a Poisson point process
of intensity $e^{-z}dxdz$, where $dx$ denotes Lebesgue measure on
$(\mathbb{R}/\mathbb{Z})^{d}$ and $dz$ denotes Lebesgue measure
on $\mathbb{R}$.\inputencoding{latin1}{}\inputencoding{latin9}\qed

\medskip\n
(3) It would be interesting to explicitly determine the dependence on $\varepsilon$ of the constant
$\Cr{couplingerr}=\Cr{couplingerr}(\varepsilon)$ from the coupling \prettyref{thm:Coupling}.
 This constant arises as the minimum of a number of constants
from sections \ref{sec:Quasistationary}-\ref{sec:ToRI}, that appear in requirements on the level $u$, the number of
boxes $n$, and the parameter $\delta$.
Several of the requirements cause us to choose $c_5(\varepsilon) \le c\varepsilon$; 
see for instance 
\eqref{eq:reqonn},
\eqref{eq:CDBallHitting}, and above
\eqref{eq:FirstUseOfPoissonLemma} and 
\eqref{eq:EtaContainedInTheta}.
Several others cause us to choose $c_5(\varepsilon) \le c(1-\varepsilon)$; see
for instance above
\eqref{eq:RWPPOEResultPrime} and
below \eqref{eq:DealingWithIt}
and \eqref{eq:FirstThing}.
It is plausible that $c_5$ can in fact be chosen to be $c\min(\varepsilon,1-\varepsilon)$,
for a small constant $c$ independent of $\varepsilon$.\qed

\end{rem}

\appendix

\section{Appendix}

Here we prove \prettyref{lem:ETTEntraceTimesTorus}. Recall \eqref{eq:DefOfD}.
\begin{proof}[Proof of \prettyref{lem:ETTEntraceTimesTorus}.]
If $V\subset\bar{C}$ then by (3.22) and (3.23) of \cite{TeixeiraWindischOnTheFrag}
with $A_{1}=V$ and $A_{2}=\bar{D}$ we have (recalling the definition
from \eqref{eq:DefOfRelativeEquilAndCap} and that the $D_{1},..,D_{n}$
are disjoint when $s\ge c(\varepsilon)$)
\begin{equation}
\Big(1-c\sup_{x\notin\bar{D}}\big|\mbox{\f $\dis\frac{E_{x}[H_{V}]}{E[H_{V}]}$}-1\big|\Big)\dsl_{i=1}^{n}\mbox{cap}_{D}(V^{i})\le\mbox{\f $\dis\frac{N^{d}}{E[H_{V}]}$} \le\mbox{\f $\dis\frac{1}{(\pi(\mathbb{T}_{N}\backslash\bar{D}))^{2}}$} \dsl_{i=1}^{n}\mbox{cap}_{D}(V^{i}).\label{eq:EETOtherSandwhich}
\end{equation}
We have $\left(\pi(\mathbb{T}_{N}\backslash\bar{D})\right)^{-2}\overset{n\le s^{c(\varepsilon)},N\ge s\ge c(\varepsilon)}{\le}1+cN^{-c(\varepsilon)}$,
and by \eqref{eq:RelativeEquilUpperBound} (with $K=V^{i}$, $r=s^{1-\frac{\varepsilon}{4}}$
and $\lambda$ such that $(1-\frac{\varepsilon}{4})(1+\lambda)=1-\frac{\varepsilon}{8}$
so that $U=D$), \prettyref{eq:DefOfeKandCap} and \prettyref{eq:DefOfRelativeEquilAndCap}
we have $\mbox{cap}_{D}(V^{i})\le(1+c(\varepsilon)s^{-c(\varepsilon)})\mbox{cap}(V^{i})$,
so \eqref{eq:ETTEntranceTimeLowerBound} follows. Thanks to \eqref{eq:RelativeEquilLowerBound}
we will obtain \eqref{eq:ETTEntraceTimeUpperBound} from \eqref{eq:EETOtherSandwhich}
once we have shown \eqref{eq:ETTStartingFromPointandUniformComparison}.
It thus only remains to show \eqref{eq:ETTStartingFromPointandUniformComparison}.

So assume $V\subset\bar{B}$. By \eqref{eq:TorusMixing} and $\sup_{x,y\in\mathbb{T}_{N}}E_{x}[H_{y}]\le cN^{d}$
(see e.g. (10.18), p. 133 in \cite{LevPerWilMarkovChainsandMixingTimes})
we have for all $x\in\mathbb{T}_{N}$ (recall from \prettyref{eq:DefOfUandtStar}
that $t^{\star}=N^{2+c\varepsilon}$)
\begin{equation}
\left|E_{x}[E_{Y_{t^{\star}}}[H_{V}]]-E[H_{V}]\right|\le ce^{-cN^{c\varepsilon}}\sup_{x\in\mathbb{T}_{N}}E_{x}[H_{V}]\overset{N\ge c(\varepsilon)}{\le}ce^{-cN^{c\varepsilon}}.\label{eq:MixingUsed}
\end{equation}
Therefore for all $x\in\mathbb{T}_{N}$ 
\begin{equation}
E_{x}[H_{V}]\le E_{x}[E_{Y_{t^{\star}}}[H_{V}]]+t^{\star}\overset{\eqref{eq:MixingUsed}}{\le}E[H_{V}]+ct^{\star}\le(1+cN^{-c(\varepsilon)})E[H_{V}],\label{eq:UpperBoundInLemma}
\end{equation}
where we have used that (recall that the capacity is monotone, see
e.g. Proposition 2.2.1, p.52 in \cite{LawlersLillaGrona}) 
\begin{equation}
ct^{\star}\overset{n\le s^{\frac{(d-2)\varepsilon}{4}}}{\le}\frac{N^{d}}{cns^{(1-\frac{\varepsilon}{2})(d-2)}}N^{-c(\varepsilon)}\overset{\eqref{eq:AsymptoticsOfCapOfBox},\eqref{eq:ETTEntranceTimeLowerBound},V\subset\bar{B}}{\le}E[H_{V}]c(\varepsilon)N^{-c(\varepsilon)}.\label{eq:Noname}
\end{equation}
Thus the upper bound of \eqref{eq:ETTStartingFromPointandUniformComparison}
holds. The lower bound follows since for all $x\notin\bar{C}$
\begin{eqnarray*}
E_{x}[H_{V}] & \ge & E_{x}[E_{Y_{t^{\star}}}[H_{V}]]-E_{x}[1_{\{H_{V}<t^{\star}\}}E_{Y_{t^{\star}}}[H_{V}]]\\
 & \overset{\eqref{eq:MixingUsed},V\subset\bar{B}}{\ge} & E[H_{V}]-c-\sup_{y\notin\bar{C}}P_{y}[H_{\bar{B}}<t^{\star}]\sup_{y\in\mathbb{T}_{N}}E_{y}[H_{V}]\\
 & \overset{\eqref{eq:UpperBoundInLemma},\eqref{eq:Noname},n\le s^{c(\varepsilon)}}{\ge} & E[H_{V}](1-c(\varepsilon)N^{-c(\varepsilon)}-c(\varepsilon)s^{-c(\varepsilon)}),
\end{eqnarray*}
where we have used that $\sup_{y\notin\bar{C}}P_{y}[H_{\bar{B}}<t^{\star}]\le c(\varepsilon)s^{-c(\varepsilon)}$
if $n\le s^{c(\varepsilon)}$, similarly to \prettyref{eq:CDBallHitting}.
\end{proof}

\begin{acknowledgement*}
The author would like to thank his Ph.D. adviser Alain-Sol Sznitman for supervising his research, 
and Bal\'azs R\'ath, Artem Sapozhnikov and Augusto Teixeira for useful discussions.
\end{acknowledgement*}

\end{document}